\documentclass[12pt]{article}

\usepackage[a4paper,margin=0.5in]{geometry}
\usepackage{amsmath,amssymb,amsthm, comment, mathtools}
\usepackage{slashed}
\usepackage[english]{babel}
\usepackage[utf8]{inputenc}

\makeatletter
\newcounter{savesection}
\newcounter{apdxsection}
\renewcommand\appendix{\par
  \setcounter{savesection}{\value{section}}%
  \setcounter{section}{\value{apdxsection}}%
  \setcounter{subsection}{0}%
  \gdef\thesection{\@Alph\c@section}}
\newcommand\unappendix{\par
  \setcounter{apdxsection}{\value{section}}%
  \setcounter{section}{\value{savesection}}%
  \setcounter{subsection}{0}%
  \gdef\thesection{\@arabic\c@section}}
\makeatother

\usepackage{color}
\usepackage{amsmath,amsthm,amssymb}
\usepackage{graphicx}

\newcommand*{\bigtwo}[1]{\vcenter{\hbox{\scalebox{1.4}{\ensuremath#1}}}}

\newcommand{\shortminus}{\scalebox{0.75}[1.0]{\( - \)}\mkern 1 mu}
\newcommand{\tminusr}{\langle t\shortminus r\rangle}
\newcommand{\tplusr}{\langle{}_{\,} t\!{}_{\!}+\! r\rangle}

\newcommand{\nang}{{\not\negmedspace\nabla\! }}

\DeclareFontFamily{U}{mathx}{\hyphenchar\font45}
\DeclareFontShape{U}{mathx}{m}{n}{<5> <6> <7> <8> <9> <10>
      <10.95> <12> <14.4> <17.28> <20.74> <24.88> mathx10}{}
\DeclareSymbolFont{mathx}{U}{mathx}{m}{n}
\DeclareFontSubstitution{U}{mathx}{m}{n}
\DeclareMathAccent{\widecheck}{0}{mathx}{"71}
\DeclareMathAccent{\wideparen}{0}{mathx}{"75}

\title{Scattering for wave equations with sources close to the lightcone and prescribed radiation fields}
\author{Hans Lindblad and Volker Schlue}

\numberwithin{equation}{section}

\usepackage[raggedright,small,bf]{titlesec}

\newcommand{\ud}{\mathrm{d}}

\newcommand{\beq}{\begin{equation}}\newcommand{\eq}{\end{equation}}
\newcommand{\beqs}{\begin{equation*}}\newcommand{\eqs}{\end{equation*}}
\def\pa{\partial}
\newcommand{\les}{\lesssim}

\theoremstyle{plain}
\newtheorem{prop}{Proposition}[section]
\newtheorem{lemma}[prop]{Lemma}
\newtheorem{remark}{Remark}[section]

\newtheorem{theorem}{Theorem}[section]
\newtheorem{cor}[theorem]{Corollary}

\allowdisplaybreaks[1]

\begin{document}

\maketitle

\begin{abstract}
  We construct  solutions with prescribed radiation fields for wave equations with polynomially decaying sources close to the lightcone.
  In this setting, which is motivated by semi-linear wave equations satisfying the weak null condition, solutions to the forward problem have a logarithmic leading order term on the lightcone and non-trivial homogeneous asymptotics in the interior of the lightcone.
  The backward scattering solutions we construct  are given to second order by explicit asymptotic solutions in the wave zone, and in the interior of the light cone which satisfy novel matching conditions.
  In the process we find novel compatibility conditions for the scattering data at null infinity.
  We also relate the asymptotics of the radiation field towards space-like infinity to explicit homogeneous solutions in the exterior of the light cone.
  This is the setting of slowly polynomially decaying data corresponding to mass, charge and angular momentum in applications.
  We show that homogeneous data of degree $-1$ and $-2$   for the wave equation results in the same logarithmic terms on the lightcone  and homogeneous asymptotics in the interior as for the equations with sources close to the lightcone. The proof requires a delicate analysis of the forward solution close to the light cone and uses the invertibility  of the Funk transform.
\end{abstract}

\bigskip
\tableofcontents

\mathtoolsset{showonlyrefs=true}

\section{Introduction}

For the classical wave equation on $\mathbb{R}^{3+1}$,
\begin{equation}
  \label{eq:wave:F}
  \Box \phi =F\,,\qquad \Box=-\pa_t^2+\triangle_x,
\end{equation}
it is well-known that solutions to the Cauchy problem have a radiation field $\mathcal{F}_0$, provided the source $F$ and the initial data decay sufficiently fast \cite{F62,F80,H97}. In other words, in the \emph{wave zone} where $r\sim t$, the forward in time solution arising from fast decaying data takes the asymptotic form
\begin{equation}
  \label{eq:wavezone:F}
  \phi(t,r\omega)\sim \frac{\mathcal{F}_0(r-t,\omega)}{r},\qquad (r\sim t,r\to \infty),
\end{equation}
and the radiation field decays\footnote{See Section~\ref{sec:radiationfield} for further discussion. For compactly supported sources and data, the radiation field $\mathcal{F}_0(q,\omega)$ is compactly supported in $q$ as a result of the strong Huygens principle. This paper is partly motivated  by non-linear wave equations for which strong Huygens \emph{fails}; see Section~\ref{sec:motiv}. This support property also fails in \emph{curved} spacetimes \cite{F75}, and is replaced by late time tails associated to the relevant interior solutions; see also \cite{BVW18}, and in particular the recent \cite{LO24}.} both in the exterior as $r-t\to\infty$, and in the interior as $r-t\to-\infty$,
\begin{equation}
  \label{eq:81}
  \lim_{q\to\pm\infty}\mathcal{F}_0(q,\omega)=0\,.
\end{equation}

This is \emph{not} true for sources which are slowly decaying in the wave zone:
For example,  we know from the work of the first author \cite{L90a,L17} --- and this will be discussed in Sections~\ref{sec:hom:int} \& \ref{sec:story} --- that for
\begin{equation}
  \label{eq:F:slow}
  F(t,r\omega)\sim \frac{n(r-t,\omega)}{r^2}\qquad (r\sim t,r\to\infty),
\end{equation}
in the wave zone,
the asymptotic form of the forward solution to \eqref{eq:wave:F} with trivial data is
\begin{equation}
  \label{eq:wavezone:log}
  \phi(t,r\omega)\sim \frac{1}{r}\ln\frac{2r}{\langle t-r\rangle} \mathcal{F}_{01}(r-t,\omega)+\frac{1}{r}\mathcal{F}_0(r-t,\omega), \qquad (r\sim t,r\to\infty),
\end{equation}
where $\mathcal{F}_{01}$ and $\mathcal{F}_0$ do not decay in the interior\footnote{The ``interior'' refers plainly to the interior of the light cone in Minkowski space, namely the set $|x|<t$. However, when a limit $r\to\infty$ along $(q=r-t,\omega)$ fixed is taken first, and quantities such as radiation fields are discussed, then the limit $q\to -\infty$ ``towards the interior'', can be viewed -- in the language of relativity -- as a limit \emph{along null infinity, towards time-like infinity}; cf.~Figure~\ref{fig:limits}. Similarly, for the ``exterior'' of the lightcone, where $q\to\infty$ corresponds to a limit towards \emph{space-like infinity}.}
\begin{equation}
  \label{eq:interior:F:limits}
  \lim_{q\to-\infty}\mathcal{F}_{01}(q,\omega)=N_{01}(\omega)\,,\qquad  \lim_{q\to-\infty}\mathcal{F}_0(q,\omega)=N_0(\omega)\,.
\end{equation}
In fact, forward solutions to \eqref{eq:wave:F} with  sources satisfying \eqref{eq:F:slow} display \emph{homogeneous} asymptotics in the interior of the light cone,\footnote{See Section~\ref{sec:hom:int} for discussion. In the interior of the lightcone, the relevant asymptotic regime can be identified with the space of time-like rays $\gamma_y(t):t\mapsto (t,yt)$  with fixed $|y|<1$; see Figure~\ref{fig:limits}. Homogeneous functions in the interior are functions of $y$ only, and play an important role in this paper since they capture the relevant asymptotics in this regime.}
\begin{equation}
  \label{eq:interior:asym}
  \phi\sim \Psi_1[N](t,r\omega)=\frac{1}{4\pi}\int_{\mathbb{S}^2}\frac{N(\sigma) \ud S(\sigma)}{t-r\langle\sigma,\omega\rangle}, \qquad (r/t<1,t\to\infty),\qquad \text{ where } N(\sigma)= \int_{-\infty}^\infty n(q,\sigma)\ud q\,.
\end{equation}
Furthermore the functions $N_{01}(\omega)$ and $N_0(\omega)$ appear in the expansion of $\Psi_1[N]$ as $r/t\to 1$ (see Section~\ref{sec:interior}).
Similar statements (see Section~\ref{sec:asym:cubic}) can be made for sources of the form
\begin{equation}
  \label{eq:F:cubic}
  F(t,r\omega)\sim \frac{m(r-t,\omega)}{r^3},\qquad (r\sim t,r\to\infty)\,.
\end{equation}

The source terms \eqref{eq:F:slow} which are falling off quadratically in $r$ in the wave zone occur effectively for wave equations with a quadratic nonlinearity satisfying the weak null condition; see Section~\ref{sec:motiv}. In particular the occurence of a log-term in the asymptotics \eqref{eq:wavezone:log} is relevant for these nonlinearities. Similary  source terms \eqref{eq:F:cubic} with cubic fall-off in $r$ occur for wave equations satisfying the null condition.

\subsection{Scattering problems with sources in the wave zone.}
\label{sec:scattering:intro}

In settings where the asymptotics are captured by \eqref{eq:wavezone:F}, and the decay of the radiation field is sufficiently fast,
there is a well-defined scattering problem:

\smallskip
\begingroup
\leftskip 10pt
\rightskip 10pt

\noindent\textsl{Given a radiation field $\mathcal{F}_0$, show the existence of a unique solution to \eqref{eq:wave:F} so that \eqref{eq:wavezone:F} holds with the prescribed radiation field $\mathcal{F}_0$.}

\endgroup

\smallskip
\noindent We have shown in \cite{LS20} that for radiation fields that decay like $|\mathcal{F}_0(q,\omega)|\les 1/\langle q\rangle^\gamma$, with $1/2<\gamma<1$,  there exists  a unique solution  to \eqref{eq:wave:F}, say with $F=0$, so that \eqref{eq:wavezone:F} holds, and moreover the solutions decay globally, $|\phi|\les 1\big/\langle t+r\rangle\langle t-r\rangle^{\gamma'}$, at a rate $\gamma'< \gamma$. See Theorem~1.1 in~\cite{LS20}; for related results see Section~\ref{sec:further} below.

However, for asymptotics of the form \eqref{eq:wavezone:log} it is not \emph{a priori} clear what the scattering data should be and whether it can be chosen freely. Neither  is it \emph{a priori} clear whether the interior asymptotics \eqref{eq:interior:asym} are part of the scattering data or whether the functions in \eqref{eq:interior:F:limits} can be chosen freely, or independently.

In this paper,  we identify precisely the part of the scattering data \emph{that can be prescribed freely} for the problems \eqref{eq:F:slow} and \eqref{eq:F:cubic},
and we will derive the compatibility conditions imposed by the source on the prescription of the radiation field at null infinity.
The procedure developed in this paper can be applied to a large class of semi-linear wave equations satisfying the weak null condition; see Section~\ref{sec:motiv}.

\subsubsection{Quadratic source terms.}
\label{sec:intro:quadratic}

Consider the equation
\begin{equation}
  \label{eq:sourcbegining}
  \Box\phi+\frac{n(r-t,\omega)}{r^2}\chi=0,
\end{equation}
where $\chi$ is a cutoff supported in the wave zone  away from the origin; in fact, we pick $\chi=\chi(\tminusr/r)$ with $\chi(s)=1$ when $s<1/4$ and $\chi=0$ when $s>1/2$.
We want to solve the equation \eqref{eq:sourcbegining} with prescribed asymptotics as $t$ tends to infinity. However, the forward solution with vanishing data at past infinity does not have the simple asymptotics of solutions to the homogeneous wave equation, and convolution of the source term in  \eqref{eq:sourcbegining} with the backward fundamental solution of $\Box$ is unbounded.

We recall from \cite{L90a,L17} that the convolution of the source term in \eqref{eq:sourcbegining} with the forward fundamental solution $E_+=\delta(t^2-|x|^2)H(t)/2\pi$ of $-\Box$,
is given by\footnote{See Appendix~\ref{sec:sourseformulas} for discussion, and \eqref{eq:k2sol} and \eqref{eq:chi:rho} for precise form of the cutoff.}
 \begin{equation}\label{eq:Phi:21sec2intro}
\phi(t,x)=\Phi[n](t,r\omega)=\!\frac{1}{4\pi}\!\int_{{}_{\!}r-t}^{\infty} \! \int_{\mathbb{S}^2}\frac{\chi\,
 n({q},{\sigma})\, dS({\sigma})\,d {q}\!\!}{t\!-\!r\!+\!{q}\!+\!r(1\shortminus \langle\, \omega,{\sigma}\rangle)} \,.
\end{equation}
From this formula it follows that one has asymptotics in $t-r$ and $1/r$ in the wave zone  $r\sim t$, and asymptotics in $(r-t)/r$ and $1/r$ in the interior, near the lightcone for $r/t<1$.

In fact, in the wave zone, we write $\Phi[n](t,r\omega)=\Phi[n-n_\omega]+\Phi[n_\omega]$, and find that  for $\omega$ and $r-t$ fixed, as $r\to\infty$, the main contribution comes from the source  $n_\omega(q)=n(q,\omega)$:
\begin{equation}\label{eq:asymptotics:log:div}
    \Phi[n_\omega](t,r\omega)\sim \frac{1}{r}\ln\Bigl(\frac{2r}{\tminusr}\Bigr) \mathcal{F}_{01}(r-t,\omega)\,,\qquad \mathcal{F}_{01}(r-t,\omega)= \frac{1}{2}\int_{r-t}^\infty n(q,\omega)\ud q\,, \qquad (r\sim t,r\to\infty)\,.
  \end{equation}
  We refer the reader to \cite[Section 1.2.2 \& Section 7]{L17} for further discussion and derivations.

Moreover, for the forward asymptotics in the interior, for $r/t< 1$ fixed, as $t\to\infty$, we derive from the formula \eqref{eq:Phi:21sec2intro} (see also \cite[Section 9]{L17}) that to leading order
\begin{equation}\label{eq:interior:leading:hom}
  \Phi[n](t,r\omega) \sim \frac{1}{4\pi }\int_{\mathbb{S}^2}\frac{N(\sigma)}{t-r\langle\omega,\sigma\rangle}\ud S(\sigma),\quad (r/t<1,t\to\infty)\,,\quad \text{ where}\quad N(\sigma)= \int_{-\infty}^\infty\!\!\! n(q,\sigma)\ud q\,.
\end{equation}
This is a homogeneous function of degree $-1$, which can be viewed as a superposition of plane waves.

\subsubsection{Matching of the radiation field to the interior homogeneous solution.}
\label{sec:matching:intro}

For the forward problem, the asymptotics in the wave zone \emph{automatically match} the asymptotics in the interior of the lightcone.
Let us verify this at leading order with the formulas above; see Figure~\ref{fig:limits} for geometric illustration:

In \eqref{eq:asymptotics:log:div} we have \emph{first} derived the asymptotic form for $\Phi[n]$ as $r\to\infty$  with $r-t$ and $\omega$ fixed,
 and can now \emph{secondly} take the limit $q\to -\infty$ towards the interior:
\begin{equation}
  \label{eq:match:F01:intro}
  2\mathcal{F}_{01}(q,\omega)= \int_{q}^\infty n(q,\omega)\ud q \longrightarrow N(\omega) \qquad (q\to-\infty)\,.
\end{equation}
This matches the limit obtained from \eqref{eq:interior:leading:hom} where we have \emph{first} derived the asymptotic form as $t\to\infty$ with $r/t<1$ fixed, and now \emph{secondly}  take the limit $r/t\to 1$, which can be read off from the expansion of the homogeneous solution towards the lightcone:
\begin{equation}\label{eq:phi:1:infty}
  \phi_{1,\infty}= \frac{1}{4\pi}\int_{\mathbb{S}^2}\frac{N(\sigma) \ud S(\sigma)}{t-r\langle\omega,\sigma\rangle}\sim \frac{1}{2r} N(\omega) \ln\Bigl\lvert \frac{2r}{r-t}\Bigr\rvert+\frac{1}{r}N_0(\omega),\qquad (r/t\to 1).
\end{equation}

\begin{figure}[tb]
  \centering
  \includegraphics[scale=0.4]{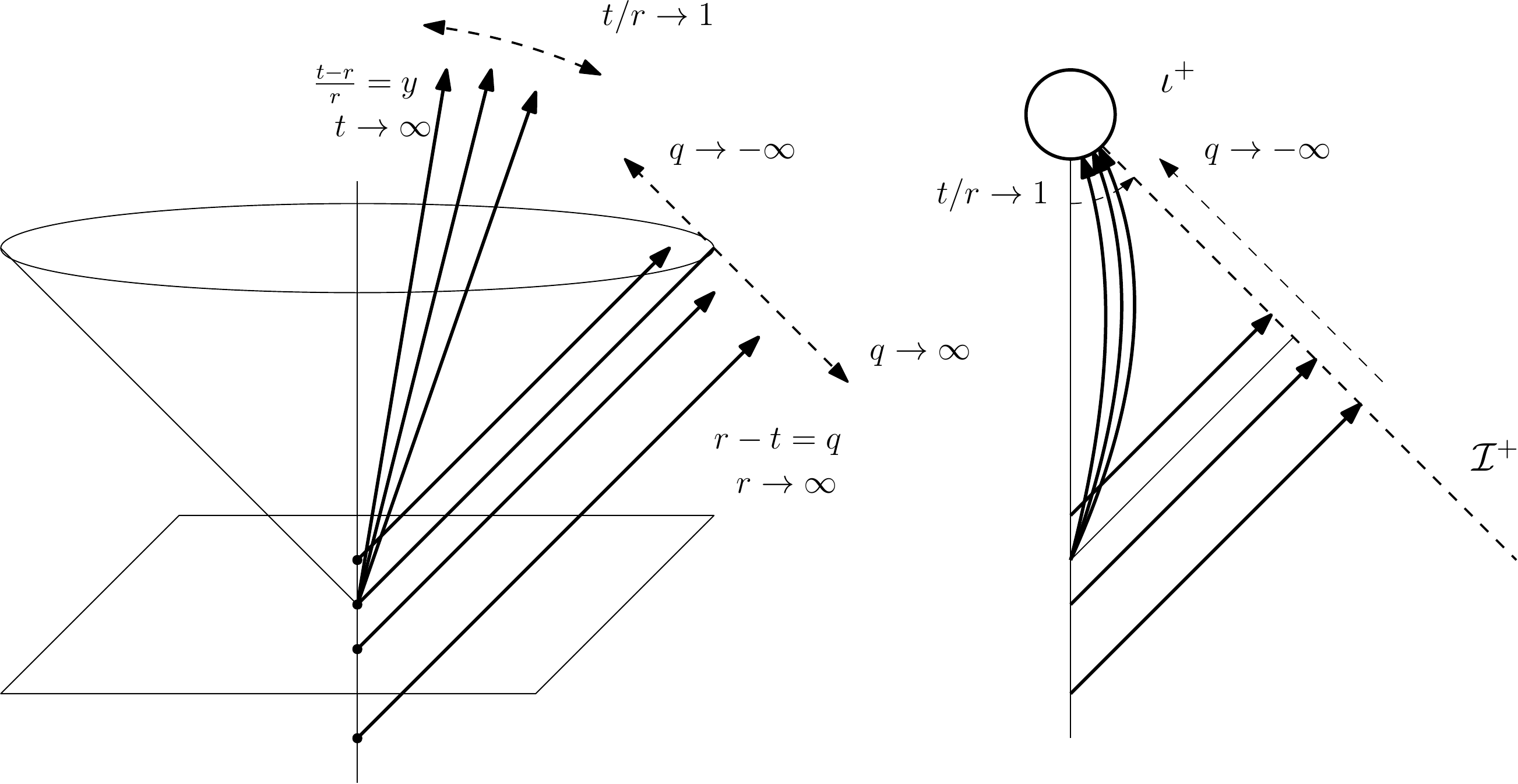}
  \caption{Time-like and null rays depicted in Minkowski spacetime (left) and the Penrose diagram
  (right) to illustrate the matching conditions arising from limits taken first along null rays
  $r-t=q$, towards null infinity $(\mathcal{I}^+)$ as $r\to \infty$, and then second towards the interior as $q\to-\infty$, or, alternatively, first along time-like curves $t/r-1=y$, towards time-like infinity $(\iota^+)$ as $t\to\infty$, and then secondly towards the lightcone as $r/t\to 1$. }
  \label{fig:limits}
\end{figure}

Similarly at the next order, the expansion of $\Phi[n]$ in the wave zone,
\begin{equation}
  \label{eq:rad:intro}
  \Phi[n]\sim \phi_{rad}=\frac{1}{r}\ln\frac{2r}{\tminusr}\  \mathcal{F}_{01}(r-t,\omega)+\frac{1}{r}\mathcal{F}_0(r-t,\omega),\qquad (r\sim t,r\to\infty),
\end{equation}
matches the expansion of the homogeneous solution \eqref{eq:phi:1:infty} in the interior, in the sense that:
\begin{equation}
  \label{eq:match:F0:intro}
  \lim_{q\to -\infty}\mathcal{F}_0(q,\omega)=N_0(\omega) .
\end{equation}

For the forward problem \eqref{eq:match:F01:intro} and \eqref{eq:match:F0:intro} are automatically satisfied as a result of the formula \eqref{eq:Phi:21sec2intro}.
The asymptotics towards time-like infinity match the asymptotics at null-infinity; the limits depicted in Figure~\ref{fig:limits} \emph{commute}.

However, for the scattering problem these are \emph{matching conditions} that relate  the expansion  in the wave zone, $\phi_{rad}$ in \eqref{eq:rad:intro}, to a non-trivial homogeneous solution in the interior, as given in \eqref{eq:phi:1:infty}. The condition \eqref{eq:match:F0:intro} says in particular that the radiation field $\mathcal{F}_0(q,\omega)$ does not decay in the interior as $q\to -\infty$, and we will see that  as a result of the other matching conditions, even after subtracting the limit $N_0(\omega)$, $\mathcal{F}_0-N_0\chi_{q<0}$ cannot be  prescribed arbitrarily.

Here and below $\chi_{q<0}=\chi_{q<0}(s)$ is a smooth cutoff function so that  $\chi_{q<0}(s)=1$ when $s<-1$ and $\chi_{q<0}(s)=0$ for $s>0$. It is introduced because the limit $N_0(\omega)$ only needs to be subtracted from the radiation field $\mathcal{F}_0$ in the interior.

\subsubsection{Scattering results with  quadratic source terms.}

In this paper, \emph{we will identify which part of the radiation field can be prescribed freely and independently of the source}, and we find a solution to \eqref{eq:sourcbegining} with the prescribed scattering data, that matches to an explicit homogeneous solution in the interior  determined by the source.

\begin{theorem}[Scattering solutions with quadratic source terms] \label{thm:n}
  Let $n(q,\omega)$ be a smooth source function satisfying $|(\langle q\rangle \partial_q)^k\partial_\omega^\alpha n|\les \langle q\rangle^{-2-2\gamma}$,  for some $1/2\leq \gamma<1$, 
  and let $\mathcal{H}_0(q,\omega)$ be a smooth function satisfying \[|(\langle q\rangle \partial_q)^k\partial_\omega^\alpha\mathcal{H}_0|\les \langle q\rangle^{-2}\,.\]
  Then  there is a smooth solution $\phi$ to \eqref{eq:sourcbegining} with the asymptotics \eqref{eq:rad:intro} in the wave zone,
 \begin{equation}
   \phi\sim \frac{1}{r}\ln\frac{2r}{\tminusr}\  \mathcal{F}_{01}(r-t,\omega)+\frac{1}{r}\mathcal{F}_0(r-t,\omega),\qquad (r\sim t,r\to\infty),
 \end{equation}
 where $\mathcal{F}_{01}(q,\omega)$ is given by \eqref{eq:match:F01:intro}, and the radiation field is given by
 \begin{equation}
   \label{eq:thm:F}
   \mathcal{F}_0(q,\omega)=N_0(\omega)\chi_{q<0}+M_0(\omega)\frac{q}{\langle q\rangle^2}\chi_{q<0}+\mathcal{H}_0(q,\omega)\,,
 \end{equation}
 \emph{provided $\mathcal{H}_0$ satisfies a compatibility condition of the form}
 \begin{equation}
   \label{eq:thm:n:compatibility}
   \int_{-\infty}^{\infty}\mathcal{H}_0(q,\omega)\ud q=\mathcal{P}(\omega)\,.
 \end{equation}
 Here $N_0$, $M_0$, and $\mathcal{P}$ are functions on the sphere determined from the source function $n$ alone. 
 Furthermore the solution $\phi$ has the interior asymptotics \eqref{eq:interior:leading:hom}.
\end{theorem}

The solution $\phi$ obtained in the theorem is of the form
\begin{equation}
  \phi=\psi_{asym}+\psi_{rem},
\end{equation}
where $\psi_{asym}$ is an approximate solution, and constructed, guided by the forward solution, by interpolating between a second order expansion $\psi_{rad}$ in the wave zone, and a homogeneous solution $\psi_{hom}$ in the interior:
\begin{equation}
  \psi_{asym}=\psi_{rad}\, \chi+(1-\chi)\, \psi_{hom}\,,\qquad \psi_{hom}=\phi_{1,\infty}+\phi_{2,\infty}\,.
\end{equation}
Here $\chi=\chi(\tminusr/r^{1/2})$ is a smooth cutoff close to the lightcone, and  $\phi_{1,\infty}=\phi_{1,\infty}[N]$ and $\phi_{2,\infty}=\phi_{2,\infty}[M]$ are solutions to $\Box\psi=0$ in the interior $t>|x|$, which are \emph{homogeneous}  in $(t,x)$ of degree $-1$, and $-2$, respectively, and only determined from the functions
\begin{equation}
  N(\omega)=\int_{-\infty}^{\infty}n(q,\omega)\ud q\,,\qquad M(\omega)=\int_{-\infty}^{\infty} q \,n(q,\omega)\ud q\,.
\end{equation}

The second order expansion in the wave zone will be discussed in Section~\ref{sec:higherordernull}, see~\eqref{eq:secondorderapproximatesource2}.  The relevant homogeneous solutions are described in Section~\ref{sec:interior}. The proof that an approximate solution with matching asymptotics can be constructed in this way for \eqref{eq:sourcbegining} is the content of Section~\ref{sec:matching}.

In fact, we show that $\psi_{asym}$ satisfies
\beq
\Box\, \psi_{asym} +\chi\frac{n(r\!-t,\omega)}{r^2}=R_{asym},
\eq
where the error $R_{asym}$ decays faster than the source term in \eqref{eq:sourcbegining}, (see Section~\ref{sec:approximate} for precise bounds).
Then we estimate the remainder
\beq\label{eq:remsol}
\Box\, \psi_{rem}=-R_{asym}\,,\qquad \psi_{rem}=\phi-\psi_{asym}\,,
\eq
by convolving with the backward fundamental solution, which now is well-defined and decaying faster than a solution to the homogeneous equation; see  Section~\ref{sec:remainder}.

The origin of the $1/q$ tail in \eqref{eq:thm:F}, which cannot be chosen freely, is the matching condition with  the homogeneous degree $-2$ solution $\phi_{2,\infty}$ which contributes to leading order the following term as $r/t\to 1$:
\begin{equation}
  \phi_{2,\infty}[M](t,r\omega)\sim\frac{1}{r^2}\frac{r}{r-t} M_0(\omega),\qquad (t\to\infty, r/t\to 1),\qquad M_0(\omega)=\frac{1}{2}M(\omega)\,.
\end{equation}

The compatibility condition is derived in Section~\ref{sec:integrabilitycondition}.
The function  $\mathcal{P}(\omega)$ in \eqref{eq:thm:n:compatibility} is the solution to an elliptic equation \eqref{eq:triangle:P} on the sphere.

\subsubsection{Cubic source terms.}
\label{sec:cubic:intro}

We proceed similarly for the equation
\begin{equation}
  \label{eq:sourcbegining:cubic}
  \Box\phi+\frac{m(r-t,\omega)}{r^3}\chi=0\,.
\end{equation}
Here we recall from \cite{L90a,L17} that the convolution of the cubic source term with the fundamental solution of the wave operator is given by\footnote{See Appendix~\ref{sec:sourseformulas} for discussion, and \eqref{eq:k3solformula} and \eqref{eq:chi:rho} for precise form of the cutoff.}
\beq
{\Phi_{\chi}}^{\!\!\!3}[m](t,r\omega)=\int_{r-t}^{\infty} \frac{1}{2\pi}\int_{\mathbb{S}^2}{\frac{
 m({q},{\sigma})\chi \ud S({\sigma})}{(t-r+{q})(t+r+q)}\,
}\, \ud {q}\,.
\eq
As in the situation with quadratic source terms discussed above, 
solutions to \eqref{eq:sourcbegining:cubic} have homogeneous asymptotics in the interior of the lightcone $(r/t<1,t\to\infty)$
\begin{equation}
  \label{eq:Phi3:intro}
  \phi\sim \Phi^3[m]=\frac{2 \overline{M}}{(t+r)(t-r)} \,, \quad\text{ where }  \overline{M}=\frac{1}{4\pi}\int_{\mathbb{S}^2}M(\sigma)\ud S(\sigma)\,,\quad M({\sigma})=\int_{-\infty}^{+\infty} m(q,\sigma)\, \ud q\,.
\end{equation}
The relevant homogeneous solution is discussed in Section~\ref{sec:asym:cubic},
and for the asymptotics in the wave zone we have from Section~\ref{sec:radiationfield} that
\begin{equation}
  \label{eq:wavezone:cubic:intro}
  \phi \sim \frac{\mathcal{G}_0(r\shortminus t,\omega)}{r}
  +\frac{\mathcal{G}_1(r\shortminus t,\omega)}{r^2},\qquad
  \mathcal{G}_1(q,\omega)=-\int_q^\infty  m(q',\sigma)\, \ud q'/2
  \qquad (r\sim t,r\to\infty)\,.
\end{equation}
Thus while $\phi$ has a radiation field,
the $1/\langle q\rangle$ tail of the radiation field $\mathcal{G}_0(q,\omega)$ is determined by the source,
\begin{equation}
  \label{eq:G0:tail}
  \mathcal{G}_0(q,\omega)\sim \overline{M}/\langle q\rangle, \qquad (q\to -\infty)\,.
\end{equation}

\begin{theorem}[Scattering solutions with cubic source terms]\label{thm:m}
  Let $m$ be a smooth source function satisfying $|(\langle q\rangle\partial_q)^k\partial_\omega^\alpha m(q,\omega)|\les  \langle q\rangle^{-1-2\gamma}$,  for some $1/2\leq \gamma<1$,
  and let $\widetilde{\mathcal{G}_0}$ be a smooth function satisfying \[|(\langle q\rangle \pa_q)^k\pa^\alpha_\omega \widetilde{\mathcal{G}_0}| \les \langle q\rangle^{-2}\,.\]
  Then there exists a smooth solution $\phi$ to \eqref{eq:sourcbegining:cubic} with the asymptotics \eqref{eq:wavezone:cubic:intro} in the wave zone, where
  \begin{equation}
    \mathcal{G}_0(q,\omega)=\overline{M}{\langle q\rangle^{-1}}\chi_{q<0}+\widetilde{\mathcal{G}_0}(q,\omega)
  \end{equation}
provided $\widetilde{\mathcal{G}_0}$ satisfies a compatibility condition of the form
\begin{equation}
  \label{eq:thm:m:compatibility}
  \int_{-\infty}^\infty \widetilde{\mathcal{G}_0}(q,\omega)=Q(\omega)\,.
\end{equation}
Here $M(\omega)$ and $\overline{M}$ are determined by the source $m$ as in \eqref{eq:Phi3:intro},
and $Q(\omega)$ is also determined by the source  as the unique solution to the following elliptic equation on $\mathbb{S}^2$:
\begin{equation}
  \triangle_\omega Q(\omega)=-M(\omega)+\overline{M}\,.
\end{equation}
\end{theorem}

The matching condition \eqref{eq:G0:tail} as well as the compatibility conditions  are derived in Section~\ref{sec:matching:cubic}.

\subsection{Exterior problem with slowly decaying data.}
\label{sec:ext:intro}

The scattering solutions discussed in Theorem~\ref{thm:n} \& \ref{thm:m} all have homogeneous asymptotics in the interior, but decay fast in the exterior of the light cone. In particular we made the choice
\begin{equation}
  \label{eq:F0:ext:intro}
  \lim_{q\to\infty}\mathcal{F}_{01}(q,\omega)=0\,,\qquad   \lim_{q\to\infty}\mathcal{F}_0(q,\omega)=0\,.
\end{equation}
This is justified by the asymptotics for the forward problem \emph{with fast decaying data}.
Indeed we recall from \cite{H97} that the radiation field of solutions to the homogeneous equation
\begin{equation} \label{eq:wave:intro:data}
\Box \,\phi=0,\qquad \phi\,\Big|_{t=0}=f,\quad \pa_t\phi\, \Big|_{t=0}=g\,,
\end{equation}
is given by
\beq\label{eq:F0:Radon}
{\mathcal F}_0(q,\omega)=\mathcal{R}[g](q,\omega)-\pa_q \mathcal{R}[f](q,\omega),
\eq
where $\mathcal{R}[g]$ denotes the Radon transform of the data $g$:
\beq
\mathcal{R}[g](q,\omega)=\int \delta\big(q-\langle \omega,y\rangle\big) g(y)\, dy =\int_{\langle \omega,y\rangle=q} g(y) \, dS(y)\,.
\eq
For $\mathcal{F}_0$ to be well defined we need $g(y)$ and $\nabla f(y)$ to decay like $\langle y\rangle^{-2-\kappa}$ for $\kappa\!>\!0$.  In that case it follows that $\mathcal{F}_0(q,\omega)\to 0$ as $q\!\to\!\infty$.
However, this excludes data that are relevant in many important physical situations that involve the description of mass, charge and angular moment, which correspond to terms that decay like $r^{-1}$ or $r^{-2}$.

In this section, we present results that allow us to solve the scattering problem in situations when the radiation field does not decay in the exterior, corresponding forward solutions with \emph{slowly decaying data}.
In fact, for \emph{arbitrary} limits
\begin{equation}
  \label{eq:F0:ext:N:intro}
  \lim_{q\to\infty}\mathcal{F}_{01}(q,\omega)=N_{01}(\omega)\,,\qquad   \lim_{q\to\infty}\mathcal{F}_0(q,\omega)=N_0(\omega)\,,
\end{equation}
we find a homogeneous solution in the exterior that matches these limits on the lightcone.

 \subsubsection{Homogeneous degree $-1$ and $-2$ solutions to the homogeneous equation in the exterior.}
\label{sec:hom:ext:intro}

In Section~\ref{sec:hom:ext} we will give a complete characterisation of the homogeneous degree $-1$ solutions in the exterior of the light cone corresponding to the limits \eqref{eq:F0:ext:N:intro}.

\begin{theorem}[Homogeneous degree $-1$ solutions in the exterior]\label{thm:hom:N}
  Let $N_{01}(\omega)$ and $N_0(\omega)$ be smooth functions on the sphere $\mathbb{S}^2$.
  Then there exists a unique homogeneous degree $-1$ solution $\phi_1$ to the wave equation $\Box\phi=0$ in the exterior of the lightcone $|x|>t$, so that
  \begin{equation}
    \phi_1\sim \frac{1}{r}\ln\frac{2r}{r-t}N_{01}(\omega)+\frac{1}{r} N_{0}(\omega),\qquad (r/t\to 1)\,.
  \end{equation}
\end{theorem}
To prove the theorem we consider solutions to the initial value problem with homogeneous initial data,
\begin{equation}
  \Box\phi_1=0\,,\qquad \phi_1(0,r\omega)=\frac{M(\omega)}{r}\,,\quad \partial_t\phi_1(0,r\omega)=\frac{N(\omega)}{r^2}\,,\qquad (|x|>t)\,,
\end{equation}
and calculate explicitly the leading order coefficents in the expansion near the light cone as $r/t\to 1$.
In fact, to highest order we find that
\begin{equation}
  \label{eq:hom:ext:asym:intro}
  \phi_1\sim \frac{1}{2r}\ln\frac{2r}{r-t}
  \Bigl(\mathcal{F}[N](\omega)+\mathcal{G}[M](\omega)\Bigr)\,,\qquad (r/t\to 1),
\end{equation}
where $\mathcal{F}$ and $\mathcal{G}$ are known transformations of functions on the sphere: $\mathcal{F}$ is the Funk transform, which maps even to even functions on the sphere \cite{F13}, and $\mathcal{G}$ is a related transformation which maps odd to odd functions on the sphere; see Section~\ref{sec:hom:ext:one}. We appeal to the invertibility results in \cite{BEGM03}  to show that the transformation that maps the data $(M,N)$ to the functions $(N_{01},N_0)$ is bijective, and conclude that the functions $M$ and $N$ are \emph{uniquely determined} by $N_{01}$ and $N_0$.

\begin{remark}
  Note that for data
  \begin{equation}
    \phi_1(0,r\omega)=\frac{M(\omega)}{r}\,,\ M(\omega)=M(-\omega) \text{ even,}\qquad \partial_t\phi_1(0,r\omega)=\frac{N(\omega)}{r^2}\,,\ N(\omega)=-N(-\omega) \text{ odd,}
  \end{equation}
  we see from \eqref{eq:hom:ext:asym:intro} that $N_{01}(\omega)=0$, because even functions are in the kernel of $\mathcal{G}$, and odd functions in the kernel of $\mathcal{F}$, and hence there is no logarithmic term to leading order in the expansion.
  This is in particular true when $M(\omega)=M$ is a constant and $N(\omega)=\omega$. 
\end{remark}

One can proceed similarly in the case of homogeneous degree $-2$ solutions, and compute explicitly the asymptotics of solutions to
\begin{equation}
  \Box\phi_2=0\qquad (|x|>t)\,,\qquad \phi_2(0,r\omega)=\frac{K(\omega)}{r^2}\,,\quad \partial_t\phi_2(0,r\omega)=\frac{L(\omega)}{r^3}\,.
\end{equation}
However, it turns out that one can infer immediately from the homogeneous degree $-1$ result that:
\begin{cor}[Homogeneous degree $-2$ solutions in the exterior]\label{thm:hom:M}
  Suppose  $M_{0}(\omega)$ and $M_1(\omega)$ are functions on the sphere $\mathbb{S}^2$.
  Then there exists a unique homogeneous degree $-2$ solution $\phi_2$ to the wave equation $\Box\phi=0$ in the exterior of the lightcone $|x|>t$, such  that
  \begin{equation}
    \label{eq:phi2:intro}
    \phi_2\sim \frac{1}{r^2}\frac{r}{r-t} M_0(\omega)+\frac{1}{r^2}\ln\frac{2r}{r-t}M_{11}(\omega)+\frac{1}{r^2} M_{1}(\omega)\qquad (r/t\to 1)\, ,
  \end{equation}
  where $2M_{11}(\omega)=-\triangle_\omega M_0(\omega)$.
\end{cor}

These results are proven in Sections~\ref{sec:hom:ext:one} \& \ref{sec:hom:ext:two}, using the representation formulas derived in Appendix~\ref{app:ext}.

\subsubsection{Scattering solutions with homogeneous asymptotics in the exterior and interior.}

In Section~\ref{sec:ext:homogeneous} the homogeneous solutions in the exterior are connected to scattering solutions with prescribed radiation fields. We prove this for the homogeneous wave equation
\begin{equation}
  \label{eq:wave:intro}
  \Box\phi=0\,,
\end{equation}
and show that such solutions necessarily asymptote to homogeneous solutions in the interior, as described in Section~\ref{sec:scattering:intro}.

\begin{theorem}[Scattering solutions with homogeneous asymptotics in the exterior]\label{thm:ext:scattering}
  Let $N_{01}(\omega)$, $N_0^{ext}(\omega)$, and $M_0^{ext}(\omega)$ be smooth functions on the sphere, and $C_0$ a constant.
  Suppose $\mathcal{H}_0(q,\omega)$ is a smooth function satisfying \[|(\langle q\rangle\partial_q)^k\partial_\omega^\alpha \mathcal{H}_0|\les \langle q\rangle^{-2}\,.\]
  Then there exists a solution $\phi$ to the homogeneous wave equation \eqref{eq:wave:intro} with the asymptotics
  in the wave zone,
  \begin{equation}
    \label{eq:phi:asym:hom}
 \phi\sim    \ln{\!\Big|\!\frac{2\,r}{\langle t\shortminus r{}_{\!}\rangle{}_{\!}}\Big|}
\frac{\mathcal{F}_{\!01\!}(r\!-\!t,\omega) \! \! }{r}+\frac{\!\mathcal{F}_{\!0}(r\!-\!t,\omega)\!\!}{r} \qquad (r\sim t,r\to\infty)\,,
  \end{equation}
   where $\mathcal{F}_{01}(q,\omega)=N_{01}(\omega)$ is a function of $\omega$ only, and the radiation field $\mathcal{F}_0$ is given by
  \begin{equation}
    \label{eq:F0:scattering}
    \mathcal{F}_0(q,\omega)=N_0^{ext}(\omega)\chi_{q>0}+M_0^{ext}(\omega) \frac{q}{\langle q\rangle^2} \:\chi_{q>0}
    +N_0^{int}(\omega)\chi_{q<0}+M_0^{int}(\omega)\frac{q}{\langle q\rangle^2}\: \chi_{q<0}+\mathcal{H}_0(q,\omega)\,,
  \end{equation}
  where $\chi_{q<0}$ and $\chi_{q>0}=1-\chi_{q<0}$ are smooth cutoff functions in $q$ as above, and
  \begin{equation}
    \label{eq:lim:ext:int:rel}
    N_0^{int}(\omega)=\frac{1}{2\pi}\int_{\mathbb{S}^2}\frac{N_{01}(\sigma)-N_{01}(\omega)}{1-\langle \sigma,\omega\rangle}\ud S(\sigma)\,,\qquad M_0^{int}(\omega)=M_0^{ext}(\omega)+C_0\,.
  \end{equation}
  Furthermore, the solution $\phi$ has  the interior homogeneous asymptotics  \eqref{eq:interior:leading:hom} (with $N_{01}$ as the source function), and  exterior homogeneous asymptotics given by Theorem~\ref{thm:hom:N}  (with $N_0=N_0^{ext}$).
\end{theorem}

\begin{remark}
  In \cite{LS20} we have shown that for any smooth function $\mathcal{F}_0$ satisfying $|(\langle q\rangle\partial_q)^k\partial_\omega^\alpha \mathcal{F}_0|\les \langle q\rangle^{-\gamma}$, $1/2<\gamma<1$ there exists a solution $\phi$ to \eqref{eq:wave:intro} with the asymptotics $\phi\sim\mathcal{F}_0(r-t,\omega)/r$  in the wave zone.
  This does not cover the case of slowly decaying data. Indeed,  initial data that decays like $r^{-1}$ corresponds to $\gamma=0$.
  Theorem~\ref{thm:ext:scattering}  covers all cases related to slowly  decaying data and gives a complete characterisation of the asymptotics  of the solution  towards time-like and space-like infinity at this order.
\end{remark}

\begin{remark}
  Friedlander showed in \cite{F80} that for all finite energy solutions $\phi$ to \eqref{eq:wave:intro:data},
  with $\|\phi\|_{\mathcal{H}_E}^2=\int |\nabla f|^2+g^2 \ud x<\infty$, the time-derivative of $\phi$ has a radiation field,
  \begin{equation}
    \label{eq:41}
    \partial_t\phi\sim\frac{\mathcal{G}_0(r-t,\omega)}{r}\,,\text{ and }\qquad \int_{-\infty}^{\infty} \int_{\mathbb{S}^2} \mathcal{G}_0(q,\omega)^2\ud S(\omega) \ud q \les \|\phi\|_{\mathcal{H}_E}^2\,.
  \end{equation}
This is consistent with the expansion \eqref{eq:phi:asym:hom}, where $\mathcal{F}_{01}(q,\omega)=N_{01}(\omega)$,
which after differentiating corresponds to $\mathcal{G}_0(q,\omega)\sim N_{01}(\omega)\frac{q}{\langle q \rangle^2} -\partial_q\mathcal{F}_0(q,\omega)$.
\end{remark}

\subsubsection{Compatibility condition.}

For solutions to the forward problem \eqref{eq:wave:intro:data} with fast decaying data,
we see from \eqref{eq:F0:Radon} that ${\mathcal F}_0(q,\omega)$ cannot be an arbitrary function of compact support in $q$ but has to satisfy a certain compatibility condition since
\beq\label{eq:F0:comp}
\int {\mathcal F}_0(q,\omega)\, dq= \int g(y)\, dy
\eq
is independent of $\omega$.
The compatibility conditions \eqref{eq:thm:n:compatibility} and \eqref{eq:thm:m:compatibility} stated in Theorem~\ref{thm:n},~\ref{thm:m} can be viewed as generalisations of \eqref{eq:F0:comp} to equations with sources. In our proof these compatibility conditions arise as a result of matching conditions, and are a consequence in Theorem~\ref{thm:n}~\&~\ref{thm:m} of our choice  to match to trivial exterior solutions,\footnote{In particular, the stated compatibility conditions ensure that the asymptotic solution is trivial in the exterior up to homogeneous degree $-2$. It does \emph{not} imply conversely that the solution is compactly supported at $t=0$.}
as indicated for example by \eqref{eq:F0:ext:intro}.

In Theorem~\ref{thm:ext:scattering} the compatibility condition is absent because the solution is now allowed to have non-trivial exterior asymptotics, as indicated by \eqref{eq:F0:ext:N:intro}. Note that in Theorem~\ref{thm:ext:scattering} we do not  prescribe the homogeneous degree $-2$ solution in the exterior.

Here it is useful to consider the following alternative formulation of Theorem~\ref{thm:ext:scattering}:
If in addition to $N_{01}$ and $N_0^{ext}$, and $M_0^{ext}$ we also prescribe the function $M_1^{ext}$,
then the exterior homogeneous solutions are fully determined by Theorem~\ref{thm:hom:N} and Corollary~\ref{thm:hom:M}.
While our proof allows us to match to any homogeneous solutions in the exterior, this produces a compatibility condition of the form
\begin{equation}
  \label{eq:H0:compatibility}
  \int_{-\infty}^{\infty} \mathcal{H}_0(q,\omega)\ud q=\mathcal{P}(\omega)
\end{equation}
where $\mathcal{P}(\omega)$ is the solution of an elliptic problem on $\mathbb{S}^2$ determined by the four functions on the sphere above. Moreover in the case of fast decaying radiation fields, \footnote{The radiation field $\mathcal{F}_0$ may decay fast towards the exterior, but show no decay towards the interior: this is the case when $N_0^{ext}=0$ and $M_0^{ext}=0$, but $N_0^{int}$ and $M_0^{int}$ do not vanish. The radiation field $\mathcal{F}_0$ only decays towards the interior if the integral for $N_0^{int}$ in \eqref{eq:lim:ext:int:rel} vanishes, so in particular if $N_{01}=0$, and if $C_0=0$ in \eqref{eq:lim:ext:int:rel}. In this case the asymptotic solution has trivial homogeneous degree $-1$ and $-2$ behaviour \emph{both} in the exterior and interior, but $\mathcal{F}_0=\mathcal{H}_0$ is subject to the stated compatibility condition.}
namely when $\mathcal{F}_0=\mathcal{H}_0$, the compatibility condition \eqref{eq:H0:compatibility} reduces to the statement that $\mathcal{P}(\omega)$ is independent of $\omega$.

The statement of Theorem~\ref{thm:ext:scattering} avoids the compatibility condition altogether by not prescribing the function $M_1^{ext}(\omega)$. In this way the exterior homogeneous solution is determined by the radiation field $\mathcal{H}_0$.
In a similar manner the compatibility conditions \eqref{eq:thm:n:compatibility} and \eqref{eq:thm:m:compatibility} can be removed by allowing non-trivial exterior asymptotics, as in Theorem~\ref{thm:hom:N} and Corollary~\ref{thm:hom:M}, determined by   the radiation field $\mathcal{F}_0$.

\subsection{Scattering problems for semilinear equations.}
\label{sec:motiv}

The source terms discussed in Sections~\ref{sec:intro:quadratic} and~\ref{sec:cubic:intro}
are motivated by the scattering problem for  wave equations satisfying the null or  weak null condition.

 Consider, for example, the problem of constructing solutions backwards in time from infinity for the system
\begin{equation}
  \label{eq:weak:2}
  \Box\, \phi = (\partial_t\psi)^2\qquad \Box\, \psi=0\,.
\end{equation}
Even in the simplest case, for solutions $\psi$ 
with decaying radiation field $\mathcal{F}_0$, when
\[\psi(t,x)\sim \mathcal{F}_0(r-t,\omega)/r+\mathcal{F}_1(r-t,\omega)/r^2\] in the wave zone $t\sim r$, we see after inserting in the first equation in \eqref{eq:weak:2},
that $\psi$ obeys
\begin{equation}
   \Box\phi\sim\frac{(\mathcal{F}_0')^2}{r^2}+\frac{2\mathcal{F}_0'\mathcal{F}_1'}{r^3} \qquad (t\sim r)\,.
\end{equation}
This is precisely the motivation for studying the equations \eqref{eq:sourcbegining} and \eqref{eq:sourcbegining:cubic}.
Here the sources $n(q,\omega)\!\sim\! (\mathcal{F}_0')^2(q,\omega)$ and $m(q,\omega)\!\sim\!2\mathcal{F}_0'(q,\omega)\mathcal{F}_1'(q,\omega)$  are decaying  in $q\!=\!r\shortminus t$,
namely for $|\mathcal{F}_0|\les 1/\langle q\rangle^{\gamma}$:
\begin{equation}
|n(q,\omega)|\les \langle q\rangle^{-2-2\gamma}\,,\qquad |m(q,\omega)|\les \langle q\rangle^{-1-2\gamma}\,.
\end{equation}

\begin{remark}
The system \eqref{eq:weak:2} is  the simplest example of a system of wave equations satisfying the weak null condition, see \cite{LR03}, and is itself a model of the Einstein equations in harmonic coordinates \cite{LR05,LR10,L17}. We refer the reader to the introduction of \cite{LS20}, for a discussion of the asymptotics established in \cite{L17}, and the relevant \emph{semi-linear} model problems in this setting.
\end{remark}

\subsubsection{Scattering with weak null condition in Sobolev regularity.}

The ideas in this paper can in principle be combined with our approach in \cite{LS20} to obtain  scattering results in $L^2$ spaces for non-linear systems satisfying the null condition or weak null condition.
We illustrate this for \eqref{eq:weak:2} providing an alternative construction to \cite[Section 7]{LS20}.

In \cite[Section 4]{LS20} we have solved the linear equation $\Box\psi=0$ from infinity with radiation fields $\mathcal{G}_0$ in suitable weighted Sobolev spaces on $\mathbb{R}\times\mathbb{S}^2$,
satisfying $\|\mathcal{G}_0\|_{N,\gamma-1/2}<\infty$  for some $1/2<\gamma<1$; see \cite[Theorem 1.1]{LS20} for the definition of the norms, and precise statements.
This is the setting of decaying radiation fields, without homogeneous solutions in the interior and exterior.

In the proof we used an asymptotic solution $\Psi_{asym}$ of the form
  \begin{equation}
    \Psi_{asym}=\Psi_{rad}\: \chi\bigl(\tfrac{\langle t-r\rangle}{r}\bigr)\,,\qquad \Psi_{rad}=\frac{\mathcal{G}_0(r-t,\omega)}{r}+\frac{\mathcal{G}_1(r-t,\omega)}{r^2}\,,
  \end{equation}
  and  obtained  a scattering solution $\psi=\Psi_{asym}+\psi_{rem}$ with trivial data  for the remainder $\psi_{rem}$ as $t\to\infty$.
  The remainder then satisfies the estimates of \cite[Proposition 4.8]{LS20}.
  In particular, $\psi_{rem}$ does not have a radiation field,
  \begin{equation}
     |\psi_{rem}|\lesssim \frac{\|\mathcal{G}_0\|_{5,\gamma-1/2}}{\langle t+r\rangle \langle t\rangle^{\gamma-\gamma'}\langle t-r\rangle^{\gamma'}}\qquad (\gamma'<\gamma)
  \end{equation}
  and hence $\psi(t,x)\sim \mathcal{G}_0(r-t,\omega)/r$ in the wave zone when $r\sim t$.

Inserting this solution in the equation for $\phi$ in \eqref{eq:weak:2} we obtain
  \begin{equation}
    \begin{split}
      \Box\,\phi=& (\partial_t\psi)^2=(\partial_t\Psi_{asym})^2+2\partial_t\Psi_{rad}\: \partial_t \psi_{rem}+(\partial_t\psi_{rem})^2\\
      =& (\partial_t\Psi_{rad}\:\chi)^2+(\partial_t\Psi_{asym})^2-(\partial_t\Psi_{rad}\:\chi)^2+2\partial_t\Psi_{rad}\: \partial_t \psi_{rem}+(\partial_t\psi_{rem})^2\,.
    \end{split}
  \end{equation}

    The main term we have added and subtracted on the RHS contains the relevant source terms,
  \begin{equation}
    (\partial_t\Psi_{rad}\:\chi)^2=\frac{n}{r^2}\chi+\frac{m}{r^3}\chi+\frac{l}{r^4}\chi
 \end{equation}
 for which we  construct an asymptotic solution in \emph{this paper}. Here
 \begin{equation}
   n(q,\omega)=(\mathcal{F}_0')^2(q,\omega)\,, \quad m(q,\omega)=2\mathcal{F}_0'(q,\omega)\mathcal{F}_1'(q,\omega)\,,\quad l(q,\omega)=(\mathcal{F}_1')^2(q,\omega)\,,
 \end{equation}
and so in particular $n(q,\omega)\les \langle q\rangle^{-2-2\gamma}$.
 For such quadratic sources $n$ and cubic sources $m$ we construct an asymptotic solution $\Phi_{asym}$ so that
\begin{equation}
  \Box \Phi_{asym}+\frac{n(r-t,\omega)}{r^2}\chi+\frac{m(r-t,\omega)}{r^3}\chi=R_{asym}
\end{equation}
with a free radiation field $\mathcal{H}_0$  as discussed in Theorem~\ref{thm:n} and Theorem~\ref{thm:m}.
This is the main difference from \cite[Section 7]{LS20}, where an auxiliary forward solution was invoked to deal with these terms.

Now with this choice of the asymptotic solution $\Phi_{asym}$, the remainder $\phi_{rem}=\Phi_{asym}+\phi$ satisfies
\begin{equation}
  \Box\phi_{rem}=R_{asym}+\frac{l}{r^4}\chi+(\partial_t\Psi_{asym})^2-(\partial_t\Psi_{rad}\:\chi)^2+2\partial_t\Psi_{rad}\: \partial_t \psi_{rem}+(\partial_t\psi_{rem})^2\,.
\end{equation}
We have already described in \cite[Section 7.1]{LS20} how the results in \cite{LS20} are applied to estimate most of these error  terms in $L^2$.
We have seen that the bilinear estimate of \cite[Lemma 6.4]{LS20} gives
 \begin{equation}
   \|(\partial_t\Psi_{asym})^2(t,\cdot)-(\partial_t\Psi_{rad}\:\chi)^2(t,\cdot)\|_{L^2_x}\lesssim \frac{\|\mathcal{G}_0\|_{8,\gamma-1/2}^2}{\langle t\rangle^{3/2+\gamma}}\,.
 \end{equation}
Furthermore using \cite[Lemma 6.9]{LS20} and \cite[Proposition 4.6]{LS20} we have
 \begin{equation}
   \| \partial_t\Psi_{rad}\: \partial_t \psi_{rem}(t,\cdot)+(\partial_t\psi_{rem})^2(t,\cdot) \|_{L^2_x} \lesssim \frac{\|\mathcal{G}_0\|_{8,\gamma-1/2}^2}{\langle t\rangle^{3/2+\gamma}}\,,
 \end{equation}
and suitable higher order versions of these estimates.

 It remains to estimate $R_{asym}$  in $L^2$.
 The asymptotic error term $R_{asym}$ is given -- for an asymptotic solution with homogeneous interior solution --  by \eqref{eq:R:asym:int},
 and estimated in this paper in $L^\infty$; see Appendix~\ref{sec:RadialEstimates}.
 In principle, it would also be possible to derive an estimate in $L^2$, such as
 \begin{equation}
    \| R_{asym}(t,\cdot) \|_{L^2_x}\les \frac{D}{\langle t\rangle^{2-\epsilon}} \qquad (\epsilon>0)
  \end{equation}
  (and suitable higher order versions)
  where $D$ depends on the weighted Sobolev norms of $\mathcal{G}_0$ and $\mathcal{H}_0$. However, we do not pursue this here because it is not the emphasis of this paper.

  A standard energy estimate as in \cite[Section 4.2]{LS20} then shows that with trivial data as $t\to\infty$, the remainder $\phi_{rem}$ does not have a radiation field, while $\phi$ has a radiation field as presented in Theorems~\ref{thm:n}, \ref{thm:m}.
  In this way it possible to prove $L^2$ versions of the theorems presented in this paper, with the free radiation fields in weighted Sobolev spaces.

\subsubsection{Null condition.}

For nonlinear wave equations with classical null condition the interior solution corresponding to cubic source terms discussed in Section~\ref{sec:cubic:intro} is relevant. While for the equation
\begin{equation}
  \label{eq:108}
  \Box\phi=Q_0(\partial \phi,\partial \phi)=-(\partial_t\phi)^2+|\nabla\phi|^2
\end{equation}
we find that \eqref{eq:wavezone:cubic:intro} produces a cubic source term $m$, it does not lead to a homogeneous solution in the interior:
\begin{equation}
  m(q,\omega)=2\mathcal{G}_0(q,\omega)\mathcal{G}_0'(q,\omega)\,,\qquad M(\omega)=\int_{-\infty}^{\infty}\!\!\!\!m(q,\omega)\ud q=\mathcal{G}_0^2\rvert^\infty_{-\infty}=0\,.
\end{equation}
However, more generally for systems, there is a non-vanishing homogeneous solution in the interior:
\begin{equation}
  \Box\phi=Q(\partial \phi,\partial \psi)\,,\qquad \Box\psi=Q(\partial \phi,\partial \psi)\,.
\end{equation}
Indeed, with $\phi\sim\mathcal{F}_0(r-t,\omega)/r$, and $\psi\sim\mathcal{G}_0(r-t,\omega)/r$ we find
\begin{equation}
  Q_0(\partial \phi,\partial \psi)\sim -\frac{\mathcal{F}_0'\mathcal{G}_0+\mathcal{F}_0\mathcal{G}_0'}{r^3} \qquad (r\sim t,r\to\infty)
\end{equation}
and similarly for $Q_{\alpha\beta}=\partial_\alpha\phi\,\partial_\beta\psi-\partial_\beta\phi\,\partial_\alpha\psi$ $(\alpha\neq \beta)$. These cubic source terms lead to interior homogeneous asymptotics of the form \eqref{eq:Phi3:intro},
and explain in particular the $1/q$ tail in the radiation field, cf.~\eqref{eq:G0:tail},
which were observed in \cite{C86,K86}.

\subsection{Further comments on related problems.}

\subsubsection{Asymptotic expansions.}

The asymptotic expansions of the wave equation that are so central to this paper, namely the expansions in $1/r$ and $t-r$ in the wave zone, as well as the expansion in $(r-t)/r$ and $1/r$ in the interior,
feature also in the microlocal  analysis of wave operators \cite{BVW15,BVW18,HV17,HV23,H23a}.
In  \cite{BVW15,BVW18}, Baskin, Vasy and Wunsch established, for a class of spacetimes, asymptotic expansions of solutions to the wave equation in all asymptotic regimes, that includes in particular  an asymptotic expansion of the radiation field. Homogeneity of degree $-2$ under scalings in the interior plays a role here as well \cite{BVW15}. Polyhomogeneous expansions have also been established for Einstein's equations, in the context of the stability of Minkowski space, by Hintz and Vasy in \cite{HV17}. More recently,  a microlocal approach for the analysis near null infinity has been further developed by Hintz and Vasy \cite{HV23}.
In these settings, the existence of expansions is related to a notion regularity (``edge-b regularity'') in the radial compactification of the space-time, and the above variables appear in the blow-up of null infinity in this picture; see \cite{HV23} (Figure~1.2).
Moreover, Hintz applied this framework to obtain asymptotics of wave equations with inverse square potentials, on non-stationary spacetimes in \cite{H23a}; see also \cite{BGRM22}.

The interplay between the precise form of the polyhomogeneous expansion near null infinity, and late time tails of solutions to the wave equation  towards time-like infinity, have also been studied in the context of Price's law on black hole spacetimes \cite{AAG23,H20,K22b}. The physical space methods developed in this context \cite{DR09,S13,AAG18a}, have also led Gajic to derive late-time asymptotics for geometric wave equations with inverse square potentials \cite{G22}. Keir used  $r^p$-weighted estimates to show global existence for wave equations satisfying the weak null condition \cite{K18}.
Further results on asymptotics that are relevant to this paper include \cite{CKL19,K17,L17,H21}.

\subsubsection{Scattering.}
\label{sec:further}

In \cite{LS05a}, Lindblad and Soffer solved the scattering problem for a non-linear Klein-Gordon equation in one dimension, and showed the existence of solutions with prescribed asymptotics at infinity, of a form previously derived by Delort \cite{D01}; see also \cite{LS05b,LS06b,LLS20}.

In \cite{W14} Wang studied the scattering problem for Einstein's equations in space dimensions $4$ or higher.

In \cite{LS20} we approached the scattering problem for a class of semi-linear wave equations and showed the existence of solutions from scattering data at infinity, which was assumed to decay along null infinity. We showed that approximate solutions can be constructed for non-linear wave equations satisfying the null condition, and used a version of the fractional Morawetz estimate, see \cite{LS06b}, to control the remainder for the backwards problem, and showed decay of the solution at a rate corresponding to the decay of the radiation field. In the talk \cite{H21talk} we announced some of our results on scattering with homogeneous asymptotics in the interior.

Since then this approach has been applied and extended in several directions:
Lili He proved a scattering result for the Maxwell Klein Gordon-equations \cite{H21}.
Solutions to this system have  slower decay in the interior than those treated in \cite{LS20} and to prove existence from infinity it is necessary to subtract the leading order homogeneous interior asymptotics. In fact, a suitable model for the weak null structure of this system is, as discussed in \cite{H21},
\begin{equation}
  \label{eq:MKG}
  \Box \phi = i \psi \partial_t\phi\,,\qquad \Box \psi =0\,,\qquad \Box \underline{\psi}=\textrm{Im}\bigl(\phi\overline{\partial_t\phi}\bigr)\,
\end{equation}
where $\phi$ is complex valued, and $(\psi,\underline{\psi})$ should be thought of as components of the electromagnetic potential.
This system is also instructive to illustrate how our  results in this paper can be applied:
In the presence of charges $Q$, $\psi$ is has exterior homogeneous asymptotics and can be construced as in Theorem~\ref{thm:ext:scattering}, leading to a solution with a radiation field that does not decay, $\psi\sim Q/r$. The scalar field has asymptotics $\phi\sim \exp[-iQ\ln(r)]\Phi_0(r-t,\omega)/r$, and thus produces a quadratic source term in the equation for $\underline{\psi}$, leading to a scattering problem with interior homogeneous asymptotics as discussed in Section~\ref{sec:intro:quadratic},
\begin{equation}
  \label{eq:60}
  \Box\underline{\psi}\sim n/r^2\,,\qquad n(q,\omega)=\textrm{Im}(\Phi_0(q,\omega)\overline{\partial_q\Phi_0(q,\omega)})\,.
\end{equation}

Dongxiao Yu proved a scattering result for a class of quasi-linear wave equations  \cite{Y20,Y22}, with compactly supported scattering data,  by adapting the construction of the approximate solution to fit the true characteristics of the equation, asymptotically. (See \cite{L92,L08} for the  forward in time  results.)  A similar correction, using the asymptotic solutions to the eikonal equation, is also relevant for Einstein's equations in \cite{L17}.

In \cite{LC23}, Chen and Lindblad  obtained scattering results for  coupled wave Klein-Gordon systems:
\begin{equation}\label{semilinearWLKGsystem}
-\Box u=(\pa_t \phi)^2+\phi^2,\qquad -\Box \phi+\phi=u\phi,
\end{equation}
in a setting where the interior asymptotics of the Klein-Gordon field affects the asymptotics for the wave
equation in the interior and the asymptotics of the wave equation cause a logarithmic correction
to the phase of the Klein-Gordon field. With $\rho=\sqrt{t^2-|x|^2}$ and $y=x/t$  it was proven that
\beq
u\sim U(y)/\rho,\quad t/r>1,\quad\text{and}\quad  u\sim \mathcal{F}_0(t-r,\omega)/r,\quad t\sim r,\
\qquad \text{as}\quad t\to\infty,
\eq
and
\beq
 \phi\sim\rho^{-\frac 32}\bigtwo(e^{i\rho-\frac i2 U(y)\ln \rho} a_+(y)+e^{-i\rho+\frac i2 U(y)\ln \rho} a_-(y)\bigtwo),
\eq
where $a_{\pm}(y)$ decay as $|y|\to 1$, and
\begin{equation}
    -\Box(\frac{U(y)}\rho)=2\rho^{-3} (1+(1-|y|^2)^{-1}) a_+(y) a_-(y).
\end{equation}
Note that the correction to the asymptotics of $u$ in the interior is exactly
a homogeneous function of degree minus one as for the wave equation with
an inhomogeneous term satisfying the weak null condition.
Although the right hand side of the first equation in \eqref{eq:weak:2} only decay like $r^{-2}$ it is focused close to the light cone $r\sim t$, whereas
the right hand side of the first equation in \eqref{semilinearWLKGsystem} decay like $r^{-3}$ but is spread out everywhere in $r<t$. From a global perspective the interior does not see the fine structure along the light cone, both right hand sides are homogeneous of degree minus three, cf.~\eqref{eq:phiinfinityformula}, and so it makes sense that the solution of the wave equation is homogeneous of degree minus one asymptotically.

A microlocal approach to the scattering problem, for an equation with asymptotics towards time-like infinity, is given for the  Schr\"odinger equation in~\cite{GRGH22}.
Finally, for the wave equation on black hole exteriors a scattering theory is established in  \cite{DRS18}.

\subsection{Structure of the paper}
In Section~\ref{sec:hom:int} we introduce the homogeneous solutions that capture the interior asymptotics of \eqref{eq:sourcbegining} and \eqref{eq:sourcbegining:cubic}. In fact, they arise in the homogeneous scaling limit of these equations, namely as distributional solutions of wave equations with sources supported on the lightcone.
In Section~\ref{sec:story} we begin to reverse engineer the solution to \eqref{eq:sourcbegining} and start with deriving the asymptotics near the lightcone. We discuss the differential equations and matching conditions that have to be satisfied by the radiation fields in order to obtain an approximate solution that matches with a compatible homogeneous solution in the interior.
In Section~\ref{sec:interiorexpansion} we derive the expansion of the homogeneous solutions near the lightcone, and the functions that enter the matching conditions. We also derive the equations that relate the coefficients in this expansion near the light cone.
In Section~\ref{sec:matching} we carry out the matching procedure of the expansion at null infinity to the interior asymptotics determined by the source, and derive in the process the compatibility conditions for the scattering data.  In particular, we construct the asymptotic solution  $\psi_{asym}$ and  solve \eqref{eq:remsol}  for the faster decaying remainder from infinity, in $L^\infty$ norms.
In Section~\ref{sec:hom:ext} we turn to the discussion of the homogeneous solutions in the exterior. Finally in Section~\ref{sec:ext:homogeneous} we match the exterior homogeneous asymptotics to an expansion in the wave zone, and thus produce scattering solutions for the homogeneous wave equation with with radiation fields that do not decay.

Appendix~\ref{app:int} is for the interior solutions, and contains the calculation of the spherical integrals needed to get an expansion of the homogeneous solutions, and also give the radial $L^\infty$  estimates needed to control the remainder.
Appendix~\ref{app:ext} is for the exterior problem, and contains the calculation of the homogeneous solutions in the exterior, and a discussion of the invertible transformations that appear in this context.

\paragraph{Acknowledgements.} We would like to thank Mike Eastwood, Lionel Mason, Gabriel Paternain and Richard Melrose for helpful discussions. We would also like to thank the Mittag-Leffler Institute for their hospitality in the Fall Semester 2019. H.L. was supported in part by Simons Collaboration Grant 638955.

\section{Interior homogeneous solutions to the homogeneous wave equation}
\label{sec:hom:int}

In this section we introduce the homogeneous solutions  that capture the leading order homogeneous asymptotics in the interior of the light cone $t>|x|$ for wave equations with sources on the light cone.

\subsection{Asymptotics with sources on light cones: quadratic terms.}\label{sec:interior}

The interior asymptotics of forward solutions to the wave equation with sources on light cones were studied in Lindblad \cite{L90a}. We begin with a derivation of the interior asymptotics for the equation
\begin{equation}
  \label{eq:box:n:2}
-\Box \,\phi=n(r\!-\!t,\omega)/r^2,
\end{equation}
where $n(q,\omega)$ is concentrated close to the light cone $q\sim 0$.
More precisely, we derive an expansion in homogeneous functions to second order in the interior of the lightcone.

\subsubsection{Highest order interior asymptotics.}

The forward problem for this equation has homogeneous asymptotics
\begin{equation}\label{eq:phiasymptotic}
\phi(t,x)\sim \phi_{1,\infty}(t,x),\qquad t>|x|,\quad t\to \infty,\quad \text{where}\quad  \phi_{1,\infty}(at,ax)= \phi_{1,\infty}(t,x)/a.
\end{equation}
In fact, the following scaling argument shows that, in the interior $t/r>1$,  $\phi$ is asymptotic to a solution $\phi_{1,\infty}$ of a wave equation with a homogeneous source term on the light cone:
\begin{equation}
\label{eq:phiinfinityformula}
-\Box \,\phi_{1,\infty}=N(\omega)\,\delta(r\!-\!t)/r^2.
\end{equation}
To see this, note that if $\phi$ solves \eqref{eq:box:n:2} then the rescaled solution $\phi_{1,a}(t,x)=a\,\phi(at,ax)$ satisfies
\begin{equation*}
-\Box \,\phi_{1,a}=n_a(r\!-\!t,\omega)/r^2, \qquad n_a(q,\omega)=a\, n(aq,\omega).
\end{equation*}
As $a\to\infty$, in the sense of distribution theory\footnote{This is the statement that for $\varphi\in\mathcal{S}(\mathbb{R})$, $\lim_{a\to\infty}\int_{-\infty}^{\infty} n_a(q,\omega) \varphi(q)\ud q=\varphi(0)N(\omega)$, which can be verified in an elementary way.}

\begin{equation*}
n_a(q,\omega)=a\, n(aq,\omega)\to \delta(q)N(\omega),\qquad\text{where}\quad N(\omega)= \int_{-\infty}^{+\infty}\!\!\! n(q,\omega) \, dq,
\end{equation*}
and $ \delta(q)$ is the delta function, provided that $|n(q,\omega)|\lesssim \langle q\rangle^{-1-\epsilon}$, $\epsilon>0$.

Since $\phi_{1,a}=E_+\ast F_a$ for $t>r$,
where $E_+$ is the forward fundamental solution of $-\Box$ and $F_a=n_a(r-t,\omega)/r^2$, and $F_a\to \mu_1$ in the sense of distributions as $a\to \infty$\footnote{This is seen by multiplying by a test function  and changing variables $t=r\!+\!q$, $x=r\omega$ in the resulting integral. In fact the distribution $\delta(r\shortminus t)$ is defined as a composition of the delta function with this change of variables.},
where $\mu_1=N(\omega) \delta(r-t)/r^2$,
it follows that when $t>r$, $\phi_{1,a}=E_+\ast F_a\to \phi_{1,\infty}=E_+\ast \mu_1$ as $a\to \infty$.
This is the statement \eqref{eq:phiinfinityformula}.
Since $\mu_1$ is a homogeneous distribution of  degree $\!-3$, $\phi_{1,\infty}$ is homogeneous of degree $\!-1$.

\begin{remark} By the statement \eqref{eq:phiasymptotic} that $\phi(t,x)\sim\phi_{1,\infty}(t,x)$ we mean that
\beq
a\phi(at,ax)\to \phi_{1,\infty}(t,x), \qquad \text{as}\quad a\to \infty
\eq
in the sense of distributions in the set $t>r$. This does not directly imply a pointwise rate of convergence with estimates, but it gives, in a very simple way, what the forward asymptotics have to be and this is what we need to know to construct solutions to the backward scattering problem  from infinity.
\end{remark}

\subsubsection{Second order interior asymptotics.}\label{sec:interior:lower}
There are homogeneous forward asymptotics   in the interior also to next order if $n$ decays sufficiently fast away from the light cone.
We show by a scaling argument that
\begin{equation*}
\phi_{\,2}(t,x)=\phi(t,x)-\phi_{1,\infty}(t,x)\sim \phi_{\,2,\infty}(t,x),\qquad t>|x|,\quad t\to \infty,\quad \text{where}\quad  \phi_{\,2,\infty}(at,ax)= \phi_{\,2,\infty}(t,x)/a^2.
\end{equation*}
In fact, if $\phi$ solves \eqref{eq:box:n:2} then $\phi_{\,2,a}(t,x)=a^2\,\phi_{\,2}(at,ax)$ satisfies
\begin{equation*}
-\Box \,\phi_{\,2,a}=m_a(r\!-\!t,\omega)/r^2, \qquad m_a(q,\omega)=a\big(a\, n(aq,\omega)-\delta(q)N(\omega)\big).
\end{equation*}
Moreover as $a\to\infty$, in the sense of distribution theory
\begin{equation*}
m_a(q,\omega)\to -\delta^{\,\prime\!}(q)M(\omega),\qquad\text{where}\quad M(\omega)= \int_{-\infty}^{+\infty}\!\!\! q\,  n(q,\omega) \, dq,
\end{equation*}
and $ \delta^{\,\prime\!}(q)$ is the distributional derivative of the delta function,
 provided that we have additional decay $|n(q,\omega)|\lesssim \langle q\rangle^{-2-\epsilon}$, $\epsilon>0$.
Hence
$\phi_{\,2,a}\to \phi_{\,2,\infty}$ as $a\to\infty$ where
\begin{equation}\label{eq:phi2infinityformula}
-\Box \,\phi_{\,2,\infty}=-M(\omega)\,\delta^{\,\prime}(r\!-\!t)/r^2\,.
\end{equation}
Since this distribution is homogeneous of degree $\!-4$, $\phi_{\,2,\infty}$ is homogeneous of degree $\!-2$.

\subsubsection{Asymptotics towards the light cone for the leading order homogeneous solution.}\label{sec:hom:asymptotics:leading}
We claim that $\phi_{1,\infty}$ has an asymptotic expansion as we approach the light cone.
To see this we use a formula from  \cite{L90a} for the solution of \eqref{eq:phiinfinityformula} obtained by convolving with the forward fundamental solution 
of $-\Box$:
\begin{equation}\label{eq:wavesourceconeformulasec2.3}
 \phi_{1,\infty}(t,r\omega)=\Psi_1[N](t,r\omega)=\frac{1}{\!4\pi }\!\int_{\mathbb{S}^2}\!\!
  \frac{N(\sigma{}_{\!})\, dS(\sigma) }{t\!-\!\langle\sigma,r\omega{}_{\!}\rangle}.
\end{equation}
Writing
\begin{equation}
r\Psi_1[N](t,r\omega)=\frac{1}{\!4\pi }\!\int_{\mathbb{S}^2}\!\!
  \frac{N(\sigma{}_{\!})\, dS(\sigma) }{t/r\!-\!\langle\sigma,\omega{}_{\!}\rangle}
  =\frac{1}{\!4\pi }\!\int_{\mathbb{S}^2}\!\!
  \frac{N(\omega{}_{\!})\, dS(\sigma) }{t/r\!-\!\langle\sigma,\omega{}_{\!}\rangle}
  +\frac{1}{\!4\pi }\!\int_{\mathbb{S}^2}\!\!
  \frac{\big(N{}_{\!}(\sigma{}_{\!})\!-\!N{}_{\!}(\omega{}_{\!})\big)\,\ud S(\sigma{}_{\!})\!}{t/r\!-\!\langle\sigma,\omega{}_{\!}\rangle},
\end{equation}
the first integral can be explicitly calculated,
and the second converges as $t/r\to 1$ since the singularity is cancelled.
In fact, by a rotation of the cartesian coordinates $(\sigma_1,\sigma_2,\sigma_3)$,  we can align the $\sigma_1$-axis with $\omega$ so that $\omega=(1,0,0)$.
  Now introducing spherical coordinates $(\sigma_1=\cos\theta,\sigma_2=\sin\theta\cos\phi,\sigma_3=\sin\theta\sin\phi)$ we see that the volume form on $\mathbb{S}^2$ can be expressed as:
  $\ud S(\sigma)=\sin\theta\ud\theta\ud\phi=\ud\phi\ud \sigma_1$.
  Hence we compute for the first term,
\beq
\frac{1}{\!4\pi }\!\int_{\mathbb{S}^2}
  \frac{\, dS(\sigma) }{t/r\!-\!\langle\sigma,\omega{}_{\!}\rangle}
  =\frac{1}{2 }\!\int_{-1}^{1}
  \frac{\, d\sigma_1 }{t/r\!-\!\sigma_1}=\frac{1}{2}\ln{\Big|\frac{t/r+1}{t/r-1}\Big|}\,,
\eq
and we obtain to leading order as $t/r\to 1$,
\begin{equation}\label{eq:lowestinteriorasymptotics}
 r\Psi_1[N](t,r\omega)
 \sim \frac{N(\omega)}{2}  \ln\Big|\frac{t+r}{t\!-\!r}\Big|
 +\frac{1}{\!4\pi }\!\int_{\mathbb{S}^2}\!\!
  \frac{\big(N{}_{\!}(\sigma{}_{\!})\!-\!N{}_{\!}(\omega{}_{\!})\big)\,\ud S(\sigma{}_{\!})\!}{1\!-\!\langle\sigma,\omega{}_{\!}\rangle}\,,
\end{equation}
where the second term is bounded; see the discussion of spherical integrals in Appendix~\ref{app:nonsingular}.
Thus for future reference we note (using the notation of Section~\ref{sec:homoexpansion})
that the leading order coefficients in the expansion of the homogeneous degree $-1$ solution are:
\begin{equation}
  \label{eq:N:01:0}
  N_{01}(\omega)=\frac{1}{2}N(\omega)\,,\qquad N_0(\omega)=\frac{1}{4\pi}\int_{\mathbb{S}^2} \frac{N(\sigma)-N(\omega)}{1-\langle \sigma,\omega\rangle} \ud S(\sigma)\,.
\end{equation}
These coeffcients will be used in Section~\ref{sec:logcorrection}, for the matching of the asymptotics along null infinity, see \eqref{eq:approximatesource2},  to the interior asymptotics \eqref{eq:lowestinteriorasymptotics}.

\subsubsection{Asymptotics towards the light cone for the second order homogeneous solution.}
\label{sec:hom:asymptotics:second}

We proceed similarly for the second order homogeneous solution, using that the solution to \eqref{eq:phi2infinityformula} can be obtained by taking the time derivative of \eqref{eq:phiinfinityformula} and replacing $N(\omega)$ by $M(\omega)$:
\begin{equation} \label{eq:Psi:2}
 \phi_{\,2,\infty}(t,r\omega)=\Psi_2[M](t,r\omega)=-\frac{1}{\!4\pi }\!\int_{\mathbb{S}^2}\!\!
  \frac{M(\sigma{}_{\!})\, dS(\sigma) }{\big(t\!-\!\langle\sigma,r\omega{}_{\!}\rangle\big)^2}.
\end{equation}
Writing the Taylor expansion for $M(\sigma)$ at $\sigma=\omega$ up to second order,
\begin{multline}
r^2\Psi_2[M](t,r\omega)
=-\frac{1}{\!4\pi }\!\int_{\mathbb{S}^2}\!\!
  \frac{M(\omega{}_{\!})\, dS(\sigma) }{\big(t/r\!-\!\langle\sigma,\omega{}_{\!}\rangle\big)^2}
-\frac{1}{\!4\pi }{\frac{1}{4}}\!\int_{\mathbb{S}^2}\!\!
  \frac{\big(1\shortminus \langle\sigma,\omega\rangle^2\big)\triangle_\omega M(\omega{}_{\!})\!\! }{\big(t/r\!-\!\langle\sigma,\omega{}_{\!}\rangle\big)^2}\, dS(\sigma)\\
  -\frac{1}{\!4\pi }\!\int_{\mathbb{S}^2}\!\!
  \frac{M(\sigma{}_{\!})-M(\omega{}_{\!})-\frac{1}{4}\big(1\shortminus \langle\sigma,\omega\rangle^2\big)\triangle_\omega M(\omega{}_{\!}) \!\!}{\big(t/r\!-\!\langle\sigma,\omega{}_{\!}\rangle\big)^2}\, dS(\sigma)  ,
\end{multline}
the first two integrals can be explicitly calculated,
and the last converges as $t/r\to 1$,  because the remainder cancels the singularity, see Section \ref{sec:SphericalTaylor}.
In fact,
\begin{equation}\label{eq:lowestinteriorasymptotics2}
r^2\Psi_2[M](t,r\omega)
\sim -\frac{M(\omega_{{}_{\!}})}{\!\!(t/r)^2\!-\!1\!\!}-
 \frac{\!\triangle_\omega M(\omega)\!}{4} \Big(\frac{t}{r} \ln\Big|\frac{t/r\!+\!1}{t/r\!-\!1}\Big|-2\Big)+M_1^\ast(\omega),
\end{equation}
 where
  \begin{equation}
    \label{eq:M1:ast}
    M_1^\ast(\omega)=-\frac{1}{\!4\pi }\!\int_{\mathbb{S}^2}\!\!
  \frac{M(\sigma_{{}_{\!}})\shortminus M(\omega_{{}_{\!}})\shortminus\frac{1}{4}\big(1\shortminus \langle\sigma,\omega\rangle^2\big)\triangle_\omega M(\omega_{{}_{\!}}) \!\!}{\big(1\!-\!\langle\sigma,\omega{}_{\!}\rangle\big)^2}\, dS(_{{}_{\!}}\sigma_{{}_{\!}})\,,
\end{equation}
is finite, $|M_1^\ast(\omega)|\lesssim \|M\|_{\mathrm{C}^4(\mathbb{S}^4)}$. 

 These terms are again relevant for the matching of the second order asymptotics along null infinity, to the homogeneous degree $-2$ asymptotics in the interior,  that we discuss in Section~\ref{sec:expansion:second:order:interior}.

 \subsection{Asymptotics with sources on light cones: cubic terms.}
 \label{sec:asym:cubic}

 We may derive in a similar manner homogeneous interior asymptotics for the equation
 \begin{equation} \label{eq:highersourcbegining}
-\Box \,\phi=m(r\!-t,\omega)/r^3\,.
\end{equation}

If we repeat the scaling argument in Section~\ref{sec:interior:lower} we  see  that the forward solution to \eqref{eq:highersourcbegining}
satisfies
$a^2\phi(at,a x)\to \Psi^3[M]$, as $t\to\infty$ if $t>r$, where, using the formulas in Appendix~\ref{sec:sourseformulas},
\begin{equation}
\Psi^3[M](t,r\omega)= \frac{1}{2\pi}\int_{\bold{S}^2}{\frac{
 M({\sigma}) \ud S({\sigma})}{(t+r)(t-r)}\,
},\qquad  M({\sigma})=\int_{-\infty}^{+\infty} m(q,\sigma)\, \ud q.
\end{equation}
In fact,  given  a solution $\phi$ to  \eqref{eq:highersourcbegining}
then $\phi_{a}(t,x)=a^2\phi(at,ax)$ satisfies
  \begin{equation*}
    -\Box\phi_{a}=m_a(r-t,\omega)/r^3\,,\qquad m_a(q,\omega)=am(aq,\omega).
  \end{equation*}
  As $a\!\to\! \infty$, $\int \! m_a({}_{\!}q,\omega{}_{\!}) \psi({}_{\!}q{}_{\!}) dq\!\to\! \psi_q(0)$, i.e.~in the sense of distribution theory $m_a\!\to\! \delta(q)M$, provided that $|m(q,\omega)|\!\lesssim \!\langle q\rangle^{-1-\epsilon}$, $\epsilon\!>\!0$. Hence $\phi_{a}\!\to\!\phi_{\infty}$ and  this solves
  \begin{equation*}
    -\Box \phi_{\infty}=M(\omega)\delta(q)/r^3\,.
  \end{equation*}

  \begin{remark}
  Here the right hand side does not define a distribution at $(0,0)$
  and this formula should be interpreted as the statement that the convolution of the right hand side with the forward fundamental solution gives $\Psi^3$ for $t>r$.
Hence
\beq
{\Phi_{\chi}}^{\!\!\!3}[m](t,r\omega)\sim \Psi^3[M](t,r\omega),\qquad t/r>1\,,
\eq
and as $r/t\to 1$,
\beq\label{eq:k3homsolformulaasym}
\Psi^3[M](t,r\omega)\sim M_0\Big( \frac{1
}{r(t-r)}-\frac{1}{2r^2}+\frac{t-r}{4 r^3}\Big),\,\qquad M_0=\frac{1}{4\pi}\int_{\bold{S}^2} M({\sigma}) dS({\sigma})\,.
\eq
\end{remark}

 \section{Radiation fields at null infinity for the wave equation with a source at the light cone}
\label{sec:story}

In this section we discuss the asymptotics in the wave zone of solutions to wave equations with sources. We discuss the expansions at null infinity that will be used for the scattering solutions in this paper, and derive several necessary conditions for the construction of asymptotic solutions from prescribed radiation fields.

\subsection{Existence of radiation field at null infinity for a solution to the homogeneous equation.}
\label{sec:radiationfield}
 A solution of a linear homogeneous wave equation $\Box\,\psi=0$
 with smooth  initial data decaying like $r^{-1}$,
 decays like $t^{-1}$ and has a radiation field\footnote{The discussion in this section is introductory, and based on integration along characteristics arguments outlined below. A precise characterisation of the relationship between slowly decaying data, namely data that contains $r^{-1}$ or $r^{-2}$ tails, and slow decay in time, namely homogenoeus solutions of degree $-1$ and $-2$ corresponding to $t^{-1}$ or $t^{-2}$ tails, is a subject of this paper, and addressed in Section~\ref{sec:ext:intro}; see in particular Theorem~\ref{thm:ext:scattering}.}
\begin{equation}\label{eq:radiationfield}
\psi(t,x)\sim \mathcal{F}(r-t,\omega)/r,
\qquad\text{where}\quad | \mathcal{F}(q,\omega)|+\langle q\rangle | \mathcal{F}'(q,\omega)|\lesssim 1,\qquad
\langle q\rangle= (1+|q|^2)^{1/2}.
\end{equation}
For any function $\mathcal{F}(q,\omega)$ we always denote by $\mathcal{F}'=\partial_q\mathcal{F}=\mathcal{F}_q$ the derivative with respect to $q=r-t$.

The same is true if only
\beq\label{eq:decayforradiation}
|\Box \,\psi_{{}_{\,}} |+r^{-2}|{\triangle}_\omega \psi|\lesssim
r^{-1} \langle t+r\rangle^{-1-\varepsilon}\langle t-r\rangle^{-1+\varepsilon}\langle (r-t)_+\rangle^{-\varepsilon},\qquad
\varepsilon>0,
\eq
and data decays like $r^{-1}$.  This can be seen by expressing the wave operator in spherical
coordinates:
\beq\label{eq:sphericalwaveoperator}
\Box\, \psi =-r^{-1}(\pa_t+\pa_r)(\pa_t-\pa_r)(r\psi)+r^{-2}{\triangle}_\omega \psi,
\eq
 and integrating, in the $t\!+_{\!}r$ direction and in the $t\!-\!r$ direction, to obtain
a bound for $r\psi$ and the asymptotics \eqref{eq:radiationfield}; see for instance Section~5, 7 in \cite{L17}.
In particular for a solution of $\Box \psi=0$ with spherically symmetric compactly supported data, $\psi=\mathcal{F}(r-t)/r$ is exact for large $r$.

\subsection{Log correction to radiation field at null infinity for wave equation with sources.}
\label{sec:logcorrection}

General quadratic inhomogeneous terms as in \eqref{eq:F:slow} do not decay enough
along the light cone for \eqref{eq:decayforradiation} to hold.
This arises in particular for  \eqref{eq:weak:2}, where by \eqref{eq:radiationfield},
\beq\label{eq:inhomogeneitywithradiationfield}
\psi_t(t,x)^2\sim \mathcal{F}^{\,\prime}(r\!-t,\omega)^2\!/r^2 \,,
\eq
but we can get additional decay $| \mathcal{F}^{\,\prime}(q,\omega)|\lesssim  \langle q\rangle^{-1-\gamma}$, for some $0<\gamma<1$, if initial data decays like $r^{-1-\gamma}$.
 The first equation in \eqref{eq:weak:2} is modelled by
\begin{equation} \label{eq:source}
-\Box \,\phi=n(r\!-t,\omega)/r^2,
\end{equation}
where we assume at a minimum that $|n(q,\omega)|\lesssim \langle q\rangle^{-1-\epsilon}$, $\epsilon>0$.
 The solution to the forward problem for \eqref{eq:source} then has  a logarithmic correction to the asymptotic behavior:
\beq \label{eq:approximatesource2}
\phi(t,r\omega)
\sim\ln{\Big|\frac{r}{\langle t\!-\!r\rangle}\Big|}
\,\frac{\mathcal{F}_{01}(r\!-\!t,\omega) \! }{r}+\frac{\mathcal{F}_0(r\!-\!t,\omega)\!}{r},\qquad\text{as}\quad t\to\infty,\quad \text{while}\quad r\!\sim t\,.
\eq
In fact, applying the expression for the wave operator in spherical coordinates \eqref{eq:sphericalwaveoperator} to this gives
\begin{equation*}
-\Box \phi(t,r\omega)
\sim -\frac{2 {\mathcal{F}_{01}^{\,\prime}(r\!-\!t,\omega)}}{r^2}.
\end{equation*}
and so if
\begin{equation}\label{eq:F01primecond}
 2\mathcal{F}_{01}^{\,\prime}(q,\omega)=-n(q,\omega),
 \end{equation}
  then \eqref{eq:source} holds up to an error bounded by the right hand side of \eqref{eq:decayforradiation}, if we also assume some additional decay of the scattering data for $\mathcal{F}_0$. This error hence corresponds to a regular radiation field without a log correction, that can be included in $\mathcal{F}_0$ above.
 The condition \eqref{eq:F01primecond} only determines $\mathcal{F}_{01}(q,\omega)$ up to a function of $\omega$,
 \begin{equation*}
\lim_{q\to -\infty}\mathcal{F}_{01}(q,\omega)=N_{01}(\omega),
 \end{equation*}
which in turn is determined from interior homogeneous asymptotics.
We will see that also the radiation field  $\mathcal{F}_{0}(q,\omega)$   has to match interior homogeneous asymptotics:
 \begin{equation*}
\lim_{q\to -\infty}\mathcal{F}_{0}(q,\omega)=N_{0}(\omega) .
 \end{equation*}

\subsection{Higher order expansion in the wave zone towards null infinity.}\label{sec:higherordernull}

In \eqref{eq:approximatesource2} we have already seen the leading orders of the expansion for \eqref{eq:source} at null infinity.
We will need to work with a higher order expansion at null  infinity for the solution of \eqref{eq:source}:
\beq \label{eq:secondorderapproximatesource2}
\phi(t,{}_{\!}r\omega)\!
\sim \!\Psi_{{}_{\!}rad}(r\!-\!t,\omega,{}_{\!}1{}_{\!}/r{}_{\!}){}_{\!}
=\ln{\!\Big|\!\frac{2\,r}{\langle t\shortminus r{}_{\!}\rangle{}_{\!}}\Big|}
\frac{\mathcal{F}_{\!01\!}(r\!-\!t,\omega) \! \! }{r}+\frac{\!\mathcal{F}_{\!0}(r\!-\!t,\omega)\!\!}{r}
+\ln{\!\Big|\!\frac{2\,r}{\langle t\shortminus r{}_{\!}\rangle{}_{\!}}\Big|}
\frac{\mathcal{F}_{\!11\!}(r\!-\!t,\omega) \! \! }{r^2}+\frac{\!\mathcal{F}_{\!1\!}(r\!-\!t,\omega)\!\!}{r^2}
\eq
Applying the expression for the wave operator in spherical coordinates
\eqref{eq:sphericalwaveoperator} gives
\begin{multline*}
-r^{-1}(\pa_t+\pa_r)(\pa_t-\pa_r)(r\Psi_{rad})
\!\sim\! \frac{\!2\mathcal{F}_{\!01}^{\,\prime}(r\shortminus t,\omega{}_{\!}) \! \!}{r^2}\,
- \frac{\!{}_{\!}\mathcal{F}_{\!01{}_{\!}}(r\shortminus t,\omega{}_{\!}) \! \!}{r^3}\,\\
+\,\frac{\!2\mathcal{F}_{\!11}^{\,\prime}(r\shortminus t,\omega{}_{\!}) \!\! }{r^3}\,
-2\ln{\!\Big|\!\frac{2\, r}{\langle t\shortminus r{}_{\!}\rangle{}_{\!}}\Big|}
\frac{\mathcal{F}_{\!11}^{\, \prime}(r\shortminus t,\omega{}_{\!}) \!\! }{r^3}\,
-\frac{\!2\mathcal{F}_{\!1}^{\,\prime}(r\shortminus t,\omega{}_{\!}) \! \!}{r^3}\,
+\frac{2(r\shortminus t)}{\langle t\shortminus r{}_{\!}\rangle^2{}_{\!}\!}
\frac{\!\mathcal{F}_{\!11}(r\shortminus t,\omega{}_{\!}) \! \! }{r^3},
\end{multline*}
up to terms of higher order in $1/r$,
and similarly
\beq
\frac{1}{r^2}\triangle_\omega \Psi_{rad}
\sim \ln{\Big|\frac{2\,r}{\!\langle t\!-\!r\rangle}\Big|}
\,\frac{\! \triangle_\omega \mathcal{F}_{01}(r\!-\!t,\omega)}{r^3}+\frac{\triangle_\omega \mathcal{F}_0(r\!-\!t,\omega)\!}{r^3}.
\eq
Equating coefficients of order $r^{-2}$ in
\eqref{eq:source} gives
\beq\label{eq:F01eq}
2\mathcal{F}_{01}^{\,\prime}(q,\omega) =-n(q,\omega),
\eq
and equating coefficient of order $r^{-3} \ln{\big|\frac{2\,r}{\langle t-r\rangle}\big|}$ in
\eqref{eq:source} gives
\beq\label{eq:F11eq}
2\mathcal{F}_{11}^{\, \prime}(q,\omega)
=\triangle_\omega \mathcal{F}_{01}(q,\omega),
\eq
Finally equating coefficients of order $r^{-3}$ and using \eqref{eq:F11eq} gives
\begin{equation}\label{eq:F1eq}
2\mathcal{F}_1^\prime(q,\omega)=-\mathcal{F}_{01}(q,\omega)+\Delta_\omega\mathcal{F}_{01}(q,\omega)+\Delta_\omega \mathcal{F}_0(q,\omega)+2 (q/\langle q\rangle^2) \mathcal{F}_{11} .
\end{equation}
Given $ \mathcal{F}_{\!0}(q,\omega)$ the system \eqref{eq:F01eq}, \eqref{eq:F11eq}, \eqref{eq:F1eq} has a unique solution modulo integration constants, that depend on $\omega$, which are determined from the source $n(q,\omega)$. In fact, $\mathcal{F}_{01}(q,\omega)$ is determined up to a constant and
$\mathcal{F}_{11}(q,\omega)$  and $\mathcal{F}_{1}(q,\omega)$ are determined up to
a linear function in $q$ that depend on $\omega$. Note in particular that as opposed to $\mathcal{F}_{01}(q,\omega)$, and $\mathcal{F}_0(q,\omega)$ which are bounded,  $\mathcal{F}_{11}(q,\omega)$  and $\mathcal{F}_1(q,\omega)$ grow linearly as $q\to-\infty$.
The constants will be determined by matching to the interior solution as $q\to -\infty$ in the next section.

While $\mathcal{F}_0$ can be viewed as given,  we will see that the  matching procedure also gives conditions that $\mathcal{F}_0$ has to satisfy and as a result \emph{it is not completely free to choose}; see Section~\ref{sec:matching}.

In summary, we record that:
Suppose that we have a solution to the system
\eqref{eq:F01eq}, \eqref{eq:F11eq}, \eqref{eq:F1eq}, so $\Psi_{rad}$ is a solution to \eqref{eq:source} up to order $r^{-3}$.
Then taking into account the lower order terms we get
\beq
\Box \Psi_{rad}+ \frac{\!n(r\!-\!t,\omega)\!\!}{r^2}=R_{rad},
\eq
where
\beq\label{eq:Rrad}
R_{rad}=\!\Big(\!2\ln{\!\Big|\frac{2\,r}{\!\langle t\!-\!r{}_{\!}\rangle{}_{\!}}\Big|}\!-\!3\!\Big)
\frac{\!\mathcal{F}_{\!11\!}(r\!-\!t,\omega)\!\!}{r^4}
+\frac{\!2\mathcal{F}_{\!1\!}(r\!-\!t,\omega)\!\!}{r^4}
+\ln{\!\Big|\frac{2\,r}{\!\langle t\!-\!r{}_{\!}\rangle{}_{\!}}\Big|}
\frac{\! \triangle_\omega \mathcal{F}_{11}(r\!-\!t,\omega)\!\!}{r^4}+\frac{\!\triangle_\omega \mathcal{F}_1(r\!-\!t,\omega)\!\!}{r^4}.
\eq

\section{Expansion close to the light cone of the interior homogeneous asymptotics}
\label{sec:interiorexpansion}

We have seen in Section~\ref{sec:hom:int} that in presence of quadratic or cubic source terms,
solutions to the forward problems \eqref{eq:box:n:2} and \eqref{eq:highersourcbegining} have homogeneous asymptotics in the interior $r<t$. In order to show existence for the backward problem with data at null infinity the asymptotic behavior at time-like infinity has to be taken into account. This is because the solution in this region will influence the solution along the backward light cone and hence near null infinity.

In view of the characterisation of the interior asymptotics in terms of homogeneous functions obtained in Section~\ref{sec:hom:int}, we are now interested in their behaviour close to the lightcone.
Since the homogeneous functions in the interior are solutions to a wave equation with sources on the lightcone, we can derive their expansions towards the lightcone explicitly. The terms in this expansion are relevant for the matching procedure; recall here also the discussion in Section~\ref{sec:matching:intro} and in particular Fig.~\ref{fig:limits}.

\subsection{Interior homogeneous asymptotics.}

Consider the equation \eqref{eq:source}.
We have seen by a scaling argument in Section~\ref{sec:hom:int},
that the leading order interior asymptotics is given by \eqref{eq:wavesourceconeformulasec2.3}.

Alternatively we can use the fundamental solution directly to compute the asymptotics directly; cf.~\cite{L90a,L17}. The convolution of the right hand side of \eqref{eq:source} with the forward fundamental solution $E=\delta(t^2-|x|^2)H(t)/2\pi$ of $-\Box$ is given by
 \begin{equation}\label{eq:Phi:21}
\Phi[n](t,r\omega)=\int_{r-t}^{\infty} \frac{1}{4\pi}\int_{\mathbb{S}^2}{\frac{\chi\,
 n({q},{\sigma})\, dS({\sigma})\,d {q}}{t-r+{q}+r\big(1-\langle\, \omega,{\sigma}\rangle\big)}\,}\,.
\end{equation}
Although we do not need to use this here one can get a rate at which \eqref{eq:Phi:21} converges to
\begin{equation}\label{eq:wavesourceconeformula}
 \Psi_1[N](t,r\omega)=\frac{1}{\!4\pi }\!\int_{\mathbb{S}^2}\!\!
  \frac{N(\sigma{}_{\!})\, dS(\sigma) }{t\!-\!\langle\sigma,r\omega{}_{\!}\rangle},\qquad\text{where}\quad
   N{}_{\!}(\omega\!)=\!\! \int_{-\infty}^\infty\!\!n(q,\omega{}_{\!})\,\ud q,
\end{equation}
as $t\to\infty$ while $r/t\leq c<1$.
 To get improved asymptotics we will assume that $n$ is more concentrated close to the light cone:
\begin{equation}\label{eq:ndecay}
 |\,(\langle q\rangle\pa_{q})^k\pa_\omega^\alpha n({q},\omega)| \leq C
 \langle q\rangle^{-2-2\gamma} ,\qquad k+|\alpha|\leq N,\quad 0<\gamma<1,
\end{equation}
as is the case for \eqref{eq:inhomogeneitywithradiationfield}.
 We will also use that for $t>r$,
\begin{equation}
  \Box\Psi_1[N]=0,\qquad (t>r),
\end{equation}
which follows immediately from \eqref{eq:phiinfinityformula}, given that the source of $\phi_{1,\infty}$ is only supported on the light cone, or alternatively directly from the formula \eqref{eq:wavesourceconeformula} which shows that $\Psi_1$ is a superposition of plane waves.

\subsection{Expansion of the leading order interior homogeneous asymptotics $r<t$.}
\label{sec:homoexpansion}

As $r\to t$ we  have an expansion
\beq\label{eq:homoexpansion}
\Psi_1[N](t,r\omega{}_{\!})
= N_{01}{}_{\!}(\omega{}_{\!})\frac{1}{r}\!\ln{\!\Big|\frac{ 2\,r}{t\shortminus r}\Big|}+N_{0}(\omega{}_{\!})\frac{1}{r}
+N_{{}_{\!}11{}_{\!}}(\omega{}_{\!})\frac{r\shortminus t\!}{\,r^2}\ln{\!\Big|\frac{ 2\, r}{t\shortminus r}\Big|}
+N_{{}_{\!}1{}_{\!}}(\omega{}_{\!})\frac{r\shortminus t\!}{\,r^2}
+O\Big(\!\frac{(r\shortminus t)^2\!\!\!}{\,r^3}
\ln{\!\Big|\frac{ 2\, r}{t\shortminus r}\Big|}\Big) .
\eq

To see that such an expansion exists let
\beq
N(z,\omega)\!=\!\int_{\langle \sigma,\omega\rangle =z}\!\!\!\!\! \!\! \!  N(\sigma) \,ds(\sigma) \Big/\!\!\int_{\langle \sigma,\omega\rangle =z} \!\!\!\! \!\! \!  ds(\sigma)
\eq
 be the average of $N(\sigma)$ over the circle $\langle \omega,\sigma\rangle\!=\!z$.
Then $N(\omega,z)$ is a smooth function on $\mathbb{S}^2\times[-1,1]$ if $N(\omega)$ is a smooth function on $\mathbb{S}^2$ , and with $y\!=\!(t-r\!)/r $, we have
\beq
r\Psi_1[N](t,r\omega)= \int_{-1}^1 N(\omega,z)(1+y -z)^{-1} dz.
\eq
The expansion \eqref{eq:homoexpansion} follows from expanding
$ N(\omega,z)$ in a Taylor series and using polynomial division and partial fractions to evaluate the integrals; see Appendix~\ref{sec:SphericalTaylor}.

We have seen in Section~\ref{sec:hom:asymptotics:leading} that  $N_{01}$ and $N_0$ are given by \eqref{eq:N:01:0}. We will now show that $N_{11}$ and $N_1$ are determined from $N_{01}$ and $N_0$. This means that all coefficients in \eqref{eq:homoexpansion} can be computed from $N(\omega)$ alone.

Putting the expansion \eqref{eq:homoexpansion}  into the equation
$ \Box \Psi_1[N]=0$  gives relations between the coefficients:
\beq\label{eq:homocond}
2N_{11}(\omega)=\triangle_\omega N_{01}(\omega),\qquad
-N_{01}(\omega)+4N_{11}(\omega)-2N_1(\omega)+ \Delta_\omega N_0(\omega)=0,
\eq
and using the first equation the last equation becomes:
  \begin{equation}
    \label{eq:N1}
    2N_1(\omega)=2\Delta_\omega N_{01}(\omega)+\Delta_\omega N_0-N_{01}(\omega) .
  \end{equation}
Writing $r\Psi_1[N](t,r\omega)=\Psi_1(s,\omega)$, with $s=t/r$, and applying the wave operator to the expansion expressed in spherical coordinates $(s,\omega)$ gives the conditions \eqref{eq:homocond}.
In fact, using the expression for the wave operator in spherical coordinates
\eqref{eq:sphericalwaveoperator} we get
\beq
r^3 \Box \Psi_1[N]=-r^2(\pa_t^2-\pa_r^2) \Psi_1(s,\omega)
+ \triangle_\omega \Psi_1(s,\omega)
= \pa_s\big((s^2\!-1)\pa_s\Psi_1(s,\omega)\big) +\triangle_\omega \Psi_1(s,\omega) =0,
\eq
and applying the first operator to $r$ times the expansion \eqref{eq:homoexpansion} gives,  modulo $O\big(({}_{\!}s\shortminus 1{}_{\!})
\ln{\!\big|\frac{ 2}{{}_{\!}s\shortminus 1{}_{\!}}\big|}\big)$,
\begin{multline*}
\pa_s \Big((s^2\!-\!1)\pa_s\Big( N_{01}{}_{\!}(\omega{}_{\!})\ln{\Big|\frac{ 2}{s\shortminus 1}\Big|}+N_{0{}_{\!}}(\omega{}_{\!})
+N_{{}_{\!}11{}_{\!}}(\omega{}_{\!})(1\!-\!s)\ln{\Big|\frac{ 2}{s\shortminus 1}\Big|}
+N_{{}_{\!}1\!}(\omega{}_{\!})(1\!-\!s)\Big)\Big)\\
=\pa_s \Big(\!(s^2\!-{}_{\!}1)\Big( {}_{\!} \frac{\shortminus N_{01}{}_{\!}(\omega{}_{\!})\!}{s\shortminus 1}
-N_{{}_{\!}11{}_{\!}}(\omega{}_{\!})
\Big(\!\ln{\!\Big|\frac{ 2}{{}_{\!}s\shortminus{}_{\!}1{}_{\!}}\Big|}{-1}\!\Big){}_{\!}
-N_{{}_{\!}1{}_{\!}}(\omega{}_{\!})\!\Big) \!\Big)\!
=\shortminus N_{01}{}_{\!}(\omega{}_{\!})
+N_{{}_{\!}11{}_{\!}}(\omega{}_{\!})\Big(\!s+1{}_{\!}-2s\Big(\!\ln{\!\Big|\frac{ 2}{{}_{\!}s\shortminus 1{}_{\!}}\Big|}-1\!\Big)\!\Big){}_{\!}
-2sN_{{}_{\!}1{}_{\!}}(\omega{}_{\!})\\
=-2N_{{}_{\!}11{}_{\!}}(\omega{}_{\!})\ln{\!\Big|\frac{ 2}{{}_{\!}s\shortminus 1}\Big|}
-N_{01}{}_{\!}(\omega{}_{\!})  +4N_{11}(\omega) -2N_{{}_{\!}1{}_{\!}}(\omega{}_{\!})
-2N_{{}_{\!}11{}_{\!}}(\omega{}_{\!})\big({}_{\!}s\shortminus 1{}_{\!}\big)
\ln{\!\Big|\frac{ 2}{{}_{\!}s\shortminus 1 }\Big|}
 +3N_{{}_{\!}11{}_{\!}}(\omega{}_{\!})\big({}_{\!}s\shortminus 1{}_{\!}\big)
-2\big({}_{\!}s\shortminus 1{}_{\!}\big)N_{{}_{\!}1{}_{\!}}(\omega{}_{\!}),
\end{multline*}
and matching this to $-\triangle_\omega $ applied to the expansion \eqref{eq:homoexpansion} gives the conditions \eqref{eq:homocond}.

\subsection{Expansion of the second order interior homogeneous asymptotics $r<t$.} \label{sec:expansion:second:order:interior}

Recall the second order homogeneous asymptotics $\Psi_2$ from Section~\ref{sec:hom:asymptotics:second}.
As $r\to t$ we have an expansion
\beq\label{eq:homoexpansion2}
\Psi_2[M](t,{}_{\!}r\omega{}_{\!})
= \frac{\!M_{0}(\omega{}_{\!})\!}{r^2}\frac{r}{\!r\shortminus t\!}
+\frac{\!M_{11}{}_{\!}(\omega{}_{\!})\!}{r^2}\ln{\!\Big|{}_{\!}\frac{ 2\, r}{t\shortminus r}{}_{\!}\Big|}
+ \frac{\!M_{1}{}_{\!}(\omega{}_{\!})\!}{r^2}
+\frac{\!M_{21}{}_{\!}(\omega{}_{\!})}{r^2}\frac{r\shortminus t\!}{\,r}
\ln{\!\Big|{}_{\!}\frac{ 2\, r}{t\shortminus r}{}_{\!}\Big|}
+\frac{\!M_{2}{}_{\!}(\omega{}_{\!})}{r^2}\frac{r\shortminus t}{\,r}
+O\Big(\!\frac{(r\shortminus t)^2\!\!\!}{\,r^4}
\ln{\!\Big|{}_{\!}\frac{ 2\, r}{t\shortminus r}{}_{\!}\Big|}\Big).
\eq

Similarly to the above  the  existence of an expansion follows from
\beq
r^2\Psi_2[M](t,r\omega)= \int_{-1}^1 \! \!\! \! M(\omega,z)(1+y -z)^{-2} dz,
\quad\text{where}\quad
M(z,\omega)\!=\!\int_{\langle \sigma,\omega\rangle =z}\!\! \!\! \!\!\! \! \!\! M(\sigma) \,ds(\sigma) \Big/\!\!\int_{\langle \sigma,\omega\rangle =z} \!\!\! \!\! \!\!\!\! \!\! \!  ds(\sigma).
\eq

We already know from \eqref{eq:lowestinteriorasymptotics2} that
\begin{equation}
  \label{eq:M01:given}
  M_0(\omega)=\frac{1}{2}M(\omega) \,,\qquad M_{11}(\omega)=-\frac{1}{4}\Delta_\omega M(\omega)\,, \qquad
  M_1(\omega)=\frac{1}{4}M(\omega)-\frac{1}{4}\Delta_\omega M(\omega)+M_1^\ast(\omega),
\end{equation}
where also $M_1^*(\omega)$, given by \eqref{eq:M1:ast}, is computed from $M(\omega)$ only.

We now show that knowing $M_0$ and $M_1$ alone, we can compute the coefficients $M_{11}$, $M_{21}$ and $M_2$ using that $\Box\Psi_2=0$ away from the lightcone by \eqref{eq:phi2infinityformula}.  This  shows that all coefficients in \eqref{eq:homoexpansion2} are determined from $M(\omega)$.

Writing $r^2\Psi_2[M](t,r\omega)=\Psi_2(s,\omega)$, with $s=t/r$, and applying the wave operator to the expansion expressed in spherical coordinates $(s,\omega)$ gives
\begin{equation*}
r^{4\,}\Box \Psi_2[N]=-r^2(\pa_t^2\!-\pa_r^2) \Psi_2(s,\omega)
+2\big(1\!-\!r\pa_r\big)\Psi_2(s,\omega)
+ \triangle_\omega \Psi_2(s,\omega)
= \pa_s^2\big((s^2\!-\!1)\Psi_2(s,\omega)\big)
 +\triangle_\omega \Psi_2(s,\omega).
\end{equation*}
Applying the first operator to the expansion \eqref{eq:homoexpansion2} gives,
\begin{multline}
\pa_s^2 \Big((s^2\!-\!1)\Big( M_{0}{}_{\!}(\omega{}_{\!})\frac{ 1}{1\shortminus s}
+M_{{}_{\!}11{}_{\!}}(\omega{}_{\!})\ln{\Big|\frac{ 2}{s\shortminus{}_{\!}1}\Big|}+M_{1{}_{\!}}(\omega{}_{\!})
+M_{{}_{\!}21{}_{\!}}(\omega{}_{\!})(1\shortminus s)\ln{\Big|\frac{ 2}{s\shortminus{}_{\!}1}\Big|}+M_{2{}_{\!}}(\omega{}_{\!})(1\shortminus s)\Big)
\Big)\\
=2\Big(M_{{}_{\!}11{}_{\!}}(\omega{}_{\!})\ln{\Big|\frac{ 2}{s\shortminus{}_{\!}1}\Big|}+M_{1{}_{\!}}(\omega{}_{\!})
+M_{{}_{\!}21{}_{\!}}(\omega{}_{\!})(1\shortminus s)\ln{\Big|\frac{ 2}{s\shortminus{}_{\!}1}\Big|}+M_{2{}_{\!}}(\omega{}_{\!})(1\shortminus s)\Big)\\
-4s\Big(\frac{M_{{}_{\!}11{}_{\!}}(\omega{}_{\!})}{s\shortminus{}_{\!}1}
+M_{{}_{\!}21{}_{\!}}(\omega{}_{\!})\ln{\Big|\frac{ 2}{s\shortminus{}_{\!}1}\Big|}-M_{{}_{\!}21{}_{\!}}(\omega{}_{\!})
+M_{2{}_{\!}}(\omega{}_{\!})\Big)
+(s^2\!-1)\Big(\frac{M_{{}_{\!}11{}_{\!}}(\omega{}_{\!})}{(s\shortminus{}_{\!}1)^2}
 +\frac{M_{{}_{\!}21{}_{\!}}(\omega{}_{\!})}{s\shortminus{}_{\!}1}\Big)
\\
= -2\frac{M_{{}_{\!}11{}_{\!}}(\omega{}_{\!})}{s\shortminus 1}
+\big(\!\shortminus 4 M_{{}_{\!}21{}_{\!}}(\omega{}_{\!})+2M_{{}_{\!}11{}_{\!}}(\omega{}_{\!})\big)
\ln{\Big|\frac{ 2}{s\shortminus{}_{\!}1{}_{\!}}\Big|}+ 2M_{1}{}_{\!}(\omega{}_{\!}) -2M_{{}_{\!}11{}_{\!}}(\omega{}_{\!})
 +6M_{{}_{\!}21{}_{\!}}(\omega{}_{\!})-4M_{{}_{\!}2{}_{\!}}(\omega{}_{\!})
+O\Big(\!({}_{\!}s\shortminus{}_{\!}1{}_{\!})
\ln{\!\big|\frac{ 2}{{}_{\!}s\shortminus{}_{\!}1{}_{\!}}\big|}\Big),
\end{multline}
and comparing this to $-\triangle_\omega $ applied to the expansion \eqref{eq:homoexpansion2} gives the conditions:
\beq\label{eq:M0cond}
M_{11}(\omega)= -\frac{1}{2}\triangle_\omega M_{0}{}_{\!}(\omega{}_{\!}),\qquad
M_{21}(\omega)=\frac{1}{4} \triangle_\omega M_{{}_{\!}11{}_{\!}}(\omega{}_{\!}) +\frac{1}{2}M_{11},
 \eq
 \beq
M_2(\omega) =\frac{1}{2}M_{1}{}_{\!}(\omega{}_{\!}) + \frac{1}{4} \triangle_\omega M_{1}{}_{\!}(\omega{}_{\!})-\frac{1}{2}M_{{}_{\!}11{}_{\!}}(\omega{}_{\!})
+\frac{3}{2}M_{{}_{\!}21{}_{\!}}(\omega{}_{\!}) .
\eq
 Note that these formulas in particular verify the expression for  $M_{11}$ found in \eqref{eq:lowestinteriorasymptotics2}.

\section{Matching of the asymptotics of the interior homogeneous solution to the asymptotics at null infinity  in the case of radiation fields decaying in the exterior}\label{sec:matching}

We now want to match  the asymptotics at null infinity \eqref{eq:secondorderapproximatesource2}
to the interior homogeneous asymptotics \eqref{eq:homoexpansion} and \eqref{eq:homoexpansion2}.
 This means in particular that the limits of $\mathcal{F}_{i}$, and $\mathcal{F}_{i1}$, are non-trivial as $q\to -\infty$. On the other hand, we will  impose that
 \beq
\label{eq:limitF0plus}
\lim_{q\to +\infty}  \mathcal{F}_{i}(q,\omega)=0,\qquad
\lim_{q\to +\infty}  \mathcal{F}_{i1}(q,\omega)=0,\qquad i=0,1,
\eq
because this is true for the forward solution with sufficiently fast decaying data and if $n(q,\omega)$ decays sufficiently fast. This can be seen from the explicit solution formula
\eqref{eq:Phi:21}  if \eqref{eq:ndecay} holds. In Section~\ref{sec:hom:ext} we will extend the construction to the case when these limits do not vanish.

The source $n(q,\omega)$ is free to choose and $\mathcal{F}_0(q,\omega)$ is free to choose up to certain conditions that we derive below. As we will see $\mathcal{F}_1$, and $\mathcal{F}_{i1}$, $i\!=\!1,2$ are determined from these and matching conditions with the interior solution.

Matching the asymptotics allows us to define an asymptotic solution, extending from the interior to the exterior across the lightcone, and underlies the proof of Theorems~\ref{thm:n} and~\ref{thm:m}. More precisely, the asymptotic solution interpolates smoothly across a region where $r-t\sim -r^{1/2}$, between homogeneous solutions in the interior, and polyhomogeneous expansions near null infinity.

\subsection{Matching of order $r^{-1}$.}
As we see from  \eqref{eq:F01eq}, $\mathcal{F}_{01}$  is determined from $n$ apart from an integration constant that depends on $\omega$,
that has to be chosen to match the interior homogeneous solution in \eqref{eq:homoexpansion}.
Indeed integrating \eqref{eq:F01eq} with \eqref{eq:limitF0plus} gives
\begin{equation}
  \label{eq:F01:n}
 \mathcal{F}_{01}(q,\omega)=\frac{1}{2}\int_{q}^\infty n(s,\omega)\, \ud s.
\end{equation}
In view of \eqref{eq:N:01:0} this matches the interior homogeneous solution \eqref{eq:homoexpansion}:
\beq
N_{01}(\omega)=\lim_{q\to -\infty}  \mathcal{F}_{01}(q,\omega).
\eq

 Let us set
  \begin{equation}
    \label{eq:H01}
    \mathcal{H}_{01}= \mathcal{F}_{01}(q,\omega)- N_{01}(\omega)\chi_{q\leq 0}
  \end{equation}
where $\chi_{q\leq 0}$ is a smooth cutoff so that
$\chi_{q\leq 0}=1$, when $q\leq 0$ and  $\chi_{q\leq 0}=0$, when $q> 1$.
Then in view of the decay assumption on $n$ in \eqref{eq:ndecay} it follows from \eqref{eq:F01:n}
\beq\label{eq:angulardecay}
\big|(\langle q\rangle \pa_q)^k\pa^\alpha_\omega  \mathcal{H}_{01} \big|\lesssim \langle q\rangle^{-1-2\gamma}.
\eq

We see from \eqref{eq:homoexpansion} that the radiation field $\mathcal{F}_0(q,\omega)$ must satisfy the matching condition
\beq
\lim_{q\to -\infty}  \mathcal{F}_{0}(q,\omega)=N_{0}(\omega),
\eq
and recall from \eqref{eq:N:01:0} that $N_{0}(\omega)$ is determined from $N_{01}(\omega)$.

\subsection{Matching of order $r^{-1}q^{-1}$.}
 Moreover, in view of the leading order term in \eqref{eq:homoexpansion2}, we have the matching condition
  \begin{equation}
    \lim_{q\to -\infty}q\bigl(\mathcal{F}_0(q,\omega)-N_0(\omega)\bigr)=M_0(\omega)\,.
  \end{equation}
Let us set
\begin{equation}
  \label{eq:H0}
  \mathcal{H}_0=\mathcal{F}_0(q,\omega)-N_0(\omega)\chi_{q\leq 0}-M_0(\omega) \frac{q}{\langle q\rangle^2}\chi_{q\leq 0} ,
\end{equation}
then we   assume that
\begin{equation} \label{eq:F0decay}
 \big|(\langle q\rangle \pa_q)^k\pa^\alpha_\omega\mathcal{H}_0\big|
 \les \langle q\rangle^{-\kappa},\,\qquad 1<\kappa\leq 1+\min\{1,2\gamma\},
\end{equation}
since this is true for the forward problem.

With these choices we have matched $\Psi_{rad}$ to $\Psi_1[N]+\Psi_2[M]$ at order $r^{-1}$.

\subsection{Matching of order $r^{-2}$.}
Integrating \eqref{eq:F11eq} using \eqref{eq:limitF0plus} gives
\beq 
2\mathcal{F}_{11}(q,\omega)
=-\int_q^\infty \triangle_\omega \mathcal{F}_{01}(s,\omega)\, \ud s .
\eq
In fact, integrating by parts we get
\begin{equation*}
\int_q^\infty \!\!\!\!\mathcal{F}_{01}(s,\omega{}_{\!})\, ds =
\frac{1}{2}\!\int_q^\infty \!\!\!\!\int_{\rho}^\infty  \!\!\!\! \! n(s,\omega{}_{\!})\, ds \, d\rho
=\frac{1}{2}\!\int_q^\infty\!\!\!\! (\rho-q) \,  n(\rho,\omega{}_{\!})\, d\rho
=-N_{01{}_{\!}}(\omega{}_{\!})\, q\,+{M}_{0{}_{\!}}(\omega{}_{\!})+\frac{1}{2}\!\int_{-\infty}^q\!\!\!\!( \rho-q) \,  n(\rho,\omega{}_{\!})\, d\rho,
\end{equation*}
where
\beq
{M}_{0}(\omega)=\frac{1}{2}\int_{-\infty}^\infty\!\!\! \rho \,  n(\rho,\omega)\, d\rho .
\eq
Using  \eqref{eq:homocond} and \eqref{eq:M0cond} it follows that
\begin{equation}\label{eq:F11:N11}
  \begin{split}
2\mathcal{F}_{11}(q,\omega{}_{\!})=&  \Delta_\omega N_{01}(\omega)\, q-\Delta_\omega M_0(\omega)-\frac{1}{2}\int_{-\infty}^q (\rho-q)\Delta_\omega n(\rho,\omega)\ud \rho\\
=&2N_{11}(\omega) \, q\, \chi_{q\leq 0}+ 2{M}_{11}(\omega{}_{\!})\chi_{q\leq  0}+O(\langle q\rangle^{-2\gamma}) ,
\end{split}
\end{equation}
  and also using \eqref{eq:homocond},
  \begin{equation}
    \mathcal{F}_{11}'(q,\omega)-N_{11}(\omega)\chi_{q\leq 0}=\frac{1}{2}\triangle_\omega\mathcal{F}_{01}(q,\omega)-\frac{1}{2}\triangle N_{01}(\omega)\chi_{q\leq 0}=\frac{1}{2}\triangle_\omega\mathcal{H}_{01}\les \langle q\rangle^{-1-2\gamma}\,.
  \end{equation}

In other words if we set
\begin{equation}
  \label{eq:H11}
  \mathcal{H}_{11}(q,\omega)=\mathcal{F}_{11}(q,\omega)-N_{11}(\omega)q\:\chi_{q\leq 0}-M_{11}(\omega)\chi_{q\leq 0},
\end{equation}
it follows that
\begin{equation}
  \label{eq:H11:estimate}
  |(\langle q\rangle \pa_q)^k \pa_\omega^\alpha \mathcal{H}_{11}|\les \langle q\rangle^{-2\gamma}\,,\qquad ( k\leq 1,|\alpha|\leq 1).
\end{equation}

Furthermore, let us set
\begin{equation}
  \label{eq:H1}
  \mathcal{H}_{1}(q,\omega)=\mathcal{F}_{1}(q,\omega)-N_{1}(\omega)q\:\chi_{q\leq 0}-M_1(\omega)\chi_{q\leq 0}\,.
\end{equation}
It remains to show that
\begin{equation} \label{eq:H1:estimate}
  |(\langle q\rangle \pa_q)^k \partial_\omega^\alpha \mathcal{H}_1|\les \langle q\rangle^{-\kappa+1}\,,\qquad ( k\leq 1,|\alpha|\leq 1),
\end{equation}
which then concludes the matching of $\Psi_{rad}$ in \eqref{eq:secondorderapproximatesource2} to  $\Psi_1[N]+\Psi_2[M]$ in  \eqref{eq:homoexpansion} and \eqref{eq:homoexpansion2} in the interior,  at order $r^{-2}$, as $q\to -\infty$.
By \eqref{eq:F1eq} and using \eqref{eq:F01:n} and \eqref{eq:F11:N11}
  \begin{equation}
    \begin{split}
   2\mathcal{F}_1^{\,\prime}(q,\omega)=&-\mathcal{F}_{01}(q,\omega)+\Delta_\omega\mathcal{F}_{01}(q,\omega)+\Delta_\omega \mathcal{F}_0(q,\omega)+2 (q/\langle q\rangle^2) \mathcal{F}_{11}\\
      =& \frac{1}{2}\int_q^\infty \Bigl[-n(s,\omega)+\Delta_\omega n(s,\omega)\Bigr]\ud s-\frac{1}{2}\int_{-\infty}^q (q/\langle q\rangle^2)(\rho-q)\Delta_\omega n(\rho,\omega)\ud \rho\\
      &+\Delta_\omega N_{01}(\omega) \,  (q/\langle q\rangle)^2-\Delta_\omega M_0(\omega)\,(q/\langle q\rangle^2)+\Delta_\omega \mathcal{F}_0(q,\omega)\\
      =&-N_{01}(\omega)+2\Delta_\omega N_{01}(\omega)-\Delta_\omega M_0(\omega)\,(q/\langle q\rangle^2)+\Delta_\omega \mathcal{F}_0(q,\omega)-\Delta_\omega N_{01}(\omega) /\langle q\rangle^2\\
      &+\frac{1}{2}\int_{-\infty}^q n(s,\omega)\ud s-\frac{1}{2\langle q\rangle^2}\int_{-\infty}^q\Delta_\omega n(s,\omega)\ud s -\frac{q}{2\langle q\rangle^2}\int_{-\infty}^q s\Delta_\omega n(s,\omega)\ud s ,
    \end{split}
  \end{equation}
  and by  \eqref{eq:N1}
  \begin{equation}
    -N_{01}(\omega)+2\Delta_\omega N_{01}(\omega)=2N_1(\omega)-\Delta_\omega N_0\,.
  \end{equation}

 Therefore, in view of \eqref{eq:H0},
 \begin{equation} \label{eq:F1:prime}
   \begin{split}
2\mathcal{F}_{\!1}^{\,\prime}(q,\omega)
=&2N_{1}(\omega{}_{\!})\chi_{q\leq 0}
+\triangle_\omega\big( \mathcal{F}_{\!0}(q,\omega)- N_{0}(\omega)\chi_{q\leq  0}\,-M_0(\omega)\chi_{q\leq  0}\,(q/\langle q\rangle^2)\chi_{q\leq 0}\big)
+\mathcal{R}_1(q,\omega)  \\
=& 2\pa_q \big( N_1(\omega) \chi_{q\leq 0}q\big)+(\triangle_\omega \mathcal{H}_0)(q,\omega)+\mathcal{R}_1(q,\omega) ,
\end{split}
\end{equation}
 where
  \begin{equation}
    \label{eq:Ntilde}
    \begin{split}
    \mathcal{R}_1(q,\omega) =&\Big(-\frac{1}{2 \langle q\rangle^2}\Delta_\omega N(\omega)
    +\frac{1}{2}\int_{-\infty}^q n(s,\omega)\ud s-\frac{1}{2\langle q\rangle^2}\int_{-\infty}^q\Delta_\omega n(s,\omega)\ud s -\frac{q}{2\langle q\rangle^2}\int_{-\infty}^q s\Delta_\omega n(s,\omega)\ud s\Big)\chi_{q\leq 0}\\
    &+\bigtwo(-\mathcal{F}_{01}(q,\omega)+\Delta_\omega\mathcal{F}_{01}(q,\omega)+2 (q/\langle q\rangle^2) \mathcal{F}_{11}\bigtwo)\chi_{q>0}
    -2 N_{1}(\omega{}_{\!})\chi_{q\leq 0}^\prime q
    =O(\langle q\rangle^{-1-\min(1,2\gamma)}),
  \end{split}
\end{equation}
and $\chi_{q>0}=1-\chi_{q\leq 0}$.
Integrating \eqref{eq:F1:prime} we get  with \eqref{eq:F0decay},
\begin{equation}\label{eq:F1:M1}
\begin{split}
  \mathcal{F}_{\!1}(q,\omega)
=& N_{1}(\omega{}_{\!})\chi_{q\leq 0}\, q
-\frac{1}{2}\int_{-\infty}^{+\infty} \Bigl[\Delta_\omega\mathcal{H}_0+\mathcal{R}_1(q,\omega)\Bigr]\ud q
+ O(\langle q\rangle^{-\kappa+1})\\
  =& N_{1}(\omega{}_{\!})\chi_{q\leq 0}\, q
+M_1(\omega)\chi_{q\leq 0}
+ O(\langle q\rangle^{-\kappa+1}),
\end{split}
\end{equation}
 provided
  \begin{equation}
    \int_{-\infty}^{+\infty} (\triangle_\omega\mathcal{H}_0)(q,\omega) \ud q+\int_{-\infty}^{+\infty} \mathcal{R}_1(q,\omega) \ud q = -2M_1(\omega),
  \end{equation}
  which can be viewed as a condition for $\int \mathcal{H}_0\ud q$:
  \begin{equation}
    \label{eq:F0:elliptic}
    \Delta_\omega\biggl( \int_{-\infty}^{+\infty}\mathcal{H}_0(q,\omega) \ud q \biggr) =-2M_{1}(\omega)-\int_{-\infty}^{+\infty}\mathcal{R}_1(q,\omega)\ud q .
  \end{equation}

The equation \eqref{eq:F1:M1} in particular confirms \eqref{eq:H1:estimate}.

\subsection{The integrability condition.}\label{sec:integrabilitycondition}

  The equation \eqref{eq:F0:elliptic} has an integrability condition: the average on the sphere of the right hand side must vanish.
  First note that
  \begin{equation}
    \begin{split}
      \int_{-\infty}^{+\infty}\int_{\mathbb{S}^2}\mathcal{R}_1(q,\omega)\,\ud S(\omega)\ud q=&\int_{\mathbb{S}^2}\biggl[\frac{1}{2}\int_{-\infty}^{0}\int_{-\infty}^q n(s,\omega)\ud s  \ud q-\int_0^\infty \mathcal{F}_{01}(q,\omega)\ud q\biggr]\ud S(\omega)\\
      =&-\frac{1}{2}\int_{\mathbb{S}^2}\int_{-\infty}^\infty q\, n(q,\omega) \ud q\, \ud S(\omega)=-\frac{1}{2}\int_{\mathbb{S}^2} M(q,\omega)\ud S(\omega),
  \end{split}
\end{equation}
and thus this condition reduces to
\begin{equation}
  \label{eq:8}
  -2\overline{M_1}+\frac{1}{2}\overline{M}=0,
\end{equation}
where $\overline{M}=\int_{\mathbb{S}^2}M(\omega)\ud S(\omega)/4\pi$ denotes the average of a function $M(\omega)$ on the sphere.
We now show this is indeed always satisfied and without further conditions $\overline{M_1}=\overline{M}/4$.

In fact, we compute from  \eqref{eq:Psi:2} the average of $\Psi_2[M]$ on $\mathbb{S}^2$:
  \begin{equation}
    \label{eq:Psi:2:avg}
    r^2 \overline{\Psi_2}(t,r)=-\frac{1}{(4\pi)^2}\int_{\mathbb{S}^2}
    \int_{\mathbb{S}^2}\frac{M(\sigma)\,\ud S(\sigma)\,\ud S(\omega)}{(t/r-\langle \sigma,\omega\rangle)^2}=-\frac{\overline{M}}{(t/r)^2-1},
  \end{equation}
  where $\overline{M}=\int_{\mathbb{S}^2}M(\omega)\ud S(\omega)$ denotes the average of $M(\omega)$.

Hence, arranging the terms as in \eqref{eq:homoexpansion2},
  \begin{equation}
    r^2\overline{\Psi_2[M]}=\frac{1}{2}\frac{\overline{M}r}{r-t}
    +\frac{1}{2}\frac{\overline{M}}{t/r+1},
  \end{equation}
  and we have in particular that $\overline{M_0}=\overline{M}/2$, and $\overline{M_1}=\overline{M}/4$, which also implies that $\overline{M_1^\ast}=0$.

In summary, for the choice of $\mathcal{F}_0$, first solve
\begin{equation}
  \label{eq:triangle:P}
  \triangle_\omega \mathcal{P}(\omega)=-2M_1(\omega)-\int_{-\infty}^{\infty}\mathcal{R}_1(q,\omega)\ud q\qquad \text{: on } \mathbb{S}^2 .
\end{equation}
We have seen that this is always possible because the integral on $\mathbb{S}^2$ on the right hand side vanishes identically. Moreover, $M_1(\omega)$ and $\mathcal{R}_1(q,\omega)$ are determined from $n$ alone.
Then given the solution  $\mathcal{P}(\omega)$ of \eqref{eq:triangle:P} with vanishing mean determined from $n(q,\omega)$ in this way, we can choose $\mathcal{F}_0$ freely up to the condition that for some constant $C$,
\begin{equation}
  \int_{-\infty}^{+\infty}\mathcal{H}_0(q,\omega) \ud q =\mathcal{P}(\omega)+C\,.
\end{equation}

\subsection{The asymptotic solution for radiation fields decaying in the exterior.}
\label{sec:approximate}

We are now ready to define the asymptotic solution, and prove the relevant error estimates for the proof of Theorem~\ref{thm:n}; for an overview of the proof see also the comments following the statement of the theorem on page~\pageref{thm:n}.

We have seen that
\begin{equation} \label{eq:Box:Psi:rad}
\Box \Psi_{rad}+ \frac{\!n(r\!-\!t,\omega)\!\!}{r^2}=R_{rad}\,,
\end{equation}
where by \eqref{eq:Rrad}, and \eqref{eq:H11}  and \eqref{eq:H11:estimate},
\begin{equation}
|R_{rad}|\les \ln{\!\Big|\frac{2\,r}{\tminusr}\Big|}
\frac{\langle q\rangle \chi_{q\leq 0}+\langle q\rangle^{-2\gamma}}{r^4}.
\end{equation}

This shows that the expansion  $\Psi_{rad}$ of \eqref{eq:secondorderapproximatesource2} is a good approximate solution for \eqref{eq:source} close the light cone where $r\sim t$, but when $r-t\sim - r/2$, then $R_{rad}\sim t^{-3}$ which corresponds to a remainder solution of the wave equation decaying only like $t^{-1}$ and which would cause a change to the radiation field.

 On the other hand, we have seen in Section~\ref{sec:interior} that the homogeneous asymptotics
\begin{equation}
    \Psi_{hom}=\Psi_1[N]+\Psi_2[M],
  \end{equation}
is  a solution to the homogeneous equation in the interior away from the light cone, and since $n(q,\omega)$ is decaying in $q$ this is a better approximation for \eqref{eq:source} away from the lightcone, in particular when $r\!-\!t\!\sim\! - r/2$.
In fact,
  \begin{equation}
    \Box\Psi_{hom}+N(\omega)\delta(r-t)/r^2-M(\omega)\delta'(r-t)/r^2=0\,.
  \end{equation}

    Moreover, we have seen in Section~\ref{sec:matching} that while $\Psi_1[N]$ and $\Psi_2[M]$ are determined by $N$, and $M$, respectively, the radiation fields $\mathcal{F}_I$ in the expansion $\Psi_{rad}$ are chosen so that $\Psi_{rad}$ matches the homogeneous asymptotics near the lightcone up to second order.

Therefore we transition between the two approximate solutions in a region where
$r-t\sim - r^a$, for some $0<a<1$ to be determined.
Let $\chi(s)=1$ when $s\geq -1/4$ and $\chi(s)=0$ when $s\leq -3/4$ and let
$\chi_a(t,x)=\chi\big( (r-t)/r^a\big)$ and set our asymptotic solution to be
\begin{equation}\label{eq:Psi:asym}
\Psi_{asym}=(1-\chi_a) \Psi_{hom}+\chi_a \Psi_{rad}\,.
\end{equation}
Then
 \begin{equation}
 \Box \Psi_{asym}+\chi\frac{\!n(r\!-\!t,\omega)\!\!}{r^2}=R_{asym},
\end{equation}
where
\begin{equation}
  \label{eq:R:asym:int}
 R_{asym}=(\chi-\chi_a)\frac{n}{r^2}+\chi_a R_{rad}
 +(\Box \chi_a) \Psi_{diff}+2 Q(\pa\chi_a,\pa\Psi_{diff}).
\end{equation}
Here $Q(\pa\psi,\pa\phi)=m^{\alpha\beta}\pa_\alpha\psi\,\pa_\beta\phi$, and we set
 \beq \Psi_{diff}=\Psi_{rad}-\Psi_{hom}\,.\eq
 Here we pick $\chi=\chi(r/\langle t-r\rangle)$ and $\chi(s)=1$ when $s>3/4$ and $\chi(s)=0$, when $s<1/2$.

 Subtracting the expansions \eqref{eq:homoexpansion} and \eqref{eq:homoexpansion2} from the expansion
 \eqref{eq:secondorderapproximatesource2} we get\footnote{An additional error term is generated by replacing $t-r$ by $\langle t-r\rangle$ in $\Psi_{rad}$. Since $\ln\tfrac{t-r}{\langle t-r\rangle}\sim -\tfrac{1}{2}(t-r)^{-2}$, $t>r$, decays, these terms can be absorbed in the lower order radiation fields.}
\beq \label{eq:secondorderapproximatesource2H}
\Psi_{dif{}_{\!}f}(t,{}_{\!}r\omega)\!
=\ln{\!\Big|\!\frac{2\,r}{\tminusr\!}\Big|}
\frac{\mathcal{H}_{\!01\!}(r\shortminus t,\omega{}_{\!}) \! \! }{r}+\frac{\!\mathcal{H}_{\!0\!}(r\shortminus t,\omega{}_{\!})\!\!}{r}
+\ln{\!\Big|\!\frac{2\,r}{\tminusr\!}\Big|}
\frac{\mathcal{H}_{\!11\!}(r\shortminus t,\omega{}_{\!}) \! \! }{r^2}+\frac{\!\mathcal{H}_{\!1\!}(r\shortminus t,\omega{}_{\!})\!\!}{r^2}
+O\Big(\!\frac{(r\shortminus t)^2\!\!\!}{\,r^3}
\ln{\!\Big|\frac{ 2\, r}{t\shortminus r}\Big|}\Big)\!,
\eq
where $\mathcal{H}_{01}$ is given by \eqref{eq:H01}, and $\mathcal{H}_0$ by \eqref{eq:H0}; moreover,  $\mathcal{H}_{11}$ is given by \eqref{eq:H11}, and $\mathcal{H}_1$ by \eqref{eq:H1}.
The point is that as a result of the matching the radiation fields $\mathcal{H}_I$ all decay in the exterior \emph{and} interior. In fact, by the estimates in Section~\ref{sec:matching}, see  \eqref{eq:angulardecay}, \eqref{eq:F0decay}, \eqref{eq:H11:estimate}, and \eqref{eq:H1:estimate},  we have
\beq\label{eq:H:estimate}
|\mathcal{H}_0|\!+\!\langle q\rangle^{\!-1} |\mathcal{H}_1|
\!+\!|\mathcal{H}_{01}|\!+\!\langle q\rangle^{\!-1} |\mathcal{H}_{11}|
+\langle q\rangle\bigtwo(|\mathcal{H}_{0}^{\,\prime}|\!+\!\langle q\rangle^{\!-1} |\mathcal{H}_{1}^{\,\prime}|
\!+\!|\mathcal{H}_{01}^{\,\prime}|\!+\!\langle q\rangle^{\!-1} |\mathcal{H}_{11}^{\,\prime}|\bigtwo) \les \langle q\rangle^{\!-\kappa}.
\eq

  Let us now return to \eqref{eq:Psi:asym}.
  First note that in the interior where $\chi_a=0$, we have $\Psi_{asym}=\Psi_{hom}$, and $\Box\Psi_{hom}=0$.
  Secondly when $\chi_a=1$, namely across the lightcone and in the exterior, we have $\Psi_{asym}=\Psi_{rad}$ and $\Psi_{rad}$ satisfies \eqref{eq:Box:Psi:rad}.
  In view of the assumption \eqref{eq:ndecay} on $n$,
\beq\label{eq:Rrada}
\big|\chi_a R_{rad}\big|
\les  \frac{\ln r }{r^{4}}\,\big(|q|\, \chi_a^{-}
+\langle q\rangle^{-2\gamma}\,\chi_{q\geq 0}\big)\,,
\eq
where we denote by $\chi_a^{-}=\chi((r-t)/(4r^a))\chi_{q\leq 0}$, so that $\chi_a^{-}=1$ for $-r^a<q<0$.

 We also need to estimate the difference $\Psi_{diff}$ for  $r-t\sim -r^a$.
By \eqref{eq:secondorderapproximatesource2H}, and \eqref{eq:H:estimate} we have
\begin{gather}
|\Psi_{dif{}_{\!}f}(t,r\omega{}_{\!})|\les r^{-1-\kappa a}\:\ln r\,, \qquad
|(\pa_t\!+\pa_r\!) \Psi_{dif{}_{\!}f}(t,r\omega{}_{\!})|\les r^{-2-\kappa a}\:\ln r,\\
|(\pa_t{}_{\!}-\pa_r\!) \Psi_{dif{}_{\!}f}(t,r\omega{}_{\!})|\les r^{-1-a(\kappa+1)}\:\ln r
\end{gather}
and
\beq
\big|\Box \chi_a \big|\les r^{-1-a}, \qquad
|(\pa_t+\pa_r)\chi_a|\les r^{-1},\qquad
|(\pa_t-\pa_r) \chi_a|\les r^{-a} .
\eq
Hence
\beq
\big|\Box \chi_a \big||\Psi_{dif{}_{\!}f}(t,r\omega)|\les r^{-2-a(1+\kappa)} \:\ln r\,,
\eq
and we also have
\beq
\big| Q(\pa\chi_a,\pa\Psi_{dif{}_{\!}f})\big|
\les \big|(\pa_t\!+\!\pa_r) \chi_a\big|\, \big|(\pa_t\!-\!\pa_r) \Psi_{diff}\big|
+\big|(\pa_t\!-\!\pa_r) \chi_a\big|\, \big|(\pa_t\!+\!\pa_r) \Psi_{diff}\big|
\les r^{-2-a(1+\kappa)}\:\ln r\,.
\eq
Since we also have that $|q|\sim r^a$ in the support of $\partial \chi_a$ we can estimate
\beq
\big|\Box \chi_a \big||\Psi_{dif{}_{\!}f}(t,r\omega)|
+\big| Q(\pa\chi_a,\pa\Psi_{dif{}_{\!}f})\big|\les \frac{\ln r }{r^{4}}\,|q|\, \chi_a^{-} r^{2-a(2+\kappa)}
\eq
This is equal to the first term in \eqref{eq:Rrada} if we pick $a$ so that $a(2+\kappa)=2$ and there is no advantage to picking a smaller or larger $a$ so that is the $a$ that we will pick.
Moreover,  if $\gamma\geq 1/2$ then $\kappa=2$ so $a=1/2$.
In conclusion,
\beq\label{eq:Rasymest}
\big|R_{asym}\big|
\les  \frac{\ln r }{r^{4}}\,\big(|q|\, \chi_{a}^{-}
+\langle q\rangle^{-2\gamma}\,\chi_{q\geq 0}\big)+\frac{1}{t^2\langle q\rangle^{2+2\gamma}}\chi_{ -r/2<q<-r^a},\qquad \text{if}\quad a\geq 2/(2+\kappa).
\eq

\subsection{Estimate of the remainder}
\label{sec:remainder}

For the proof of Theorem~\ref{thm:n}, it remains to show that $R_{asym}$ produces a sufficiently fast decaying remainder.
Let $E_-=\delta(t^2-|x|^2)H(-t)/2\pi$ be the backward fundamental solution of $-\Box$. Then in view of \eqref{eq:Rasymest} the convolution
\beq
\psi_{rem}=E_-*R_{asym}
\eq
is well defined,  and   we show in Appendix~\ref{sec:RadialEstimates} that
\begin{equation}
  |\psi_{rem}|\lesssim \frac{\log\tplusr}{\tplusr^2}
\end{equation}
In particular, the radiation field of  $\psi_{rem}$  vanishes.

\subsection{Matching of the interior asymptotics for sources with cubic terms.}
\label{sec:matching:cubic}

As described in Section~\ref{sec:cubic:intro} we can proceed similarly for the proof of Theorem~\ref{thm:m}. We focus here on matching the interior homogeneous asymptotics for the cubic equation to the radiation field at null infinity, and deriving the compatibility condition. The asymptotic solution can then be defined and estimated as in Sections~\ref{sec:approximate} and~\ref{sec:remainder} above.

For \eqref{eq:highersourcbegining}, we have the higher order expansion near null infinity
\beq \label{eq:approximatesource2higherorder}
\phi(t,r\omega)
\sim \!\Psi_{{}_{\!}rad}^3(r\!-\!t,\omega,{}_{\!}1{}_{\!}/r{}_{\!}){}_{\!}=
\,\frac{\mathcal{G}_{0}(r\!-\!t,\omega) \! }{r}+\frac{\mathcal{G}_1(r\!-\!t,\omega)\!}{r^2},\qquad\text{as}\quad t\to\infty,\quad \text{while}\quad r\!\sim t\,.
\eq
By the calculations in Section~\ref{sec:higherordernull}  up to terms of order $1/r^4$,
\beq
\Box \Psi_{rad}^3\sim-\frac{\!2\mathcal{G}_{\!1}^{\,\prime}(r\shortminus t,\omega{}_{\!}) \! \!}{r^3}+\frac{\triangle_\omega \mathcal{G}_0(r\!-\!t,\omega)\!}{r^3}\,,
\eq
hence we must have
\beq\label{eq:de:G1}
-2\mathcal{G}_{\!1}^{\,\prime}(q,\omega{}_{\!}) +\triangle_\omega \mathcal{G}_0(q,\omega)
= -m(q,\omega).
\eq

We can read off from \eqref{eq:k3homsolformulaasym} the matching conditions that have to be satisfied
  \begin{equation}
    \label{eq:G0:limit}
    \lim_{q\to -\infty} q\,\mathcal{G}_0(q,\omega)=-M_0\,,
  \end{equation}
  and so let us define
  \begin{equation}
    \widetilde{\mathcal{G}}_0(q,\omega)=\mathcal{G}_{0}(q,\omega)- M_{0}\langle q\rangle^{-1}\chi_{q\leq  0}\,.
  \end{equation}
  We will  assume that
\beq \label{eq:G0decay}
 \big|(\langle q\rangle \pa_q)^k\pa^\alpha_\omega \widetilde{\mathcal{G}}_0\big|
 \les \langle q\rangle^{-\kappa},\,\qquad 1<\kappa\leq 1+\min{(1,2\gamma)}\,.
\eq
Secondly $ \mathcal{G}_1(q,\omega)$ has to satisfy
\beq
\lim_{q\to -\infty}  \mathcal{G}_1(q,\omega)=-M_0/2
\eq
and it follows from integrating \eqref{eq:de:G1} that
\begin{equation}
  \label{eq:G1:integrate}
  \begin{split}
    M_0/2+\mathcal{G}_{1}(q,\omega)=&\frac{1}{2}\int_{-\infty}^q \Bigl[\triangle_\omega \widetilde{\mathcal{G}}_0(s,\omega)+m(s,\omega)\Bigr]\ud s\\
    =& -\frac{1}{2}\int_q^\infty\Bigl[\triangle_\omega\widetilde{\mathcal{G}}_0(s,\omega)+m(s,\omega)\Bigr]\ud s+\frac{1}{2}\int_{-\infty}^{+\infty}\Bigl[\triangle_\omega \widetilde{\mathcal{G}}_0(s,\omega)+m(s,\omega)\Bigr]\ud s ,
  \end{split}
\end{equation}
 where we used  that $\triangle \mathcal{G}_0=\triangle \widetilde{\mathcal{G}}_0$ because the limit \eqref{eq:G0:limit} does not depend on $\omega$.
  Assuming sufficient decay
   \begin{equation}
    \big|(\langle q\rangle \pa_q)^k\pa^\alpha_\omega m(q,\omega)\big|\les  \langle q\rangle^{-1-2\gamma}\,,
  \end{equation}
 we see from the first equality in \eqref{eq:G1:integrate} that
  \begin{equation}
    \widetilde{\mathcal{G}}_1(q,\omega)=\mathcal{G}_{1}(q,\omega) +(M_0/2)\chi_{q\leq 0} ,
  \end{equation}
  decays in the interior, $\widetilde{\mathcal{G}}_1\les \langle q\rangle^{-\kappa+1}$ as $q\to-\infty$ by \eqref{eq:G0decay}, in fact
  \beq 
 \big|(\langle q\rangle \pa_q)^k\pa^\alpha_\omega \widetilde{\mathcal{G}}_{1}(q,\omega)\big|
 \les \langle q\rangle^{-\kappa+1} \,.
\eq

  However, from the second equality in \eqref{eq:G1:integrate} we see that $\mathcal{G}_1$, or $\widetilde{\mathcal{G}}_1$, only  decays in the exterior and satisfies
\beq
\lim_{q\to \infty}  \mathcal{G}_1(q,\omega)=0 ,
\eq
provided the following  condition is satisfied
\begin{equation}
  M_0/2=\frac{1}{2}\int_{-\infty}^{+\infty}\Bigl[\triangle_\omega \widetilde{\mathcal{G}}_0(s,\omega)+m(s,\omega)\Bigr]\ud s .
\end{equation}
This can be viewed as an additional condition on $\widetilde{\mathcal{G}_0}(q,\omega)$:
\begin{equation}
\triangle_\omega \Bigl[\int_{-\infty}^{+\infty}  \widetilde{\mathcal{G}}_0(q,\omega)
\, dq\Bigr] =-M(\omega)+ M_0\,,\qquad M(\omega)=\int_{-\infty}^{+\infty}m(q,\omega)\ud q .
\end{equation}
Note that this equation has an integrability condition
\begin{equation}
  0=\frac{1}{4\pi}\int_{\mathbb{S}^2}(-M(\omega)+M_0)\ud S(\omega)=-M_0+M_0=0 ,
\end{equation}
which is indeed always  satisfied.

In conclusion, for given $m$, first solve the equation
\begin{equation}
  \triangle_{\omega}\mathcal{Q}(\omega)=-M(\omega)+M_0\,,
\end{equation}
then $\widetilde{\mathcal{G}_0}$ is free to choose
up to the condition that
\begin{equation}
  \int_{-\infty}^{+\infty}  \widetilde{\mathcal{G}}_0(q,\omega)\, dq=\mathcal{Q}(\omega)\,.
\end{equation}

\section{Homogeneous solutions to the homogeneous equation in the exterior of the lightcone}
\label{sec:hom:ext}

In this section we extend the scattering construction to the situation when
\begin{equation}
  \lim_{q\to\infty}\mathcal{F}_{01}(q,\omega)=N_{01}(\omega)\qquad \lim_{q\to\infty}\mathcal{F}_{0}(q,\omega)=N_0(\omega)\,,
\end{equation}
by relating these limits to homogeneous solutions of the wave equation in the exterior of the light cone.

In particular we give the proofs of Theorem~\ref{thm:hom:N} and Corollary~\ref{thm:hom:M} as outlined in Section~\ref{sec:hom:ext:intro}.

\subsection[Solutions to the wave equation with homogeneous data of degree $-1$.]{Solutions to the wave equation with homogeneous data of degree $-1$ and exterior homogeneous asymptotics.}
\label{sec:hom:ext:one}

 We first derive an expansion for the exterior homogeneous solution
\beq \label{eq:hom:ext:one}
\Box \,\phi_1=0,\quad |x|>t,\qquad \phi\,\Big|_{t=0}=f=M(\omega)/r,\quad \pa_t\phi\, \Big|_{t=0}=g=N(\omega)/r^2\,,
\eq
close to the light cone. In spherical symmetry, when $M$ and $N$ are constants, we have the explicit solution
\begin{equation}
\label{eq:homogeneous:ex}
\phi_1 =\frac{M}{r}+\frac{N}{2r} \ln{\!\Big|\!\frac{\,r\!+\!t}{r\shortminus t}\Big|}\,.
\end{equation}

We will first compute the expansion of \eqref{eq:hom:ext:one} in $(r-t)/r$ and $r$, as $r\to\infty$, and then take the limit $r/t\to 1$.
We will compute the coefficients in the expansion explicitly and show that the transformation that maps the data $(M,N)$ to the leading order coefficients is invertible.
This allows us to prove that, given any smooth functions $N_0(\omega)$ and $N_{01}(\omega)$ we can find smooth data $M(\omega)$ and $N(\omega)$ such that
close to the light cone
 \beq
\phi_1 \sim\frac{N_0(\omega)}{r}+\frac{N_{01}(\omega)}{r} \ln{\!\Big|\!\frac{\,r\!+\!t}{r\shortminus t}\Big|}\qquad (r/t>1,r\to\infty).
\eq

\subsubsection{Homogeneous initial data for the time-derivative.}
\label{sec:hom:ext:one:time}
In Appendix~\ref{sec:formula:ext:hom} we derive that the  solution to the initial value problem with homogeneous initial data,
\beq
\Box \,\phi=0,\quad |x|>t,\qquad \phi\,\Big|_{t=0}=f=0,\quad \pa_t\phi\, \Big|_{t=0}=g=N(\omega)/r^2\,,
\eq
is given by
\begin{equation}
  \label{eq:phi:Psi10:ext}
  \phi=\Psi_{1,0}^{\text{ext}}[N]=\frac{1}{r}\int_{z_0}^1\frac{N(\omega,z)\ud z}{\sqrt{z^2-z_0^2}}\,, \quad r>t\,,\quad z_0=\sqrt{1-(t/r)^2}\,.
\end{equation}
Here $N(\omega,z)$ is the average of $N(\sigma)$ over the circle $\langle \sigma,\omega\rangle=z$ as defined in \eqref{eq:avg},
\begin{equation}
N(\omega,z)\!=\!\int_{\langle \sigma,\omega\rangle =z}\!\!\!\!\! \!\! \!  N(\sigma) \,\ud s(\sigma) \Big/\!\!\int_{\langle \sigma,\omega\rangle =z} \!\!\!\! \!\! \!  \ud s(\sigma) \,.
\end{equation}

It turns out that the transformation $N\mapsto\Psi_{1,0}[N]$ in \eqref{eq:phi:Psi10:ext} invertible:  The source function $N(\omega)$ can be determined from the homogeneous solution by an argument that we include for completeness in Appendix~\ref{sec:abel}.

In fact, it suffices to know the leading order terms in the expansion of the homogeneous solution near the lightcone to determine  the initial data $N$ and hence the entire homogeneous solution $\Psi_{1,0}^{\text{ext}}$ in the exterior.
In Appendix~\ref{sec:ext:hom:expansion:leading} we derive that the leading order terms in the expansion are given by
 \begin{equation} \label{eq:Psi10:exp}
\Psi^{\,\text{ext}}_{1,0}[N]\sim \frac{1}{2r}\ln{\frac{2r}{r\shortminus t}}\,N(\omega,0) +\frac{1}{r}\widetilde{N}(\omega),
 \end{equation}
 where
 \beq
 N(\omega,0) = \frac{1}{2\pi} \int_{\langle \sigma,\omega\rangle =0}\!\!\!\!\! \!\! \!  N(\sigma) \,\ud s(\sigma),\quad\text{and}\quad
 \widetilde{N}(\omega)=\int_{0}^1 \frac{ N(\omega,z ){}_{\!}-N(\omega,0)}{z}\ud z .
 \eq
 The function $N(\omega,0)$ is  equal to the  \textbf{Funk transform} $\mathcal{F}[N](\omega)$,
 \begin{equation}
   \mathcal{F}[N](\omega)=   \frac{1}{2\pi} \int_{\langle \sigma,\omega\rangle =0}\!\!\!\!\! \!\! \!  N(\sigma) \,\ud s(\sigma)\,,
 \end{equation}
a function on the sphere which associates to each point $\omega$ the integral of $N$ over the  circle $C(\omega)=\{\sigma:\langle \sigma,\omega\rangle=0\}$.
 This is the spherical analogue of the Radon transform; see e.g.~\cite{G76,H99}.
 The  map $N\to\mathcal{F}[N]$ has a kernel consisting of all odd functions, namely those with $N(\omega)+N(-\omega)=0$. For symmetric functions $N(\omega)$ on the sphere, namely even functions with $N(\omega)=N(-\omega)$, the Funk transform  is invertible, and there is a reconstruction formula for $N$ in terms of its spherical averages; see Appendix~\ref{sec:transform:sphere} and \ref{sec:funk:geometric}.

 Note that for odd functions on the sphere,
 \begin{equation}
  \widetilde{N}(\omega)=\mathcal{S}[N](\omega)\,,\qquad N(\omega)=-N(-\omega)\,,
\end{equation}
where
\begin{equation}
  \mathcal{S}[N](\omega)=\frac{1}{2}\int_{-1}^1 \frac{N(\omega,z)-N(\omega,0)}{z} \ud z\,.
\end{equation}
The transformation $N\mapsto\mathcal{S}[N]$ has a kernel consisting of all even functions, and $\mathcal{S}$ maps odd functions to odd functions on the sphere. Moreover the map $N\mapsto\mathcal{S}[N]$ is invertible on odd functions as discussed in Appendix~\ref{sec:transform:sphere}.
In summary,
 \begin{equation} \label{eq:Psi10:Funk}
\Psi^{\,\text{ext}}_{1,0}[N](t,r\omega) \sim \frac{1}{2r}\ln{\frac{2r}{r\shortminus t}}\,\mathcal{F}[N](\omega) +\frac{1}{r}\big(\widetilde{N_+}(\omega)+\mathcal{S}[N](\omega) \big),
 \end{equation}
where we have decomposed $N$ into its even and odd part:
\begin{equation}
  N=N_++N_-\,, \qquad N_+(\omega)=\frac{1}{2}\bigl(N(\omega)+N(-\omega)\bigr)\,,\qquad N_-(\omega)=\frac{1}{2}\bigl(N(\omega)-N(-\omega)\bigr)\,.
\end{equation}

Note that if $N(\omega)=1$, then $\mathcal{F}[N]=1$, and $\widetilde{N_+}+\mathcal{S}[N]=0$, in agreement with the explicit example of \eqref{eq:homogeneous:ex}.

\subsubsection{Homogeneous initial data for the solution.}
\label{sec:hom:ext:deg:one}

In Appendix~\ref{sec:formula:ext:hom} we also derive that the  solution to the initial value problem
\beq
\Box \,\phi=0,\quad |x|>t,\qquad \phi\,\Big|_{t=0}=f=\frac{M(\omega)}{r},\quad \pa_t\phi\, \Big|_{t=0}=g=0\,,
\eq
is given by
\begin{equation}
  \phi=-\frac{x^i}{t}\Psi_{1,0}^{\text{ext}}[M_i]\,,\qquad \text{ where }\partial_i\bigl(M(\omega) r^{-1}\bigr)=M_i(\omega) r^{-2}\,.
\end{equation}
In view of \eqref{eq:Psi10:exp} we can infer that
\begin{equation}
  \phi  \sim - \frac{1}{2t}\ln{\frac{2r}{r\shortminus t}}\, \omega^i\mathcal{F}[M_i](\omega) -\frac{1}{t}\omega^i \widetilde{{M_i}}(\omega)\,.
\end{equation}

The transformation $M\mapsto \mathcal{G}[M]$ given by
\begin{equation}
\mathcal{G}[M](\omega)=\omega^i \mathcal{F}[M_i](\omega)
 \end{equation}
 vanishes on even functions, and is invertible as a transformation from odd functions to odd functions on the sphere.
In fact, an explicit reconstruction formula is given in \cite{BEGM03}.

 We may express this transformation more explicitly using that
 \begin{equation}
  \partial_i\bigl(M(\omega)r^{-1}\bigr)=-\frac{M(\omega)}{r^{2}}  \frac{x^i}{|x|}+\frac{1}{r^2}\nang_i M(\omega)\,,\qquad M_i(\omega)=-M(\omega)\omega^i+\nang_i M(\omega)\,,
\end{equation}
where $\nang_i$ are the cartesian components of the gradient on the standard sphere $\mathbb{S}^2$,
so $\Pi_i^j\nabla_j M=\frac{1}{r} \nang_iM$, where $\Pi_i^j(\omega)=\delta_i^j-\omega^i \omega^j$ is the projection to the tangent plane to $\mathbb{S}^2$.

Let us denote for brevity by
\begin{equation}
  \widecheck{M}_i(\sigma)=M(\sigma)\sigma^i\,,\qquad\text{so }M_i=-\widecheck{M}_i+\nang_iM\,.
\end{equation}
Since $\mathcal{F}[\widecheck{M}_i](\omega)=\int_{\langle \sigma,\omega\rangle =0} M(\sigma)\sigma^i \ud s(\sigma)$ are the components of a vector orthogonal to $\omega$, we see that $\omega^i\mathcal{F}[\widecheck{M}_i]=0$
and hence it follows  that
\begin{equation}
  \mathcal{G}[M](\omega)=\omega^i\mathcal{F}[M_i](\omega)=\omega^i  \mathcal{F}[\nang_i M](\omega)=\frac{1}{2\pi}\int_{\langle \sigma,\omega\rangle=0} \langle \omega,\nang M(\sigma)\rangle \ud s(\sigma)=\frac{1}{2\pi}\int_{\langle \sigma,\omega\rangle=0}\frac{\partial M}{\partial n} \ud s(\sigma)\,.
\end{equation}
Hence  $\mathcal{G}[M](\omega)$ can be viewed as the Funk transform of the normal derivative \cite{G76}.
 Alternatively the transformation can also be viewed as a Funk transform of the 1-form ${}^\ast\ud M$,
 and a characterisation of the kernel has been given in \cite{G76, M78}.

We discuss in Appendix~\ref{sec:transform:sphere} the relevant statements about the transformation $M\mapsto\mathcal{G}[M]$ which are proven in  \cite{BEGM03}, among which the most important states that for odd functions on the sphere
 \begin{equation}
   \label{eq:GS}
   \mathcal{G}\circ\mathcal{S}=\mathrm{id}\,,\qquad    \mathcal{S}\circ\mathcal{G}=\mathrm{id}\,.
 \end{equation}

 Next consider the transformation
 \begin{equation}
   \omega^i\widetilde{M_i}(\omega)=\int_{0}^1 \frac{\omega^i M_i(\omega,z ){}_{\!}-\omega^i M_i(\omega,0)}{z}\ud z .
 \end{equation}
By linearity of the average we have  $\omega^iM_i(\omega,z)=-\omega^i\widecheck{M_i}(\omega,z)+\omega^i(\nang_i M)(\omega,z)$, and we compute
 \begin{equation}
   \omega^i \widecheck{M}_i(\omega,z) = \omega^i \!\int_{\langle \sigma,\omega\rangle =z}\!\!\!\!\! \!\! \!  M(\sigma) \sigma^i \,\ud s(\sigma) \Big/\!\!\int_{\langle \sigma,\omega\rangle =z} \!\!\!\! \!\! \!  \ud s(\sigma) = z M(\omega,z)\,.
 \end{equation}
 and
 \begin{equation}
  \label{eq:91}
\omega^i (\nang_i M)(\omega,z)= \int_{\langle \omega,\sigma\rangle=z}\langle \omega, \nang M(\sigma)\rangle \ud s(\sigma) \Big/2\pi \sqrt{1-z^2}= \frac{1}{2\pi}\int_{\langle \omega,\sigma\rangle=z} \frac{\partial M}{\partial n}(\sigma)\ud s(\sigma)
\end{equation}
where we used that on the circle $\langle \omega,\sigma\rangle=z$ we have $\omega=z \sigma+\sqrt{1-z^2}\,n$; here $n$ is the  the unit normal to the circle $\langle \omega,\sigma\rangle=z$, $\langle n, n\rangle=1$, tangential to the sphere $\langle \sigma,n\rangle=0$, with $\langle n,\omega\rangle>0$.

\paragraph{Even case.} Suppose first that $M(\omega)=M(-\omega)$ is an even function on the sphere.
Then  in view of the formulas derived above  $\omega^iM_i(\omega,z)$ is an odd function in $z$, hence
 \begin{equation}
   \omega^i\widetilde{M_i}(\omega)=\frac{1}{2}\int_{-1}^1 \frac{\omega^iM_i(\omega,z)-\omega^i M_i(\omega,0)}{z} \ud z=\omega^i\mathcal{S}[M_i]\,.
 \end{equation}
It is shown in \cite{BEGM03} that the transformation
$M\mapsto\mathcal{T}[M]$ given by
  \begin{equation}
    \label{eq:78}
    \mathcal{T}[M](\omega)=-\omega^i\mathcal{S}[M_i](\omega)\,,
  \end{equation}
  sends even functions to even functions on the sphere, and vanishes on odd functions.
  We will discuss this transformation further in Appendix~\ref{sec:transform:sphere},
  but state already that it is proven in \cite{BEGM03} that on even functions $\mathcal{T}$ is the inverse of the Funk transform,
  \begin{equation}
    \label{eq:TF}
    \mathcal{T}\circ\mathcal{F}=\mathrm{id}\,,\qquad \mathcal{F}\circ\mathcal{T}=\mathrm{id}\,.
  \end{equation}

  \paragraph{Odd case.} For odd functions $M(\omega)=-M(-\omega)$, we see that $\omega^i M_i(\omega,z)$ is an even function in $z$, and
  \begin{equation}
    \omega^i M_i(\omega,0)=0\,.
  \end{equation}
Hence in this case
\begin{equation}
  \label{eq:70}
     \omega^i\widetilde{M_i}(\omega)=\int_{0}^1 \frac{\omega^i M_i(\omega,z ){}_{\!}}{z}\ud z = - \int_0^1 M(\omega,z)\ud z+\frac{1}{2\pi}\int_0^1 \frac{1}{z} \int_{\langle \omega,\sigma\rangle=z} \frac{\partial M}{\partial n}(\sigma)\ud s(\sigma) \ud z
   \end{equation}
   and we take the right hand side as the definition of a function  $\mathcal{U}[M](\omega)$.
   Note that $\mathcal{U}[M](-\omega)=-\mathcal{U}[M](\omega)$ is an odd function.

   In summary, given any function $M(\omega)$ we decompose into even and odd parts,
   \begin{equation}
     M(\omega)=M_+(\omega)+M_-(\omega)\,,
   \end{equation}
   and conclude that
   \begin{equation}
     \omega^i\widetilde{M_i}(\omega)=-\mathcal{T}[M_+](\omega)+\mathcal{U}[M_-](\omega)\,.
   \end{equation}
   and note that alternatively $\mathcal{T}[M_+]=\mathcal{T}[M]$ because odd functions are in the kernel of $\mathcal{T}$.

Therefore the leading orders in the expansion are
\begin{equation}
  \label{eq:Psi10:FunkInverse}
   \phi\sim -\frac{1}{2t}\ln{\frac{2r}{r\shortminus t}}\,\mathcal{G}[M] -\frac{1}{t}\big(\mathcal{U}[M_-]-\mathcal{T}[M] \big)\,.
 \end{equation}
 Note that if $M(\omega)\!=\!1$ then $\mathcal{G}[M]\!=\!0$, $\mathcal{U}[M_-]\!=\!0$, and $\mathcal{T}[M]\!=\!1$, which yields $\phi\sim 1/t$ in agreement with \eqref{eq:homogeneous:ex}.

 \subsubsection{Solvability of the inverse problem.}
In conclusion, given an expansion of the  form
\begin{equation}
  \phi\sim \frac{1}{2r}\ln{\frac{2r}{r\shortminus t}}\, N_{01}(\omega) +\frac{1}{r} N_0(\omega)
\end{equation}
the data of the corresponding homogeneous solution is determined as follows:
Given that for any data $M(\omega)$ and $N(\omega)$ the asymptotics of the homogeneous solution is given by  \eqref{eq:Psi10:Funk} and \eqref{eq:Psi10:FunkInverse}, we obtain the condition
\begin{equation}
  N_{01}=\mathcal{F}[N]-\mathcal{G}[M]=(N_{01})_++(N_{01})_-\,.
\end{equation}
Thus the even part of $N$ is determined by the even part of $N_{01}$, and the odd part of $M$ is determined by the odd part of $N_{01}$:
\begin{equation}
  \label{eq:NM:det}
  N_+=\mathcal{T}[N_{01}]\qquad M_-=-\mathcal{S}[N_{01}]\,.
\end{equation}

Using the formulas derived for the coefficients at order $r^{-1}$ in \eqref{eq:Psi10:Funk} and \eqref{eq:Psi10:FunkInverse}, we get the second condition
\begin{equation}
  \label{eq:N0:MN}
N_0(\omega)=  \widetilde{N_+}(\omega)+\mathcal{S}[N](\omega)-\mathcal{U}[M_-](\omega) +\mathcal{T}[M](\omega)\,.
\end{equation}
Since $N_+$ and $M_-$ are already determined by \eqref{eq:NM:det}, and we can view this as an equation for $N_-$ and $M_+$:
  \begin{equation}
    \mathcal{S}[N_-] +\mathcal{T}[M_+]=N_0-\widetilde{N_+}+\mathcal{U}[M_-]
  \end{equation}
  Since the image of $\mathcal{S}$ are odd functions, and the image of $\mathcal{T}$ are even functions, we obtain
  \beq
    \mathcal{S}[N_-] = (N_0)_-+\mathcal{U}[M_-],\qquad
    \mathcal{T}[M_+] = (N_0)_+-\widetilde{N_+}
  \eq
Both equations are invertible by \eqref{eq:GS} and \eqref{eq:TF} and thus $M_+$ and $N_-$ uniquely determined by
  \begin{equation}
    M_+=\mathcal{F}[{N_0}_+-\widetilde{N_+}],\qquad N_-=\mathcal{G}[(N_0)_-+\mathcal{U}[M_-]\,.
  \end{equation}

  Note that in the absence of logarithmic terms in the expansion, so for $N_{01}=0$, we have $N_+=M_-=0$ and the formulas reduces to:
  \begin{equation}
    M=\mathcal{F}[N_0],\qquad  N=\mathcal{G}[N_0]\,.
  \end{equation}

  \subsection[Solutions to the wave equation with homogeneous data of degree $-2$.]{Solutions to the wave equation with homogeneous data of degree $-2$ and exterior homogeneous asymptotics.}
\label{sec:hom:ext:two}.

  In spherical symmetry,
  \begin{equation}
    \label{eq:phi2:explicit}
    \phi_2=\frac{C}{r(r+t)}
  \end{equation}
  is an explicit solution to the wave equation in the exterior $|x|>t$, which is homogeneous of degree $-2$.

  More generally, these solutions can be obtained as solutions to the initial value problem in the exterior of the light cone, with homogeneous data of degree $-2$:
  \begin{equation}
    \label{eq:hom:ext:two}
\Box \,\phi_2=0,\quad |x|>t,\qquad \phi\,\Big|_{t=0}=K(\omega)/r^2\,,\quad \pa_t\phi\Big|_{t=0}=L(\omega)/r^3\,.
\end{equation}

We will show that it is always possible to find data $K$, and $L$, so that the expansion of $\phi_2$ near the lightcone matches prescribed asymptotics.

\subsubsection{Time-derivatives of homogeneous degree $-1$ solutions.}
\label{sec:hom:time}

In Section~\ref{sec:hom:ext:deg:one}  we have already found the expansion of solutions to \eqref{eq:hom:ext:one}.
Since the time-derivative $\psi=\partial_t\phi_1$ of the solution to \eqref{eq:hom:ext:one} solves the problem
  \begin{equation}
    \label{eq:hom:ext:onetime}
    \Box \psi=0\qquad |x|>t\,, \quad \psi\rvert_{t=0}=N(\omega)/r^2\,,\quad \partial_t\psi\rvert_{t=0}=(\triangle_\omega M)(\omega)/r^3\,,
  \end{equation}
  we obtain the expansion of the solutions to \eqref{eq:hom:ext:two},
  in the case $K(\omega)=N(\omega)$, and $L(\omega)=\triangle_\omega M(\omega)$,
  by taking the time-derivative of the  expansion  \eqref{eq:homoexpansion} near the lightcone,
\begin{equation}
\psi=\partial_t \phi_1 \sim N_{01}{}_{\!}(\omega{}_{\!})\frac{1}{r}\frac{1}{r-t}
-N_{11}(\omega) \frac{1}{r^2}\ln\frac{2r}{r-t}  +\frac{1}{r^2} \Bigl(N_{11}(\omega)-N_1(\omega)\Bigr)\,.
\end{equation}
where
\begin{equation}
  \label{eq:N01:rep}
    N_{01}(\omega)=  \frac{1}{2}\Bigl(\mathcal{F}[N]-\mathcal{G}[M]\Bigr),
\end{equation}
and  $N_{11}$ and $N_1$ satisfy \eqref{eq:homocond}, and \eqref{eq:N1}.

In other words, $\psi$ has an expansion of the form \eqref{eq:homoexpansion2}, where
\begin{align}
  \label{eq:32}
  M_0(\omega)=&N_{01}(\omega),\qquad
  M_{11}(\omega)=-N_{11}(\omega)=-\frac{1}{2}\triangle_\omega N_{01}(\omega)\,,\\
  M_1(\omega)=&N_{11}(\omega)-N_1(\omega)=-\frac{1}{2}\Delta_\omega N_{01}(\omega)-\frac{1}{2}\Delta_\omega N_0+\frac{1}{2}N_{01}(\omega) \,. \label{eq:M0:N0}
\end{align}
Thus given $M_0(\omega)$ we first determine the even part of $N$, and the odd part of $M$, so that \eqref{eq:N01:rep} holds. Then the expansion of \eqref{eq:hom:ext:two} with $K=N$ and $L=\triangle_\omega M$, has the prescribed leading order term $M_0(\omega)/r(r-t)$.

It remains to solve \eqref{eq:M0:N0} for any given function $M_1(\omega)$.
Since \eqref{eq:M0:N0} as an equation for $N_0$ has the integrability condition
\begin{equation}
  \overline{M_1}=\frac{1}{2}\overline{N_{01}}\,,
\end{equation}
we first solve \eqref{eq:M0:N0} with $M_1$ replaced by $M_1-\overline{M_1}+\frac{1}{2}\overline{N_{01}}$.
Then $N_0$ is determined, which in turn determines the even part of $M$, and the odd part of $N$  as in Section~\ref{sec:hom:ext:deg:one} so that \eqref{eq:N0:MN} holds.
Thus we have found $M$, and $N$, so that the solution to \eqref{eq:hom:ext:onetime} has the following expansion near the lightcone:
\begin{equation}
  \label{eq:36}
  \psi\sim \frac{\!M_{0}(\omega{}_{\!})\!}{r^2}\frac{r}{\!r\shortminus t\!}
-\frac{\triangle_\omega M_0}{2r^2}\ln{\!\Big|{}_{\!}\frac{ 2\, r}{t\shortminus r}{}_{\!}\Big|}
+ \frac{M_{1}(\omega)-C_1}{r^2}\,,\qquad C_1=\overline{M_1}-\frac{1}{2}\overline{M_0}\,.
\end{equation}

The remaining term $C_1/r^2$ is picked up by the solution \eqref{eq:phi2:explicit}.
Note that  \eqref{eq:phi2:explicit} corresponds to $K(\omega)=C$, and $L(\omega)=-C$, which is not in the range of \eqref{eq:hom:ext:onetime}, and  \emph{not} a time-derivative of a homogeneous solution.

\subsubsection{Explicit transformation formulas}

Alternatively, it is possible to compute the asymptotic form of solutions to \eqref{eq:hom:ext:two} directly.
Since the argument in Section~\ref{sec:hom:time} suffices for our purposes, we only give a brief presentation of the explicit formulas.

As seen in Appendix~\ref{sec:formula:ext:hom:lower},
the solution to \eqref{eq:hom:ext:two} in the exterior $|x|>t$ is given by
\beq 
\phi_2=-\Psi^{\,\text{ext}}_{2,1}[K]-x^i t^{-1}\Psi^{\,\text{ext}}_{2,0}[K_i]+\Psi^{\text{ext}}_{2,0}[L]\,,\qquad\text{where}\quad K_{i}(\omega)r^{-3}=\pa_i \big( K(\omega)r^{-2}\big).
\eq
where
\beq
\Psi^{\,\text{ext}}_{2,0}[L]=\frac{1}{2\pi( r^2-t^2)}\int_{\langle \omega,\sigma\rangle>z_0} \,\frac{  \langle \omega,\sigma\rangle L(\sigma )\, dS(\sigma)}{\sqrt{\langle \omega,\sigma\rangle^2-z_0^2}}\,,\quad \Psi_{2,1}^{\text{ext}}[K]=\frac{1}{t}\Psi_{1,0}^{\text{ext}}[K]\,, \quad z_0=\sqrt{1-(t/r)^2}\,.
\eq

The asymptotic expansion of $\Psi_{1,0}^{\text{ext}}$ is already computed in Section~\ref{sec:hom:ext:one:time}, hence
\begin{equation}
    \label{eq:42}
    \Psi_{2,1}^{\text{ext}}[K]=\frac{1}{t}\Psi_{1,0}^{\text{ext}}[K]\sim \frac{1}{2r^2}\ln\frac{2r}{r-t}\mathcal{F}[K]+\frac{1}{r^2}\widetilde{K}\,.
  \end{equation}
  Moreover for $\Psi_{2,0}^{\text{ext}}[L]$ we find in Appendix~\ref{sec:formula:ext:hom:lower:precise} that
       \begin{equation}
   \label{eq:Psi20:sim}
   \Psi_{2,0}^{\text{ext}}[L]\sim \frac{1}{2r}\frac{1}{r-t}\mathcal{M}_+[L](\omega)+\frac{1}{4 r^2}\ln\frac{2r}{r-t}\ \mathcal{G}[L](\omega)
   +\frac{1}{2r^2} \biggl[-L(\omega)+\widetilde{L_z}(\omega) +\frac{1}{2}\mathcal{G}[L](\omega)\biggr],
 \end{equation}
where
 \begin{equation}
{\mathcal M}_+[L](\omega)= \int_{\langle \sigma,\omega\rangle >0}\!\!\!\!\! \!\! \!  L(\sigma) \,dS(\sigma)=\int_0^1 L(\omega,z)\, dz
\end{equation}
is a function that assigns to each point $\omega$ the integral of $L$ over the upper half-sphere $\langle\sigma,\omega\rangle>0$,
and
\begin{equation}
  \widetilde{L_z}(\omega)=\int_0^1\frac{L_z(\omega,z)-L_z(\omega,0)}{z}\ud z\,.
\end{equation}

The transformation $L\mapsto \mathcal{M}_+[L]$ maps odd to odd functions, and it is an invertible map onto its image. This follows, for example, from the explicit form of the coefficients in \eqref{eq:Psi20:sim}, and \eqref{eq:M0cond}, which imply together
\begin{equation}
 -\triangle_\omega  \mathcal{M}_+[L]=\mathcal{G}[L]\,.
\end{equation}

Finally, it follows directly from \eqref{eq:Psi20:sim} that
     \begin{equation}
  \omega^i \Psi_{2,0}[K_i]\sim \frac{1}{2r}\frac{1}{r-t} \int_0^1 \omega^i K_i(\omega,z) \ud z+\frac{1}{4 r^2}\ln\frac{2r}{r-t}\ \omega^i \mathcal{G}[K_i](\omega)   +\frac{1}{2r^2} \mathcal{U}[K],
 \end{equation}
 where $\mathcal{U}$ is a transformation to be discussed below,
 and we compute
 \begin{equation}
    \int_0^1 \omega^i K_i(\omega,z) \ud z=-2\int_0^1 z K(\omega,z)+\int_0^1(1-z^2)\partial_z K(\omega,z)\ud z\,.
    =-K(\omega,0)=-\mathcal{F}[K](\omega)\,,
\end{equation}

In conclusion, the asymptotics to leading order reads:
\begin{equation}
  \phi_2\sim \frac{1}{2r^2}\frac{r}{r-t}\Bigl(\mathcal{F}[K](\omega)
  +\mathcal{M}_+[L](\omega)\Bigr).
\end{equation}

Thus given an expansion of the form \eqref{eq:homoexpansion2},
we see that the even part of $M_0$ determines the even part of $K$,
and the odd part of $M_0$ determines the odd part of $L$.

It remains to determine the even part of $L$, and the odd part of $K$ from $M_1$.
If $L$ is even, we find that
\begin{equation}
  \label{eq:L:two}
  -L(\omega)+\widetilde{L_z}(\omega)=\omega^i\mathcal{S}[L_i](\omega)\,.
\end{equation}
Moreover, in the case that $K$ is odd, we find after some computation that
\begin{equation}
  \label{eq:K:two}
    -2\widetilde{K}(\omega) +\mathcal{U}[K](\omega) =-2\mathcal{S}[K](\omega)-\omega^i\omega^j\mathcal{S}[K_{ij}](\omega)\,,\qquad \partial_{ij}^2\Bigl(\frac{K(\omega)}{r^2}\Bigr)=\frac{K_{ij}(\omega)}{r^4} \,.
\end{equation}

The transformations $\mathcal{F}$ and $\mathcal{G}$ that we have encountered in Section~\ref{sec:hom:ext:deg:one}  in the study of homogeneous degree $-1$ solutions
are associated in \cite{BEGM03} to transformations of homogeneous functions of degree $-1$, and $-2$.
We expect that the transformations $\mathcal{M}_+$, as well as the transformations given by the right hand sides of \eqref{eq:L:two} and \eqref{eq:K:two}, are geometrically associated to transformations of homogeneous functions of degree $-2$, and $-3$, and we expect,  in view of the results in Section~\ref{sec:hom:time}, that they are invertible (up to finite dimensional kernels).

\section{Matching of the asymptotics for the exterior homogeneous solutions to the radiation fields at null infinity for the homogeneous wave equation}
\label{sec:ext:homogeneous}

In this section we will construct solutions with prescribed radiation fields that match to homogeneous solutions of degree $-1$ and $-2$ in the exterior and the interior of the light cone.

In the following we will treat the homogeneous equation, $\Box\phi=0$,
and show  that given any homogeneous solutions $\Psi_1^{\text{ext}}$, and $\Psi_2^{\text{ext}}$ in the exterior $|x|>t$, of degree $-1$, and $-2$, respectively,
we can solve the scattering problem with prescribed radiation fields, for the homogeneous wave equation with these given exterior asymptotics.
This gives a proof of Theorem~\ref{thm:ext:scattering},
and in the process we will find a compatibility condition that replaces    \eqref{eq:F0:comp} for slowly decaying data.

\subsection{Higher order expansion in the wave zone towards null infinity.}
\label{sec:ext:rad}

Consider  the same expansion in the wave zone that we used in the presence of  a source \eqref{eq:secondorderapproximatesource2}:
\beq
\Psi_{\text{rad}}=\ln{\!\Big|\!\frac{2\,r}{\langle t\shortminus r{}_{\!}\rangle{}_{\!}}\Big|}
\frac{\mathcal{F}_{\!01\!}(r\!-\!t,\omega) \! \! }{r}+\frac{\!\mathcal{F}_{\!0}(r\!-\!t,\omega)\!\!}{r}
+\ln{\!\Big|\!\frac{2\,r}{\langle t\shortminus r{}_{\!}\rangle{}_{\!}}\Big|}
\frac{\mathcal{F}_{\!11\!}(r\!-\!t,\omega) \! \! }{r^2}+\frac{\!\mathcal{F}_{\!1\!}(r\!-\!t,\omega)\!\!}{r^2}.
\eq
For the homogeneous wave equation,  $\mathcal{F}_{01}$, $\mathcal{F}_1$ and $\mathcal{F}_{11}$ satisfy the ODEs \eqref{eq:F01eq}, \eqref{eq:F1eq}, and  \eqref{eq:F11eq}, with $n=0$.

In fact, $\mathcal{F}_{01}$  now satifies $\mathcal{F}_{01}'(q,\omega)=0$,
and it follows that for some function $N(\omega)$ independent of $q$,
\begin{equation}\label{eq:F01:const}
\mathcal{F}_{01}(q,\omega) =N(\omega)\,.
\end{equation}
It then follows from \eqref{eq:F11eq} that
\begin{equation} \label{eq:F11:lin}
\mathcal{F}_{11}(q,\omega)
=\triangle_\omega N(\omega)\, q/2+M(\omega),
\end{equation}
for some function $M(\omega)$ independent of $q$. Hence by \eqref{eq:F1eq}
\begin{equation}\label{eq:F1:M}
  \begin{split} 2\mathcal{F}_1^\prime(q,\omega)&=-\mathcal{F}_{01}(q,\omega)+\Delta_\omega\mathcal{F}_{01}(q,\omega)+\Delta_\omega \mathcal{F}_0(q,\omega)+2 (q/\langle q\rangle^2) \mathcal{F}_{11} \\
    &=-N(\omega)+\big(2-1/\langle q \rangle^2 \big)\Delta_\omega N(\omega)+\Delta_\omega \mathcal{F}_0(q,\omega)+2 (q/\langle q\rangle^2)M(\omega)\,.
  \end{split}
\end{equation}
Integrating this directly would  give a logarithmic term  $\ln{\langle q\rangle}$, which does not correspond to a homogeneous solution.
Instead this term will be included in the radiation field using the matching conditions below.

\subsection{Exterior solutions.}

We have seen in Section~\ref{sec:hom:ext} that the choice of a homogeneous solution  corresponds one to one to the choice of two functions on the sphere, which appear as leading order coefficients in the expansions towards the light cone.

In fact, for any smooth functions $N^{ext}_{01}(\omega)$ and $N^{ext}_0(\omega)$ there exists a homogeneoues degree $-1$ solution $\Psi_1^{\text{ext}}$ in the exterior, so that
\beq\label{eq:Psi:ext:one}
\Psi_1^{\text{ext}}=N^{ext}_{01}{}_{\!}(\omega{}_{\!})\frac{1}{r}\!\ln{\!\Big|\frac{ 2\,r}{t\shortminus r}\Big|}+N^{ext}_{0}(\omega{}_{\!})\frac{1}{r}
+N^{ext}_{{}_{\!}11{}_{\!}}(\omega{}_{\!})\frac{r\shortminus t\!}{\,r^2}\ln{\!\Big|\frac{ 2\, r}{t\shortminus r}\Big|}
+N^{ext}_{{}_{\!}1{}_{\!}}(\omega{}_{\!})\frac{r\shortminus t\!}{\,r^2}
+O\Big(\!\frac{(r\shortminus t)^2\!\!\!}{\,r^3}
\ln{\!\Big|\frac{ 2\, r}{t\shortminus r}\Big|}\Big) ,
\eq
where $N^{ext}_{11}(\omega)$ and $N^{ext}_1(\omega)$ are determined from $N^{ext}_{01}(\omega)$ and $N^{ext}_0(\omega)$:
\beq\label{eq:homocond:bar}
2N^{ext}_{11}(\omega)=\triangle_\omega N^{ext}_{01}(\omega),\qquad
    2N^{ext}_1(\omega)=2\Delta_\omega N^{ext}_{01}(\omega)+\Delta_\omega N^{ext}_0-N^{ext}_{01}(\omega) \,.
  \end{equation}

We have also seen that for any smooth functions $M^{ext}_{0}(\omega)$ and $M^{ext}_1(\omega)$ there exists a homogeneous degree $-2$ solution $\Psi_2^{\text{ext}}$ in the exterior, so that
\beq\label{eq:Psi:ext:two}
\Psi_2^{\text{ext}}
= \frac{\!M^{ext}_{0}(\omega{}_{\!})\!}{r^2}\frac{r}{\!r\shortminus t\!}
+\frac{\!M^{ext}_{11}{}_{\!}(\omega{}_{\!})\!}{r^2}\ln{\!\Big|{}_{\!}\frac{ 2\, r}{t\shortminus r}{}_{\!}\Big|}
+ \frac{\!M^{ext}_{1}{}_{\!}(\omega{}_{\!})\!}{r^2}
+\mathcal{O}\Bigl(\frac{r\shortminus t\!}{\,r^3}
\ln{\!\Big|{}_{\!}\frac{ 2\, r}{t\shortminus r}{}_{\!}\Big|}\Bigr)\,,
\eq
where by \eqref{eq:M0cond},
\begin{equation}\label{eq:M0cond:bar}
M^{ext}_{11}(\omega)= -\frac{1}{2}\triangle_\omega M^{ext}_{0}{}_{\!}(\omega{}_{\!})\,.
\end{equation}

\subsection{Matching to the exterior homogeneous solution and resulting interior solution.}
\label{sec:ext:match}

We will now turn to the conditions to match $\phi_{rad}$ to $\Psi_1+\Psi_2$ in the exterior, then infer the presence of interior homogeneous solutions, and write down the matching conditions to the interior as before.

\subsubsection{Matching to leading order.}
The first matching condition is
\begin{equation}
  \lim_{q\to\infty}\mathcal{F}_{01}(q,\omega)=N^{ext}_{01}(\omega)\,,
\end{equation}
which shows that  in \eqref{eq:F01:const},
\begin{equation}
  \label{eq:N:ext}
  N(\omega)=N^{ext}_{01}(\omega)\,,\quad \text{ and hence also }\lim_{q\to-\infty}\mathcal{F}_{01}(q,\omega)=N(\omega)=N^{ext}_{01}(\omega)\,.
\end{equation}
If $N(\omega)$ does not vanish identically there is thus necessarily an \emph{interior} solution $\Psi_1^{\text{int}}[N]$, determined by $N(\omega)$, see Section~\ref{sec:hom:asymptotics:leading}.
In fact, by \eqref{eq:homoexpansion}:
\beq
\Psi_{1}^{\text{int}}[N]=N_{01}^{int}(\omega{}_{\!})\frac{1}{r}\!\ln{\!\Big|\frac{ 2\,r}{t\shortminus r}\Big|}+N_{0}^{int}(\omega{}_{\!})\frac{1}{r}
+N_{{}_{\!}11{}_{\!}}^{int}(\omega{}_{\!})\frac{r\shortminus t\!}{\,r^2}\ln{\!\Big|\frac{ 2\, r}{t\shortminus r}\Big|}
+N_{{}_{\!}1{}_{\!}}^{int}(\omega{}_{\!})\frac{r\shortminus t\!}{\,r^2}
+O\Big(\!\frac{(r\shortminus t)^2\!\!\!}{\,r^3}
\ln{\!\Big|\frac{ 2\, r}{t\shortminus r}\Big|}\Big) ,
\eq
where by \eqref{eq:N:01:0}:
\begin{equation} \label{eq:N0:int}
  N_{01}^{int}(\omega)=N(\omega)=N^{ext}_{01}(\omega)\,,\qquad N_0^{int}(\omega)=2\frac{1}{4\pi}\int_{\mathbb{S}^2} \frac{N(\sigma)-N(\omega)}{1-\langle \sigma,\omega\rangle} \ud S(\sigma)\,.
\end{equation}
As opposed to the exterior, where $N^{ext}_0(\omega)$ is free, in the interior $N_{0}^{int}(\omega)$ is determined from $N(\omega)$.
Hence while $N_{01}^{int}=N^{ext}_{01}$, the functions $N_0^{int}$ and $N^{ext}_0$ are in general different.
Moreover  $N_{{}_{\!}11{}_{\!}}^{int}(\omega{}_{\!})$ and $N_{{}_{\!}1{}_{\!}}^{int}(\omega{}_{\!})$ are computed from  $N_{01}^{int}(\omega)$ and $N_0^{int}(\omega)$ as in \eqref{eq:homocond},  and thus $2N_{11}^{int}=\triangle_\omega N=\triangle_\omega N^{ext}_{01}=2N^{ext}_{11}$, but the functions $N_1^{int}$ and $N^{ext}_1$ may be different.

The  matching condition for the radiation field in the exterior is
\begin{equation}
  \label{eq:F0:ext}
  \lim_{q\to\infty}\mathcal{F}_0(q,\omega)=N^{ext}_0(\omega)\,,
\end{equation}
and given the presence of the interior solution, we must also have
\begin{equation}
  \label{eq:F0:int}
  \lim_{q\to -\infty}\mathcal{F}_0(q,\omega)=N_0^{int}(\omega)\,.
\end{equation}

\subsubsection{Matching to second order.}

From \eqref{eq:F11:lin}, we now have by \eqref{eq:homocond:bar},
\begin{equation}
  \label{eq:F11:lin:match}
  \mathcal{F}_{11}(q,\omega)=N^{ext}_{11}q +M(\omega)\,.
\end{equation}
and that matches \eqref{eq:Psi:ext:one} and \eqref{eq:Psi:ext:two} if also
\begin{equation}
  \label{eq:M:ext}
  M(\omega)=M^{ext}_{11}(\omega)\,.
\end{equation}
The interior solution $\Psi_1^{int}[N]$ is already determined and the linear term in \eqref{eq:F11:lin:match} matches because $N^{ext}_{11}=N_{11}^{int}$.
Moreover the constant term in \eqref{eq:F11:lin:match} matches $\Psi_2^{int}$ in the interior provided $M_{11}^{int}=M^{ext}_{11}$.
The homogeneous  degree $-2$ solution in the interior, $\Psi_2^{int}$, is completely determined by the choice of a source function $M_0^{int}(\omega)$, see Section~\ref{sec:hom:asymptotics:second}.
Since by \eqref{eq:M01:given}, $M_{11}^{int}=-\frac{1}{2}\triangle_\omega M_0^{int}$, and by \eqref{eq:M0cond:bar}, $M_{11}^{ext}(\omega)= -\triangle_\omega M_{0}^{ext}(\omega)/2$,
we we see that the matching condition $M_{11}^{int}=M^{ext}_{11}$ implies  that $M_0^{int}$ agrees with $M^{ext}_0$ \emph{up to a constant}. So we choose
\begin{equation}
  \label{eq:M0:C}
  M_0^{int}(\omega)=M^{ext}_0(\omega)+C_0,
\end{equation}
with a constant $C_0$ to be determined.
Finally we see from \eqref{eq:M01:given}, that as a consequence of \eqref{eq:M0:C},
\begin{equation}
  M_1^{int}(\omega)=M_1^0(\omega)+\frac{1}{4}C_0,
\end{equation}
where $M_1^0(\omega)$ is the function $M_1(\omega)$ computed from \eqref{eq:M01:given} for the choice of the source function $M_0^{ext}$.

We now return to \eqref{eq:F1:M}, which in view of \eqref{eq:N:ext}, \eqref{eq:M:ext} and \eqref{eq:M0cond:bar} we can rewrite as:
\begin{equation} \label{eq:F1:H}
  \begin{split} 2\mathcal{F}_1^\prime(q,\omega)=&-N^{ext}_{01}(\omega)+\big(2-1/\langle q \rangle^2 \big)\Delta_\omega N^{ext}_{01}(\omega)+\Delta_\omega \mathcal{F}_0(q,\omega)+2 (q/\langle q\rangle^2)M^{ext}_{11}(\omega)\\
    =&-N^{ext}_{01}(\omega)+2\Delta_\omega N^{ext}_{01}(\omega)+\Delta_\omega\bigtwo( \mathcal{F}_0(q,\omega)-(q/\langle q\rangle^2)M^{ext}_0(\omega)\bigtwo) -\Delta_\omega N^{ext}_{01}(\omega)/\langle q \rangle^2\\
    =& 2N^{ext}_1(\omega)\chi_{q>0}+2N_1^{int}(\omega)\chi_{q<0}\\&\quad +\triangle_\omega\bigtwo(\mathcal{F}_0(q,\omega)-N^{ext}_{0}(\omega)
    \chi_{q>0}-N_0^{int}(\omega)\chi_{q<0}-(q/\langle q\rangle^2)M^{ext}_0(\omega)\bigtwo)-\triangle_\omega N^{ext}_{01}/\langle q\rangle^2 ,
  \end{split}
\end{equation}
where $\chi_{q>0}+\chi_{q<0}=1$, and where we have used that by \eqref{eq:homocond:bar},
\begin{equation}
  -N^{ext}_{01}(\omega)+2\Delta_\omega N^{ext}_{01}(\omega)=2N^{ext}_1(\omega)-\Delta_\omega N^{ext}_0=2N_1^{int}(\omega)-\Delta_\omega N_0^{int}
\end{equation}
holds in the exterior and in the interior, with the difference that in the exterior $N^{ext}_0(\omega)$ is freely chosen, and in the interior $N_0^{int}(\omega)$ is determined from $N(\omega)$ by \eqref{eq:N0:int}, and this equation determines $N_1^{int}(\omega)$ and $N^{ext}_1(\omega)$ accordingly.
As remarked above, with \eqref{eq:M0:C} we can also replace
\begin{equation}
\triangle_\omega\bigtwo(  -(q/\langle q\rangle^2)M^{ext}_0(\omega)\bigtwo)=\triangle_\omega\bigtwo(  -(q/\langle q\rangle^2)M^{ext}_0(\omega)\chi_{q>0}-(q/\langle q\rangle^2)M_0(\omega)\chi_{q<0}\bigtwo)\,.
\end{equation}

Therefore $\mathcal{F}_1$ has the correct linear behaviour, namely the linear term matches exterior \emph{and} interior asymptotics, provided
\begin{equation}
  \label{eq:H0:F0}
  \mathcal{H}_0(q,\omega)=\mathcal{F}_0(q,\omega)-N_{0}^{ext}(\omega)\chi_{q>0}-N_0^{int}(\omega)\chi_{q<0}-(q/\langle q\rangle^2)M^{ext}_0(\omega)\chi_{q>0}-(q/\langle q\rangle^2)M_0^{int}(\omega)\chi_{q<0}
\end{equation}
decays sufficiently in $q$. In fact, let us assume that
\begin{equation}
 \label{eq:H0:integrable}
  |(\langle q\rangle \partial_q)^k \partial_\omega^\alpha \mathcal{H}_0(q,\omega)|\les \langle q\rangle^{-\kappa}\,,\qquad \kappa>1\,.
\end{equation}

Then in particular the matching conditions \eqref{eq:F0:ext} and \eqref{eq:F0:int} are satisfied,
along with the the matching conditions to second order
\begin{equation}
  \lim_{q\to\infty}q\bigl(\mathcal{F}_{0}(q,\omega)-N^{ext}_0(\omega)\bigr)=M^{ext}_0(\omega)\qquad
  \lim_{q\to-\infty}q\bigl(\mathcal{F}_{0}(q,\omega)-N_0^{int}(\omega)\bigr)=M_0^{int}(\omega)\,.
\end{equation}

Therefore, we can integrate \eqref{eq:F1:H} and verify the matching condition in the exterior with
\begin{equation}
  2\mathcal{F}_1(q,\omega)=2N^{ext}_1(\omega)\chi_{q>0} \, q+2 N_1^{int}(\omega) \chi_{q<0} \, q + 2M^{ext}_1(\omega) -\int_q^\infty\!\!\!\! \triangle_\omega \mathcal{H}_0(s,\omega)-\triangle_\omega N^{ext}_{01}(\omega)/\langle s\rangle^2\ud s\,,
\end{equation}
which then also satisfies the matching condition in the interior provided
\begin{equation}
  2M^{ext}_1(\omega) -\int_{-\infty}^\infty\!\!\!\! \triangle_\omega \mathcal{H}_0(s,\omega)-\triangle_\omega N^{ext}_{01}(\omega)/\langle s\rangle^2\ud s=2M_1^{int}(\omega)\,.
\end{equation}
We can view this equation as a necessary condition on the integral of $\mathcal{H}_0$:
\begin{equation}
  \label{eq:H0:compatibility:M}
\triangle_\omega \Bigl[ \int_{-\infty}^{\infty}\mathcal{H}_0(s,\omega)\ud s \Bigr]=2M^{ext}_1(\omega)-2M_1^{int}(\omega)+\pi \triangle_\omega N^{ext}_{01}(\omega) .
\end{equation}
This equation has an integrability condition:
\begin{equation}
  \overline{M_1^{ext}}=\overline{M_1^{int}}=\overline{M_1^0}+\frac{1}{4}C_0\,.
\end{equation}
Since $M_1^0$ is determined from $M^{ext}_0$,
we can choose for any functions $M_0^{ext}(\omega)$ and $M^{ext}_1(\omega)$ the constant $C_0$ so that this integrability condition holds.

In conclusion, for any functions $N^{ext}_{01}$, $M^{ext}_0$, and $M^{ext}_1$,
first solve
\begin{equation}
  \triangle_\omega \mathcal{P}(\omega)=2M^{ext}_1(\omega)-2M_1^{int}(\omega)+\pi \triangle_\omega N^{ext}_{01}(\omega)\,,
\end{equation}
(which by appropriate choice of $C_0$ is always possible) and then we are free to choose $\mathcal{H}_0(q,\omega)$ up to the condition
\begin{equation}
  \int_{-\infty}^{\infty} \mathcal{H}_0(q,\omega) \ud q=\mathcal{P}(\omega)\,.
\end{equation}

\subsubsection{Matching without a compatibility condition.}

In the previous section we have viewed the exterior homogeneous solutions as given.
Then \eqref{eq:H0:compatibility:M} is a compatibility condition for the prescription of the radiation field.
Alternatively, if $\mathcal{F}_0$ is prescribed as in \eqref{eq:H0:F0}, with $\mathcal{H}_0$ any smooth function satisfying \eqref{eq:H0:integrable}, then \eqref{eq:H0:compatibility:M} can be satisfied with an appropriate choice of $M_1^{ext}$.

This is the setting of Theorem~\ref{thm:ext:scattering}. Note that the choice of the functions $N_{01}(\omega)$ and $N_0^{ext}$ determine the homogeneous degree $-1$ solution in the exterior, but $M_0^{ext}(\omega)$ alone does not determine the homogeneous degree $-2$ solution in the exterior. Here $\mathcal{H}_0(q,\omega)$ is chosen freely, and not subject to a compatibility condition. In view of Corollary~\ref{thm:hom:M}, a second function $M_1^{ext}(\omega)$ is needed to determine the homogeneous degree $-2$ solution in the exterior, which is then fixed by the relation \eqref{eq:H0:compatibility:M}.

More precisely, for any given $N_{01}$, $N_0^{ext}$, and $M_0^{ext}$, and any smooth function $\mathcal{H}_0$ satisfying \eqref{eq:H0:integrable}, we can define $\mathcal{F}_0$  by \eqref{eq:H0:F0}, where $N_0^{int}$ is determined by $N_{01}$, and $M_0^{int}$ is given by \eqref{eq:M0:C} for any given constant $C_0$.
Then the left hand side of \eqref{eq:H0:compatibility:M} is given, and we \emph{define} the function $M_1^{ext}$ by \eqref{eq:H0:compatibility:M}:
\begin{equation}
M^{ext}_1(\omega)=\triangle_\omega \Bigl[ \frac{1}{2}\int_{-\infty}^{\infty}\mathcal{H}_0(s,\omega)\ud s \Bigr]+M_1^{int}(\omega)-\frac{\pi}{2} \triangle_\omega N^{ext}_{01}(\omega) \,.
\end{equation}
In this way the prescribed radiation field determines the second order asymptotics in the exterior.

\subsection{Asymptotic solution.}

In a manner similar to  Section~\ref{sec:approximate}, we can now define and estimate the asymptotic solution for the proof of Theorem~\ref{thm:ext:scattering}.
With our choices of $\Psi_{rad}$ in Section~\ref{sec:ext:rad} matched to the homogeneous solutions in the exterior and interior as in Section~\ref{sec:ext:match}, we take
\begin{equation}
  \Psi_{asym}=\chi_a \Psi_{rad}+\bigl(1-\chi_a^{-}\bigr)\Psi_{hom}^{int}+\bigl(1-\chi_a^+\bigr)\Psi_{hom}^{ext},
\end{equation}
where with a smooth cutuff function $\chi$, so that $\chi(s)=1$ for $s< 1/4$ and $\chi(s)=0$ for $s>3/4$,
\begin{equation}
  \chi_a=\chi(\tminusr/r^a),\qquad\chi_a^+=\chi(\langle(r-t)_+\rangle/r^a),\qquad\chi_a^-=\chi(\langle (r-t)_-\rangle /r^a)\,  .
\end{equation}

In view of \eqref{eq:Rrad} we have that
\begin{equation}
  \Box\Psi_{rad}=R_{rad}\,,\qquad |R_{rad}|\les \ln\Bigl|\frac{2r}{\tminusr}\Bigr\rvert \frac{\tminusr}{r^4} ,
\end{equation}
and away from the lightcone,
\begin{gather}
  \Box \Psi_{hom}^{ext}=0 \,, \ r-t>0\,,\qquad  \Psi_{hom}^{ext}=\Psi_1^{ext}+\Psi_2^{ext}, \\
    \Box \Psi_{hom}^{int}=0 \,, \ r-t<0\,, \qquad \Psi_{hom}^{int}=\Psi_1^{int}+\Psi_2^{int}\,.
\end{gather}
Therefore
\begin{equation}
  \Box \Psi_{asym}=R_{asym}=\chi_aR_{rad}+(\Box \chi_a^+)\Psi_{diff}^++2Q(\partial\chi_a^+,\partial\Psi_{diff}^+)
  +(\Box \chi_a^-)\Psi_{diff}^-+2Q(\partial\chi_a^-,\partial\Psi_{diff}^-)\,,
\end{equation}
where
\begin{equation}
  \Psi_{diff}^+=\Psi_{rad}-\Psi_{hom}^{ext}\,,\qquad \Psi_{diff}^-=\Psi_{rad}-\Psi_{hom}^{int}\,.
\end{equation}
For convenience we can bring the remainder $R_{asym}$ in exactly the form \eqref{eq:R:asym:int}, with $\Psi_{diff}$ replaced by
\begin{equation}
  \label{eq:102}
  \Psi_{diff}=\Psi_{rad}-\chi_{q>0}\Psi_{hom}^{ext}-\chi_{q<0}\Psi_{hom}^{int}\,.
\end{equation}
In view of the results in Section~\ref{sec:ext:match},
the radiation fields in $\Psi_{diff}$ then decay towards the exterior and interior,
and we obtain in analogy to \eqref{eq:secondorderapproximatesource2H} but without any logarithmic terms that cancel,
\beq 
\Psi_{diff}=\frac{\!\mathcal{H}_{\!0\!}(r\shortminus t,\omega)\!\!}{r}
+\frac{\!\mathcal{H}_{\!1\!}(r\shortminus t,\omega{}_{\!})\!\!}{r^2}
+O\Big(\!\frac{(r\shortminus t)^2\!\!\!}{\,r^3}
\ln{\!\Big|\frac{ 2\, r}{t\shortminus r}\Big|}\Big)\,,
\eq
where
\begin{equation}
  |\mathcal{H}_0(q,\omega)| +\langle q\rangle^{-1} |\mathcal{H}_1(q,\omega)|\les \langle q\rangle^{-\kappa}\,.
\end{equation}
Hence we obtain as before
\begin{equation}
  \label{eq:105}
  \big|\Box \chi_a \big||\Psi_{diff}| + \big| Q(\pa\chi_a,\pa\Psi_{dif{}_{\!}f})\big|\les  r^{-2-a(1+\kappa)}\,.
\end{equation}

In conclusion, as above in \eqref{eq:R:asym:int},
\begin{equation}
  \Box\Psi_{asym}=R_{asym}=\chi_a R_{rad}
 +(\Box \chi_a) \Psi_{diff}+2 Q(\pa\chi_a,\pa\Psi_{diff})\,,
\end{equation}
where for $\kappa=2$, and $a=1/2$,
\begin{equation}
  \label{eq:RHomAsym}
  |R_{asym}|\les \frac{\ln{r}}{r^4}  |q|\chi\bigl(\tfrac{\tminusr}{2r^{1/2}}\bigr)\,.
\end{equation}
\begin{remark} More generally, for $1<\kappa\leq 2$, we can choose $a$ so that $a(1+\kappa)=2$.
\end{remark}

\subsection{Estimate of the remainder.}

For the proof of Theorem~\ref{thm:ext:scattering},
it remains to show that $R_{asym}$ produces a sufficiently fast decaying remainder.
Let $E_-=\delta(t^2-|x|^2)H(-t)/2\pi$ be the backward fundamental solution of $-\Box$. Then in view of \eqref{eq:RHomAsym} the convolution
\beq
\psi_{rem}=E_-*R_{asym}
\eq
is well defined,  and   we show in Appendix~\ref{sec:RadialEstimates}, see Lemma \ref{lem:homoremainderest},  that
\begin{equation}
  |\psi_{rem}|\lesssim \frac{\ln{\tplusr}}{\tplusr^2} .
\end{equation}
In particular, the radiation field of  $\psi_{rem}$  vanishes.

\appendix

\section{Appendix for the interior problem}
\label{app:int}

\subsection{Formulas for the forward fundamental solution in the interior with homogeneous data on light cones \cite{L90a,L17}.}\label{sec:sourseformulas}

Let $E$ be the forward fundamental solution of $-\Box$ and let $\mu=w_0(y)\delta(s-\vert y\vert)$. Then
 $$
E*{\mu}(t,x)=\frac{1}{2\pi} \int\!\!\int\delta\big((t-s)^2-\vert
x-y\vert^2\big)w_0(y)\delta(s-\vert y\vert)\,dy\,ds
=\frac{1}{2\pi}  \int\delta\big((t-|y|)^2-\vert
x-y\vert^2\big)w_0(y)\,dy .
$$
Introducing polar coordinates $x=r\omega$, $y=\rho \sigma$,
where $\omega,\sigma\in \Bbb S^2$ we write
$$
(t-|y|)^2-\vert x-y\vert^2=t^2-r^2 -2\big(t -r \langle \omega,\sigma\rangle\big)\rho
\equiv f(\rho,\sigma).
$$
Let $\rho_0=\rho_0(\sigma)$ be such that $f(\rho_0)=0$.
For $t>r$ we get
 $$
E*{\mu}(t,x)=\frac{1}{2\pi} \int_{\Bbb S^2} \int_0^\infty \delta\big(f(\rho,\sigma)\big) \, w_0(\rho)\,\rho^2 d\rho \,dS(\sigma)
=\frac{1}{2\pi} \int_{\Bbb S^2}  \frac{\rho_0^2}{|\pa_\rho f(\rho_0,\sigma)|}w_0(\rho_0)\, \,dS(\sigma).
$$
Hence
$$
E*{\mu}(t,x)
=\frac{1}{4\pi} \int_{\Bbb S^2}  \frac{\rho^2 w_0(\rho\sigma)}{t -r \langle \omega,\sigma\rangle}\, \,dS(\sigma),\qquad\text{where}\quad \rho=\frac{1}{2}\, \frac{t^2-r^2}{t -r \langle \omega,\sigma\rangle}\,.
$$
Any function $F$ can be decomposed as measures on light cones
\beq
F(t,x)=\int \mu_q(t,x)\, dq,\quad\text{where} \quad \mu_q(t,x)=\delta(r-t-q) w_q(x),\quad w_q(x)=F(r-q,x)\,,
\eq
and
\beq E*F(t,x)=\int E*\mu_q (t,x)\, \ud q\,.\eq

For $k=2,3$ let
${\Phi_{\chi}}^{\!\!\!k}[m]$ be the solution of
\beq\label{eq:thekeq}
-\Box {\Phi_{\chi}}^{\!\!\!k}[m](t,r\omega)=m(r-t,\omega)r^{-k}\chi\big(\tfrac{\langle\,r-t\,\rangle}{r}\big)^2,
\eq
with vanishing data at $-\infty$, where $\chi$ is a smooth cutoff function such that $\chi(s)=0$, when $s\geq 1/2$ and
$\chi(s)=1$ when $s\leq 1/4$.
From  \cite{L90a,L17} we recall the formula for the solution
 \begin{equation}\label{eq:k2sol}
{\Phi_{\chi}}^{\!\!\!2}[m](t,r\omega)=\int_{r-t}^{\infty} \frac{1}{4\pi}\int_{\bold{S}^2}{\frac{
 m({q},{\sigma})\chi\big(\tfrac{\langle\,{q}\,\rangle}{\rho}\big)  dS({\sigma})d {q}}{t-r+{q}+r\big(1-\langle\, \omega,{\sigma}\rangle\big)}\,
}\,,
\end{equation}
and
\beq\label{eq:k3solformula}
{\Phi_{\chi}}^{\!\!\!3}[m](t,r\omega)=\int_{r-t}^{\infty} \frac{1}{2\pi}\int_{\bold{S}^2}{\frac{
 m({q},{\sigma})\chi\big(\tfrac{\langle\,{q}\,\rangle}{\rho}\big) dS({\sigma})}{(t-r+{q})(t+r+q)}\,
}\, d {q},
\eq
where
\begin{equation}\label{eq:chi:rho}
\rho=\frac{1}{2}\, \frac{(t+r+q)(t-r+q)}{t-r+{q}+r\big(1-\langle\, \omega,{\sigma}\rangle\big)}\,.
\end{equation}

For $k=2$  no cutoff function is needed. However for $k=3$ the right hand side of \eqref{eq:thekeq} is not in $L^1_{loc}$ and not even well defined as a distribution without cutting off a neighborhood of $r=0$.
The cutoff away from the light cone is of no importance since we will apply this to $m(r-t,\omega)$ constructed from radiation fields that will only be good approximations  close to the light cone $t\sim r$.
However, the asymptotic behaviour of the solution close to the light cone is dependent on the exact choice of cutoff $\chi$ and we would like to remove this dependence. Even though the right hand side of \eqref{eq:thekeq}
for $k=3$ does not make sense as a distribution when $r\sim 0$ if the cutoff function $\chi$ is removed, and the solution formula \eqref{eq:k3solformula} may still make sense if the support of the fundamental solution from the point $(t,x)$ does not contain the set where the right hand side is undefined. This in particular is the case if $m(q,\omega)=0$ for  $q\leq 0$. In the application we only know that $m$ decays in $q$.

\subsection{Taylor series of the spherical integrals for the interior homogeneous solution.} 
\label{sec:SphericalTaylor}

Recall from Section~\ref{sec:hom:asymptotics:leading}, that we can rewrite  the  homogeneous solution $\Psi_1$ in the interior given in \eqref{eq:wavesourceconeformulasec2.3} as
  \begin{equation}
    \label{eq:19}
    \Psi_1[N](t,r\omega)=\frac{1}{\!4\pi r}\!\int_{\mathbb{S}^2}\!\!
  \frac{N(\sigma{}_{\!})\, dS(\sigma) }{t/r-\!\langle\sigma,\omega{}_{\!}\rangle}=\frac{1}{2r}\int_{-1}^1\frac{N(\omega,z) \ud z}{t/r-z}
\end{equation}
where $N(\omega,z)$ is the average of $N(\sigma)$ over the circle $\langle \omega,z\rangle=z$. Similarly, for $\Psi_2$ as given in \eqref{eq:Psi:2}.
We first study the regularity of these averages, and then go on to prove the existence of appropriate expansions of $r\Psi_1$, and $r^2\Psi_2$, in $y=t/r-1$.

\subsubsection{Averaging over circles}
Let $N(\sigma)$ be a $C^{2k}$ function on $\mathbb{S}^2$.
Let
\begin{equation}
\label{eq:avg}
N(\omega,z)\!=\!\int_{\langle \sigma,\omega\rangle =z}\!\!\!\!\! \!\! \!  N(\sigma) \,ds(\sigma) \Big/\!\!\int_{\langle \sigma,\omega\rangle =z} \!\!\!\! \!\! \!  ds(\sigma) ,
\end{equation}
 be the average of $N(\sigma)$ over the circle $\langle{}_{\,} \omega,\sigma{}_{\!}\rangle\!=\!z$.
We claim that $N(\omega,z)$ is a $C^k$ function on $\mathbb{S}^2\times[-1,1]$.
To show this we start by choosing coordinates $(\sigma_1,\sigma_2,\sigma_3)$   so that $\omega=(1,0,0)$.
We can use $(\sigma_2,\sigma_3)$ as local coordinates on $\mathbb{S}^2\!$ close to $(1,0,0)$ and write $N\!\!=\!N_{\!}(\sigma_2,\sigma_3)$, as a function of these. By Taylor's formula (in one dimension applied to the function $N(t\sigma_2,t\sigma_3)$) with integral remainder:
\begin{equation} N(\sigma_2,\sigma_3)
=\sum_{j=0}^{2\ell-1}\frac{1}{j!}\big((\sigma_2\partial_2+\sigma_3\partial_3)^j
N\big)(0,0)+\int_0^1 \frac{(1-t)^{2\ell}}{(2\ell)!}\big((\sigma_2\partial_2+\sigma_3\partial_3)^{2\ell}
N\big)(t\sigma_2,t\sigma_3)\, \ud t,
\end{equation}
for $\ell\leq k$.
Introducing polar coordinates $(\sigma_2,\sigma_3)=(r\cos\theta,r\sin\theta)$,  where $r=\sqrt{1-z^2}$, we see that the integral of odd powers of $\sigma_2$ or $\sigma_3$ vanish, because for $m+n$ odd, $\int_{0}^{2\pi}\cos^m(\theta)\sin^n(\theta)\ud \theta=0$.
Therefore we get
\beq
\frac{1}{2\pi} \int_0^{2\pi}  N(r\cos\theta,r\sin\theta)\, d\theta\!
=\sum_{j=0}^\ell r^{2j} N_{2j}^{2\ell},
\eq
where $N_{2j}^{2\ell}$, for $j<\ell$ are independent of $r$ and $N_{2\ell}^{2\ell}$ is a continuous function of $r$. Moreover, because the integral is independent of rotations of $(\sigma_2,\sigma_3)$ it follows that $N_{2j}^{2\ell}$, for $j\leq \ell\leq k$ are $C^{2k-2j}$ functions of $\omega\in \mathbb{S}^2$.
This shows that $N(\omega,z)$ has a Taylor series in $1-z$ up to order $l\leq k$, because $r^2=(1-z)(1+z)$, and $N^{2\ell}_{2j}(\omega)$ are $C^{2(k-j)}$ functions of $\omega$, $j\leq \ell$.
It follows from this $N(\omega,z)$ is a $C^k$ function on $\mathbb{S}^2\times[-1,1]$.

\subsubsection{Interior Expansion of the spherical integrals  in the quadratic case}

The expansion \eqref{eq:homoexpansion} follows from expanding
$ N(\omega,z)$ in a Taylor series and using polynomial division and partial fractions to evaluate the integrals.
To see that such an expansion exists, with $y\!=\!(t-r\!)/r $, we write
\beq
r\Psi_1[N](t,r\omega)= \int_{-1}^1 N(\omega,z)(1+y -z)^{-1} dz.
\eq
We expand the $C^k$  function $N(\omega,z)$ in a Taylor series in $1-z$,
  \begin{equation}
    \label{eq:20}
    N(\omega,z)=\sum_{j=0}^{k-1}N_j(\omega)(1-z)^j+N_k(\omega,z)(1-z)^k\,,
  \end{equation}
  where $N_k(\omega,z)$ is a continuous function in $z$, and further expand $(1-z)^j=(1-z+y-y)^j$ to get
  \begin{equation}\label{eq:N:Taylor}
N(\omega,z)=\sum_{j=0}^{k-1} N_j(\omega)\,(-y)^j+ N_k(\omega,z)\,(-y)^k +(1-z+y) \sum_{j=0}^{k} F_j(\omega,y,z) ,
\end{equation}
where $ F_j(\omega,y,z)$ are polynomials in $y$ with continuous coefficients in $z$.

Substituting this into the integral we get for some $N_{1j}$ and $N_{0j}$,
\beq
\int_{-1}^1 N(\omega,z)(1+y -z)^{-1} dz=\sum_{j=0}^{k} N_{1j}(\omega)\, y^j  \ln\Big(\frac{1}{y}\Big)+\sum_{j=0}^{k-1} N_{0j}(\omega) y^j+N_{0k}(\omega,y)y^k\,.
\eq
Here we used the mean value theorem for integrals to conclude that $N_{1k}(\omega)$  is a continuous function of $\omega$.

\subsubsection{ Lower order Interior Expansion of the spherical integrals}

We can proceed similarly for
\beq
r^2\Psi_2[M](t,r\omega)= \int_{-1}^1 \! \!\! \! M(\omega,z)(1+y -z)^{-2} dz,
\quad\text{where}\quad
M(z,\omega)\!=\!\int_{\langle \sigma,\omega\rangle =z}\!\! \!\! \!\!\! \! \!\! M(\sigma) \,ds(\sigma) \Big/\!\!\int_{\langle \sigma,\omega\rangle =z} \!\!\! \!\! \!\!\!\! \!\! \!  ds(\sigma).
\eq
Here $M$ is assumed to be a $C^{2k}$ function, so  we can expand the spherical means $M(z,\omega)$  in a Taylor series in $1-z$ as in \eqref{eq:N:Taylor}. Substituting above gives for some $M_j(\omega)$,
\begin{equation}
  \int_{-1}^1 \! \!\! \! M(\omega,z)(1+y -z)^{-2} \ud z=\sum_{j=0}^{k}M_{1j}(\omega)\frac{y^{j-1}}{2+y}+\sum_{j=0}^k\int_{-1}^1 \! \!\! \! G_j(\omega,y,z)(1+y -z)^{-1} \ud z
\end{equation}
where $G_j(\omega,y,z)$ are polynomials in $y$ of order $j$ with continuous coefficients in $z$. Hence by the mean value theorem for integrals there are $M_{0j}(\omega,y)$ so that
\begin{equation}
  \sum_{j=0}^k\int_{-1}^1 \! \!\! \! G_j(\omega,y,z)(1+y -z)^{-1} \ud z=\sum_{j=0}^k M_{0j}(\omega)y^j\ln\Big(\frac{2+y}{y}\Big)\,.
\end{equation}

\subsubsection{Non-singular spherical integrals}
\label{app:nonsingular}
From the above discussion it is also clear that an integral  such as
 \begin{equation}
   N_0(\omega) = \frac{1}{4\pi}\int_{\mathbb{S}^2}\frac{N(\sigma)-N(\omega)}{1-\langle \sigma,\omega\rangle} \ud S(\sigma) ,
 \end{equation}
 which we have  encountered in Section~\ref{sec:hom:asymptotics:leading}, is finite, provided $N(\sigma)$ is sufficiently regular.
 Indeed if $N(\sigma)$ is $C^2$, then using $z=\langle \sigma,\omega\rangle$ as a coordinate, and  a Taylor expansion in $1-z$ of the average of $N(z,\omega)$ over the circle $\langle\sigma,\omega\rangle =z$, we have that
 \begin{equation}
   N_0(\omega)=\frac{1}{2}\int_{-1}^{1}\frac{N(z,\omega)-N(\omega)}{1-z}\ud z=\frac{1}{2}\int_{-1}^{1}N_1(z,\omega)\ud z
 \end{equation}
 is a continuous function on $\mathbb{S}^2$.

\subsection{The radial estimates.} 
\label{sec:RadialEstimates}

In this section we want to estimate the solution $\psi_{rem}$ to $-\Box\psi=R_{asym}$, with vanishing data at infinity, where for the problem with a source by \eqref{eq:Rasymest},
\beq\label{eq:Rasymest:B}
\big|R_{asym}\big|
\les  \frac{\ln r }{r^{4}}\,\bigtwo(|q|\, \chi_{0<-q< r^a}
+\langle q\rangle^{-2\gamma}\,\chi_{q\geq 0}\bigtwo)+\frac{1}{t^2\langle q\rangle^{2+2\gamma}}\chi_{ -r/2<q<-r^a}\,,
\eq
for $\gamma\geq 1/2$, and $a=1/2$, and for the homogeneous problem by \eqref{eq:RHomAsym}
\begin{equation}\label{eq:RHomAsymApp}
  |R_{asym}|\les \frac{\ln{r}}{r^4} |q| \chi\bigl(\tfrac{\tminusr}{2r^{1/2}}\bigr)\,.
\end{equation}

The solution is given by the convolution with the fundamental solution $E_-$, $\psi_{rem}=E_-\ast R_{asym}$, and using the positivity of the fundamental solution $E_-$, this can be estimated using radial estimates.
 We are here interested in the effect of the asymptotics of this source term on the solution $\psi_{rem}$, and the terms dealt with in this section, which have a modified cutoff compared to $R_{asym}$, are equivalent in this respect.

More generally, consider $-\Box \phi=F$,
and define $\overline{F}$ by taking the supremum over the angular variables, $\overline{F}(t,r)\!=\!\sup_{\omega\in S^2} |F(t,r\omega)|$
and let $F_0\!=\!F H$ where $H\!=\!1$,
when $t\!>\!0$ and $H\!=\!0$, when $t\!<\!0$.
 Since $|F_0|\leq \overline{F}_0$ it follows from
 the positivity of the fundamental solution $E_-$, that $|\phi|\leq
 |\overline{\phi}|$ where $\overline{\phi}$ is the solution of
 $-\Box\overline{\phi}=\overline{F}_0$ with vanishing  data at infinity.
 Since the wave operator is invariant under rotations
 it follows that $\overline{\phi}$ is independent of the angular
 variables so
 $
 (\pa_t-\pa_r)(\pa_t+\pa_r)(r\overline{\phi}(t,r))=r\overline{F}_0.
 $
 If we now introduce new variables $\xi=t+r$ and $\eta=t-r$ and
 integrate over the region
 $$
 R=\{(\xi,\eta);\,
 t-r\leq \eta\leq t+r,\, \, t+r\leq \xi< \infty\},
 $$
 using that $r\overline{\phi}(t,r)$ vanishes when
 when $r=0$ and when $\xi=\infty$, we obtain
 \beq
 r\overline{\phi}(t,r)=4\int_{t+r}^{\infty} \int_{t-r}^{t+r}
 \rho \overline{F}_0(s,\rho)  \, d\eta d\xi,
 \qquad s=\frac{\xi+\eta}{2},\quad \rho=\frac{\xi-\eta}{2}.
 \eq

The following Lemma deals with  \eqref{eq:RHomAsymApp} and the first two terms in \eqref{eq:Rasymest:B}.

\begin{lemma}\label{lem:homoremainderest}
  Suppose in the region $R$,
\beq
 \rho \overline{F}_0(s,\rho)
 \leq g(\xi) \big(|\eta| \chi_{|\eta|<\xi^a} +\langle\eta\rangle^{-2\gamma} \chi_{\eta<0}  \big),\quad\text{where}\quad g(\xi)=\frac{\ln\langle \xi\rangle}{\langle\xi\rangle^3}\,,
\eq
and $\gamma\geq 1/2$, and $a=1/2$. Then
\begin{equation}\label{eq:remainder:radial}
\overline{\phi}(t,r)\leq \frac{\ln{\tplusr}}{\tplusr^2} .
\eq
\end{lemma}
\begin{proof}
  We divide the proof up in the two cases $r>t$, and $r<t$.

First  if $r>t$ then we must evaluate
\beq
\int_{t-r}^{t+r}
 \big(|\eta| \chi_{|\eta|<\xi^a} +\langle\eta\rangle^{-2\gamma} \chi_{\eta<0}  \big)    d\eta\leq \int_{0}^{\min\{t+r,\xi^a\}}\!\!\!\!|\eta|
  \, d\eta+\int_{-\min\{|t-r|,\xi^a\}}^{0}\!\!\!\!
|\eta| d\eta+\int_{t-r}^{0}
\langle\eta\rangle^{-2\gamma}  d\eta\les \min\{(t+r)^2,\xi\}+\ln{\langle t\shortminus r\rangle},
\eq
using  now  that $\gamma>1/2$ and $a=1/2$.
Therefore
\begin{multline}
\int_{t+r}^{\infty} g(\xi) \int_{t-r}^{t+r}
 \big(|\eta| \chi_{|\eta|<\xi^a} +\langle\eta\rangle^{-2\gamma} \chi_{\eta<0}  \big)   \, d\eta\, d\xi \leq \int_{t+r}^{\infty} g(\xi) \big(\min\{(t+r)^2,\xi\}+\ln{\langle t\shortminus r\rangle}\big)\, d\xi\\
 \!=\!\int_{t+r}^{(t+r)^2}\!\!\!\!\!\!\!\! g(\xi)\xi\, d\xi+ (t\!+\!r)^2\!\!\! \int_{(t+r)^2}^\infty\!\! g(\xi)\, d\xi +\ln{\langle t\shortminus r\rangle}\!\!\!\int_{t+r}^\infty\!\!\! g(\xi)\ud \xi
 \les \frac{\ln{\xi}}{\xi}\Big|_{\xi=\langle t+r\rangle}\!\! +(t\!+\!r)^2 \frac{\ln{\xi}}{\xi^2}\Big|_{\xi=\langle t+r\rangle^2}\!\!+\ln{\langle t\shortminus r\rangle}\frac{\ln{\xi}}{\xi^2}\Big|_{\xi=\langle t+r\rangle}\!.\!\!\!\!\!
\end{multline}
Hence in that case we can estimate
\beq
\overline{\phi}(t,r)\les \frac{\ln{\tplusr}}{r\tplusr},\quad r>t.
\eq

Second if $r<t$,  
we evaluate first for $\xi\geq t+r$,
\beq
\int_{t-r}^{t+r}
 \big(|\eta| \chi_{0<\eta<\xi^{1/2}} +\langle\eta\rangle^{-2\gamma} \chi_{\eta<0}  \big)   \, d\eta\leq \int_{t-r}^{\min\{t+r,\xi^{1/2}\}}|\eta|  \, d\eta\les \min\{(t+r)^2,\xi\}-(t-r)^2,
\eq
and note that the integral vanishes if $t-r>\xi^{1/2}$.
Hence when $t-r>0$ we have
\begin{equation}
\int_{t+r}^{\infty} g(\xi) \int_{t-r}^{t+r}
 \big(|\eta| \chi_{0<\eta<\xi^{1/2}} +\langle\eta\rangle^{-2\gamma} \chi_{\eta<0}  \big)   \, d\eta\, d\xi \leq \int_{\max\{t+r,(t-r)^2\}}^{\infty} g(\xi) \big(\min\{(t+r)^2,\xi\}-(t-r)^2\big)\, d\xi.
\end{equation}
Now if $t+r>(t-r)^2$ then we estimate this by
\begin{multline}
 \int_{t+r}^{(t+r)^2} g(\xi) \big(\xi-(t-r)^2\big)\, d\xi
 +\int_{(t+r)^2}^{\infty} g(\xi) \big((t+r)^2-(t-r)^2\big)\, d\xi
 \les  \frac{\ln{\xi}}{\xi}\Big|_{\xi=\langle t+r\rangle}+tr \frac{\ln{\xi}}{\xi^2}\Big|_{\xi=\langle t+r\rangle^2}\,,
\end{multline}
which gives  again that
\beq
\overline{\phi}(t,r)\les \frac{\ln{\tplusr}}{r\tplusr},\qquad |r-t|<(t+r)^{1/2}.
\eq

In the remaining case when $r\!<\!t$ and $(t\!-\!r)^{\,2}\!>\!t\!+\!r$,
we distinguish  the cases $r\!<\!t/2$, and $r\!>\!t/2$.
We have
\begin{equation}
 \int_{(t-r)^2}^{(t+r)^2} g(\xi) \big(\xi-(t-r)^2\big)\, d\xi
 +\int_{(t+r)^2}^{\infty} g(\xi) \big((t+r)^2-(t-r)^2\big)\, d\xi
 \les \frac{\ln{\xi}}{\xi}\Big|_{\xi=\langle t-r\rangle^2}+tr\frac{\ln{\xi}}{\xi^2}\Big|_{\xi=\langle t+r\rangle^2}\,,
\end{equation}
which also gives \eqref{eq:remainder:radial}, at least  for $t/2<r<t$:
\beq
\overline{\phi}(t,r)\les \frac{\ln{\langle t-r\rangle}}{r\langle t-r\rangle)^2}+\frac{\ln\tplusr}{\tplusr^3}\les \frac{\ln \tplusr}{\tplusr^2},\quad |r-t|^2>t+r.
\eq

If $r<t/2$,  we estimate the first integral differently:
\begin{equation*}
 \int_{(t-r)^2}^{(t+r)^2} g(\xi) \big(\xi-(t-r)^2\big)\, d\xi
  \les \frac{\ln\tplusr}{\tplusr^6}
  \int_{0}^{(t+r)^2-(t-r)^2}\!\!\!\!\!\!\!\!\!\!\!\!\!\!s\, ds
   \les \frac{\ln\tplusr}{\tplusr^4}r^2,
\end{equation*}
which in particular gives \eqref{eq:remainder:radial}, also in the case $r<t/2$.
\end{proof}

  The second Lemma applies to the last term in \eqref{eq:Rasymest:B}.

  \begin{lemma}
    Suppose in the region $R$,
    \beq
     \rho \overline{F}_0(s,\rho)
 \leq
 \frac{1}{\langle \xi\rangle \langle \eta\rangle^{2+2\gamma}}\chi_{ \xi^a < \eta < \xi/2},
\eq
where $\gamma\geq 1/2$, and $a=1/2$. Then
\beq
\overline{\phi}(t,r)\les  \frac{1}{\tplusr^2}.
\eq
\end{lemma}
\begin{proof}
    We integrate a function supported for $\xi^a\!<\!\eta\!<\!\xi/2$, over the region $t\!-\!r\!\leq \!\eta\!\leq \!t\!+\!r$, $\xi\!\geq\! t\!+\!r$, where we take $t\!+\!r\!\geq \!4$ so $\xi/2\!>\!\xi^a$, $a\!=\!1/2$.
    Note first that in $R$, the function vanishes for $\xi\!\geq\! (t\!+\!r)^2$\!.
  So if $2\!<\!\xi^a\!<\!t\!+\!r$,
\beq
\int_{t-r}^{t+r}
\frac{1}{\langle\eta\rangle^{2+2\gamma}} \chi_{\eta>\xi^a} \, d\eta\les \int_{\max\{t-r,\,\xi^a\}}^{t+r}\frac{1}{\langle \eta\rangle^{2+2\gamma}}  \, d\eta\les
\frac{1}{\langle \eta \rangle^{1+2\gamma}}\Bigr\rvert_{\eta=\max\{t-r,\,\xi^a\}}.
\eq
and it vanishes if $\xi^a>t+r$.
 Hence
  \begin{equation}
    \int_{t+r}^{\infty}\frac{1}{\langle\xi\rangle}\int_{t-r}^{t+r}
\frac{1}{\langle\eta\rangle^{2+2\gamma}} \chi_{\eta>\xi^a} \, \ud\eta\ud\xi\les \int_{t+r}^{(t+r)^2} \frac{1}{\langle \xi\rangle^{1+a(1+2\gamma)}}\ud \xi\les\frac{1}{\tplusr^{a(1+2\gamma)}}\les \frac{1}{\tplusr}\,.
\end{equation}
  This gives, after dividing by $r$, the stated estimate, at least for  $t/2<r$.
  In the case $r<t/2$ we estimate the first integral differently.
  Note that $t-r>\xi^a$ for as long as $\xi<(t-r)^2$, and $t+r<\xi<(t-r)^2$ is a nonempty interval for $r<t/2$ , at least for $t\geq 6$.
  So we simply estimate $\chi_{\eta>\xi^a}\leq 1$ on $R$, and for $\xi^a<t+r$,
  \beq
\int_{t-r}^{t+r}
\frac{1}{\langle \eta\rangle^{2+2\gamma}} \chi_{\eta>\xi^a} \, d\eta\les \int_{t-r}^{t+r}\frac{1}{\eta^{2+2\gamma}}  \, d\eta\les
\frac{1}{(t-r)^{1+2\gamma}}-\frac{1}{(t+r)^{1+2\gamma}}\les
\frac{r}{(t+r)^{2+2\gamma}}\,,
\eq
where  we used the mean value theorem, to get the necessary factor of $r$, in the case $r<t/2$.
Finally,
\begin{equation}
  \int_{t+r}^{(t+r)^2}\frac{1}{\langle\xi\rangle}\int_{t-r}^{t+r}
\frac{1}{\langle\eta\rangle^{2+2\gamma}} \chi_{\eta>\xi^a} \, \ud\eta\ud\xi\les \frac{r\log\langle t+r\rangle}{\langle t+r\rangle^{2+2\gamma}}\,,
\end{equation}
which implies the stated estimate, after dividing by $r$, with $\gamma>0$.
\end{proof}

\section{Appendix for the exterior problem}

\label{app:ext}

\subsection{Formulas for the homogeneous solution in the exterior.}\label{sec:exteriorhomogeneous}

\subsubsection{Formulas for the solution in the exterior with  homogeneous initial data of degree $-1$}
\label{sec:formula:ext:hom}

In this section we will derive formulas for the solution to the initial value problem with homogeneous initial data, in the exterior,
\beq \label{eq:app:hom:ext:one}
\Box \,\phi_1=0,\quad |x|>t,\qquad \phi\,\Big|_{t=0}=f=M(\omega)/r,\qquad \pa_t\phi\, \Big|_{t=0}=g=N(\omega)/r^2\,.
\eq

\begin{lemma}
  The solution to \eqref{eq:app:hom:ext:one} is given by
 \beq
\phi_1=-x^i t^{-1}\Psi^{\,\text{ext}}_{1,0}[M_i]+\Psi^{\text{ext}}_{1,0}[N]\,, \qquad |x|>t\,,
\eq
where $M_i$ is defined by
\beq
 M_{i}(\omega)r^{-2}=\pa_i \big( M(\omega)r^{-1}\big)
\eq
and
\beq \label{eq:Psi:1:ext}
\Psi^{\,\text{ext}}_{1,0}[N](t,r\omega)
=r^{-1}{\mathcal{I}_{\,}}[N]\big(\omega,z_0\big),\qquad
r>t\,,\qquad  z_0=\sqrt{1-(t/r)^2}\,,
\eq
where
\beq
\mathcal{I}_{\,}[N](\omega,z_0)=\frac{1}{2\pi_{\,}}\int_{\langle \omega,\sigma\rangle>z_0} \,\frac{  N(\sigma )\, dS(\sigma)}{\sqrt{\langle \omega,\sigma\rangle^2-z_0^2}}\,.
\eq
\end{lemma}

Note that with
\beq
\mathcal{J}_{\,}[N](\omega,z_0)=\frac{1}{2\pi_{\,}}\int_{\langle \omega,\sigma\rangle>z_0} \,\frac{ \langle \omega,\sigma\rangle N(\sigma )\, dS(\sigma)}{\sqrt{\langle \omega,\sigma\rangle^2-z_0^2}}
\eq
we can write
 \beq
-x^i t^{-1}\Psi^{\,\text{ext}}_{1,0}[M_i]
=-t^{-1}\mathcal{J}_{\,}[M](\omega,z_0)+t^{-1}\omega^i\mathcal{I}_{\,}[ \nang_{i} M](\omega,z_0),
\eq
because
\beq
 M_{i}(\omega)=-\omega_i M(\omega)+\nang_i M(\omega) \,.
\eq

\begin{proof}
We derive this formula using Kirchhoff's formula
\beq
\phi_1(t,x)=\frac{1}{4\pi t^2} \int_{S(x,t)} \Big( t g(y) + f(y)+\nabla f(y)\cdot (y-x) \Big)\, \ud S(y),
\eq
where $S(x,t)$ is the sphere of radius $t$ centered at $x$.

First, let $\phi_g$ be the solution corresponding to $f=0$, which is
\begin{equation}\label{eq:phi:g}
\phi_g(t,x)=\frac{1}{4\pi t}\int_{S(x,t)} g(y)\, dy=\frac{1}{2\pi }\int_{\mathbb{R}^3} \delta\big( |y-x|^2-t^2\big) \, g(y)\, dy\,.
\end{equation}
Note this can also be interpreted as $\phi_g=E\ast (g\delta(t))$ where $E$ is the forward fundamental solution to the wave equation. Here and below we can use  that
\beq
\delta\big(f(\rho)\big)=\sum_{\rho_i; f(\rho_i)=0} |f^{\,\prime}(\rho_i)|^{-1} \delta(\rho-\rho_i)\,.
\eq
Introducing spherical coordinates $x=r\omega$ and $y=\rho\sigma$, we have
\beq
\phi_g(t,x)
=\frac{1}{2\pi }\iint \delta\big( r^2-t^2-2r\rho\langle \omega,\sigma\rangle +\rho^2\big) \, g(\rho\sigma )\, \rho^2 d\rho\, dS(\sigma).
\eq
We then perform the  integral in $\rho$ first,
\beq\label{eq:basicinitialdatasol}
\phi_g(t,x)
=\frac{1}{4\pi }\int \,{\sum}_i\frac{ g(\rho_i{}_{\,}\sigma )\, \rho_i^2 dS(\sigma)}{\big|\rho_i\!-\!r\langle \omega,\sigma\rangle\big|},
\eq
where the integral is  over all $\sigma\in\mathbb{S}^2$ such that there  are real positive solutions $\rho=\rho_i(t,r,\sigma,\omega)$ to
$$
 r^2-t^2-2r\rho\langle \omega,\sigma\rangle +\rho^2=0.
 $$
Completing the square
$$
 r^2\big(1-\langle \omega,\sigma\rangle^2\big)-t^2 +\big(\rho-r\langle \omega,\sigma\rangle\big)^2=0,
$$
we see that
$$
\big(\rho-r\langle \omega,\sigma\rangle\big)^2=r^2\bigtwo( \langle \omega,\sigma\rangle^2-\big(1-(t/r)^2\big)\bigtwo).
$$
Hence
$$
\big|\rho-r\langle \omega,\sigma\rangle\big|
= r\sqrt{\langle \omega,\sigma\rangle^2-z_0^2},\qquad\text{where}\quad
z_0=\sqrt{1-(t/r)^2}\,.
$$

To find the solutions $\rho_i$,
note that in the case $r>t$,  the ray from the origin in the direction $\sigma$ in the initial data, namely the ray $\rho\mapsto (t=0,\rho\sigma)$ at $t=0$,  either does not intersect the sphere $S(x,t)$, or it intersects this sphere at two points. If $\langle \omega,\sigma\rangle<z_0$ then the ray in the direction of $\sigma$ does not intersect the sphere but if $\langle \omega,\sigma\rangle>z_0$ then it intersects the sphere at two points
\begin{equation}
\rho_\pm=r\bigtwo(\langle \omega,\sigma\rangle \pm \sqrt{\langle \omega,\sigma\rangle^2-z_0^2}\bigtwo).
\end{equation}

Hence in the case $f=0$, $g(\rho\sigma)=N(\sigma) \rho^{-2}$ we get from \eqref{eq:basicinitialdatasol}  the solution $\phi_1=\phi_g=\Psi_{1,0}^{ext}[N]$, with $\Psi_{1,0}^{ext}[N]$ given by \eqref{eq:Psi:1:ext}, by summing up the contributions from both  points.

Secondly, for the solution corresponding to $g=0$, and $f(\rho\sigma)=M(\sigma)/\rho$ we have that
\begin{equation}
  \phi_1(t,x)=-\frac{1}{4\pi t^2}\int_{S(x,t)}\nabla f(y)\cdot x\: \ud S(y)=-\frac{1}{t}\phi_{\partial_i f}\cdot x^i
\end{equation}
where $\phi_{\partial_i f}$ is given by \eqref{eq:phi:g}, with $\partial_if$ in the role of $g$.
Introducing the functions $M_i(\sigma)$, we can then apply the formula above to get
\begin{equation}
  \phi_1(t,x)=-\frac{x^i}{t}\Psi_{1,0}^{\text{ext}}[M_i]\,,\qquad \text{where}\quad  \partial_i (M(\sigma) \rho^{-1})=M_i(\sigma)\rho^{-2}\,. \tag*{\qedhere}
\end{equation}
\end{proof}

\subsubsection{Exterior expansion of the spherical integrals}
With $z=\langle\omega,\sigma\rangle$ and $z_0=\sqrt{1-(t/r)^2}$ we can write
\beq\label{eq:app:Psi:one}
r\Psi^{\,\text{ext}}_{1,0}[N]=\int_{z_0}^1 \,\frac{  N(\omega,z )\, dz}{\sqrt{z^2-z_0^2}},
\eq
and
\beq
t\Psi^{\text{ext}}_{1,1}[M]=\int_{z_0}^1  \,\frac{z  \,M(\omega,z)\, dz}{2\sqrt{z^2-z_0^2}}.
\eq
We want to find an expansion as $z_0\to 0$. If we use a Taylor expansion of $N(\omega,z)=\sum N_k(\omega) z^k$ and $M(\omega,z)=\sum M_k(\omega) z^k$ around $z=0$ we see that we must evaluate integrals of the type
\beq
\int_{z_0}^1\!\frac{z^k \, dz}{2\sqrt{z^2-z_0^2}}
=z_0^k\int_{1}^{1/z_0}\!\!\!\!\!\!\!\!\frac{z^k \, dz\!\!\!\!\!\!\!\!}{2\sqrt{z^2-1}} .
\eq
If $k$ is even we divide up the integral
\beq
z_0^k\int_{1}^{1/z_0}\!\!\!\!\!\!\!\!\frac{z^k \, dz\!\!\!\!\!\!\!\!}{2\sqrt{z^2-1}}
=z_0^k\int_{1}^{2}\!\!\frac{z^k \, dz\!\!\!}{2\sqrt{z^2-1}}
+z_0^k\int_{2}^{1/z_0}\!\!\!\!\!\!\!\!\frac{z^{k-1} \, dz\!\!\!\!\!\!\!\!}{2\sqrt{1-z^{-2}}}\,.
\eq
Here the first integral is independent of $z_0$ and the second integral can be evaluated by expanding the integrand in a power series in $z^{k-1-2\ell}$ for $\ell=0,1,\dots$. This gives a power series in $z_0^{2\ell}$  plus a term with $z_0^k \ln{z_0}$.
On the other hand if $k$ is odd we can integrate by parts a number of times
\beq
z_0^k\int_1^{1/z_0}\!\!\!\!(z^2
\shortminus 1)^{m-1/2} z^{k\shortminus 2m}\, dz=z_0^k\frac{(z^2\shortminus 1)^{m+1/2}}{ 2(m+1/2)} z^{k\shortminus 2m\shortminus 1}\Big|_{1}^{1/z_0}
-z_0^k\int_1^{1/z_0}\!\!\frac{(z^2\shortminus 1)^{m+1/2}\!\!\!\!}{ 2(m+1/2)}  (k\shortminus 2m\shortminus 1) z^{k\shortminus 2m\shortminus 2} dz .
\eq
Here the first term in the right has an even power series expansion of the form
$z_0^{2\ell}$, $\ell=0,1,\dots$, and the integral will vanish after integrating by parts
$m=(k-1)/2$ times. In either case we get a power series in $z_0^{2\ell}$, $\ell=0,1,\dots$
plus in the even case $z_0^k 2^{-1}\ln{z_0^2}$. Since $z_0^2=1-(t/r)^2 =\big(1-2(r-t)/r\big)(r-t)/r$ we conclude that  we have for either  an expansion of the form
\beq
r\Psi^{\text{ext}}_{1,0}[N]={\sum}_{j=0}^{k} N_{1j}^{\text{ext}}(\omega)\, y^j  \ln\Big(\frac{1}{y}\Big)+{\sum}_{j=0}^k N_{0j}^{\text{ext}}(\omega) y^j,\qquad y=(r-t)/r,
\eq
and
\beq
t\Psi^{\,\text{ext}}_{1,1}[M]={\sum}_{j=0}^k M_{0j}^{\text{ext}}(\omega) y^j,\qquad y=(r-t)/r\,.
\eq

\subsubsection{Leading orders of the exterior expansion for degree $-1$.}
\label{sec:ext:hom:expansion:leading}

In this section we compute the coefficients in the expansion of \eqref{eq:app:Psi:one} more precisely.
We have
\begin{multline}
\int_{z_0}^{1}\,\,\frac{ dz}{\!\!\!\!\!\sqrt{z^2\shortminus z_0^2}}
=\!\int_{1}^{1/z_0}\!\!\!\!\!\!\!\!\frac{ dz\!\!\!\!\!}{\!\!\!\sqrt{z^2\shortminus 1}}=\\
=\frac{1}{2}\ln{\Big(\frac{1}{z_0^2}\Big)}+\!\int_{1}^{\infty}\!\!\!\!\!\big(\frac{ 1}{\sqrt{z^2\shortminus 1}}-\frac{1}{z}\big)\, dz
-\!\int_{1/z_0}^{\infty}\!\!\!\big(\frac{ 1}{\sqrt{z^2\shortminus 1}}-\frac{1}{z}\big)\, dz
=\frac{1}{2}\ln{\Big(\frac{1}{z_0^2}\Big)}+\ln{2} -\frac{1}{4} z_0^2+ O(z_0^4),
\end{multline}
where we used that the primitive of $1/\sqrt{x^2-1}$ is equal to
$\ln{\big( \sqrt{x^2-1}+x\big)}$.
Since
\begin{equation}
 z_0^2=1-(t/r)^2 =\big(2-(r-t)/r\big)(r-t)/r
\end{equation}
it follows that
\beq
 \ln{\Big(\frac{1}{z_0^2}\Big)}=\ln{\frac{2r}{r\shortminus t}}-\ln{\Big(1-\frac{r\shortminus t}{2r}\Big)}-\ln{4}=\ln{\frac{2r}{r\shortminus t}}-\ln{4}+\frac{r\shortminus t}{2r}+O\Big(\frac{r\shortminus t}{2r}\Big)^2.
 \eq
 Hence
 \beq
\int_{z_0}^{1}\,\,\frac{ dz}{\!\!\!\!\!\sqrt{z^2\shortminus z_0^2}}
=\frac{1}{2}\ln{\frac{2r}{r\shortminus t}}-\frac{3}{2}\frac{r\shortminus t}{2r}+O\Big(\frac{r\shortminus t}{2r}\Big)^2.
\eq
 We also have
 \begin{equation}
 \int_{z_0}^1 \,\frac{  z dz}{\sqrt{z^2-z_0^2}}=\sqrt{1-z_0^2}=1-\frac{1}{2} z_0^2+O(z_0^4)=1-\frac{r\shortminus t}{r}+O\Big(\frac{r\shortminus t}{2r}\Big)^2.
 \end{equation}
 Moreover
 \begin{equation}
\int_{z_0}^1 \frac{ \,z^2 dz}{\!\!\!\sqrt{z^2\shortminus z_0^2}}=\int_{z_0}^1 \sqrt{z^2\shortminus z_0^2} \, dz+\int_{z_0}^{1}\,\,\frac{ z_0^2 dz}{\!\!\!\!\!\sqrt{z^2\shortminus z_0^2}},\quad\text{and}\quad
\int_{z_0}^1 \frac{ \,z^2 dz}{\!\!\!\sqrt{z^2\shortminus z_0^2}}
=\sqrt{z^2\shortminus z_0^2}\, z\Big|_{z_0}^1-\int_{z_0}^1{\!\!\!\sqrt{z^2\shortminus z_0^2}} dz ,
\end{equation}
and so
 \begin{equation}
\int_{z_0}^1 \frac{ \,z^2 dz}{\!\!\!\sqrt{z^2\shortminus z_0^2}}=\frac{1}{2} \sqrt{1-z_0^2} +\frac{1}{2}\int_{z_0}^{1}\,\,\frac{ z_0^2 dz}{\!\!\!\!\!\sqrt{z^2\shortminus z_0^2}}=\frac{1}{2}-\frac{r\shortminus t}{2r}+\frac{r\shortminus t}{2r}\ln{\frac{2r}{r\shortminus t}}+ O\Big(\frac{r\shortminus t}{2r}\Big)^2.
 \end{equation}

 We can now compute the first term in the expansion of $r\Psi_{1,0}^{\text{ext}}$ more precisely:
 \beq
\int_{z_0}^1 \,\frac{  N(\omega,z )\, dz}{\sqrt{z^2-z_0^2}}=N(\omega,0)\int_{z_0}^1 \frac{\ud z}{\sqrt{z^2-z_0^2}}+\int_{z_0}^1 \,\frac{ N_1(\omega,z)\, z dz}{\sqrt{z^2-z_0^2}}\,, \qquad\text{where}\quad
N_1(\omega,z)=\frac{ N(\omega,z ){}_{\!}-N(\omega,0)}{z},
\eq
and
\beq
\int_{z_0}^1 \,\frac{  N_1(\omega,z )\,z dz}{\sqrt{z^2-z_0^2}}
= N_1(\omega,0)\int_{z_0}^1 \,\frac{ \,z dz}{\sqrt{z^2-z_0^2}}
+\int_{z_0}^1 \,\frac{ N_2(\omega,z)\,z^2 dz}{\sqrt{z^2-z_0^2}}
\qquad\text{where}\quad
N_2(\omega,z)=\frac{ N_1(\omega,z ){}_{\!}-N_1(\omega,0)}{z}\,.
\eq
We see that the leading order terms are
 \beq
\Psi^{\,\text{ext}}_{1,0}[N]\sim \frac{1}{2r}\ln{\frac{2r}{r\shortminus t}}\,\mathcal{F}[N](\omega) +\frac{1}{r}\widetilde{N}(\omega),
 \eq
 where
 \beq
 \mathcal{F}[N](\omega) = \frac{1}{2\pi} \int_{\langle \sigma,\omega\rangle =0}\!\!\!\!\! \!\! \!  N(\sigma) \,ds(\sigma),\quad\text{and}\quad
 \widetilde{N}(\omega)=\int_{0}^1 \frac{ N(\omega,z ){}_{\!}-N(\omega,0)}{z}dz .
 \eq

\subsubsection{\!\!\!Formulas for the forward  solution with homogeneous initial data of degree $-2$.}
\label{sec:formula:ext:hom:lower}

Here we solve the initial value problem with lower order homogeneous data,
\beq\label{eq:app:hom:ext:two}
\Box \,\phi_2=0,\quad |x|>t,\qquad \phi\,\Big|_{t=0}=f=K(\omega)/r^2,\quad \pa_t\phi\, \Big|_{t=0}=g=L(\omega)/r^3\,.
\eq

\begin{lemma} The solution to \eqref{eq:app:hom:ext:two}, in the exterior $|x|>t$, is given by
\beq 
\phi_2=-\Psi^{\,\text{ext}}_{2,1}[K]-x^i t^{-1}\Psi^{\,\text{ext}}_{2,0}[K_i]+\Psi^{\text{ext}}_{2,0}[L]\,,\qquad\text{where}\quad K_{i}(\omega)r^{-3}=\pa_i \big( K(\omega)r^{-2}\big)\,,
\eq
and
\beq
\Psi^{\,\text{ext}}_{2,0}[L]=\frac{1}{2\pi( r^2-t^2)}\int_{\langle \omega,\sigma\rangle>z_0} \,\frac{  \langle \omega,\sigma\rangle L(\sigma )\, dS(\sigma)}{\sqrt{\langle \omega,\sigma\rangle^2-z_0^2}}\,,\quad \Psi_{2,1}^{\text{ext}}[K]=\frac{1}{t}\Psi_{1,0}^{\text{ext}}[K]\,, \quad z_0=\sqrt{1-(t/r)^2}\,.
\eq
\end{lemma}

\begin{proof}
First, by Kirchhoff's formula, in the case $f=0$, and $g(\rho\sigma)=L(\sigma) \rho^{-3}$ we get
\begin{equation}
  \phi_2=\frac{1}{4\pi }\int \,{\sum}_i\frac{ L(\sigma)\, \rho_i^{-1} \ud S(\sigma)}{\big|\rho_i\!-\!r\langle \omega,\sigma\rangle\big|},
\end{equation}
and after summing over the two points,
$$
\frac{1}{\rho_+}\!+\frac{1}{\rho_-}=\frac{\rho_+\! +\rho_-}{\rho_+ \rho_-}
=\frac{2 r\langle \omega,\sigma\rangle}{r^2-t^2}.
$$
this gives the solution $\phi_2=\Psi_{2,0}^{\text{ext}}[L]$.

Secondly, the solution corresponding to $g=0$, and $f(\rho\sigma)=K(\sigma)/\rho^2$, is given by
\begin{equation}
  \label{eq:24}
    \phi_2=\frac{1}{4\pi t^2} \int_{S(x,t)} \Big( f(y)+\nabla f(y)\cdot (y-x) \Big)\, \ud S(y)=\frac{1}{4\pi t^2}\int_{S(x,t)}\Bigl(-f(y)-\nabla f(y)\cdot x\Bigr)\ud S(y)\,.
\end{equation}
The first term is precisely $-t^{-1}\phi_f$, and already computed above with $g$ in \eqref{eq:phi:g} replaced by $f$.
Also the second term is already computed in the previous case, $L$ replaced by $K_i$ as defined below. Therefore
\begin{gather}
  \label{eq:26}
  \phi_2=-\Psi_{2,1}^{\text{ext}}[K]-\frac{x^i}{t}\Psi_{2,0}^{\text{ext}}[K_i]\,,\\
  \Psi_{2,1}^{\text{ext}}[K](t,r\omega)=\frac{1}{t}\Psi_{1,0}^{\text{ext}}[K](t,r\omega)=\frac{1}{2\pi rt}\int_{\langle \omega,\sigma\rangle>z_0} \,\frac{  K(\sigma )\, dS(\sigma)}{\sqrt{\langle \omega,\sigma\rangle^2-z_0^2}},  \quad(r>t)\,. \tag*{\qedhere}
\end{gather}
\end{proof}

\subsubsection{Leading orders in the exterior expansion for degree $-2$.}
\label{sec:formula:ext:hom:lower:precise}

We proceed as in Appendix~\ref{sec:ext:hom:expansion:leading}.
 Integrating by parts we obtain
 \beq
\int_{z_0}^1 \,\frac{  L(\omega,z ) \, z\, dz}{\sqrt{z^2-z_0^2}}-
\int_{0}^1 \!\!\!{  L(\omega,z )\, dz}=\big(\sqrt{1-z_0^2}-1\big)L(\omega,1 )
-\int_{z_0}^1 \big(\sqrt{z^2\shortminus z_0^2}-z\big)L_z(\omega,z)\, dz+\int_{0}^{z_0}\!\!\!\! z L_z(\omega,z)\, dz .
\eq
Here
 \begin{multline}
\int_{z_0}^1 \big(\sqrt{z^2\shortminus z_0^2}-z\big)L_z(\omega,z)\, dz
=z_0^2\int_{1}^{1/z_0} \!\!\!\big(\sqrt{t^2\shortminus 1}-t\big)L_z(\omega,z_0 t)\, dt\\
=-\frac{z_0^2}{2}\int_{1}^{1/z_0}\!\!\! t^{-1}L_z(\omega,z_0 t)\, dt
+z_0^2\int_{1}^{1/z_0} \!\!\!\big(\sqrt{t^2\shortminus 1}-t+t^{-1}/2\big)L_z(\omega,z_0 t)\,  dt .
\end{multline}
The first integral contributes a log $\ln{\big(1/z_0^2\big)}$
and the second converges as $z_0\to 0$. In fact
\begin{equation}
\int_{1}^{1/z_0} \!\!\!\big(\sqrt{t^2\shortminus 1}-t\!+\!\frac{1}{2t}\big)L_z(\omega,z_0 t)\, dt=
c_2 L_z(\omega,0)+O(z_0^2)
+\int_{1}^{1/z_0} \!\!\!\big(\sqrt{t^2\shortminus 1}-t\!+\!\frac{1}{2t}\big)
\big( L_z(\omega,z_0 t)-L_z(\omega,0)\big)\, dt,
\end{equation}
where
\begin{equation}
c_2=\int_{1}^{\infty} \!\!\!\big(\sqrt{t^2\shortminus 1}-t\!+\!\frac{1}{2t}\big)\, dt=-\frac{\ln{2}}{2}+\frac{1}{4},
\end{equation}
Here we used that the primitive
of $\sqrt{x^2\shortminus 1}$ is $\big(x\sqrt{x^2\shortminus 1}-\ln{\big|x+\sqrt{x^2\shortminus 1}\big|}\big)/2$.
Moreover
\begin{equation}
\int_{1}^{1/z_0} \!\!\!\big(\sqrt{t^2\shortminus 1}-t\!+\!\frac{1}{2t}\big)
\big( L_z(\omega,z_0 t)-L_z(\omega,0)\big)\, dt=O(z_0).
\end{equation}
We have
\begin{equation}
\int_{1}^{1/z_0} \!\!\! t^{-1}L_z(\omega,z_0 t)\, dt
=\frac{1}{2}\ln{(1/z_0^2)}L_z(\omega,0)+\int_{z_0}^{1} \!\!\! z^{-1} \big(L_z(\omega,z)-L_z(\omega,0)\big)\, dz,
\end{equation}
where
\begin{equation}
\int_{z_0}^{1} \!\!\! z^{-1} \big(L_z(\omega,z)-L_z(\omega,0)\big)\, dz
=\int_{0}^{1} \!\!\! z^{-1} \big(L_z(\omega,z)-L_z(\omega,0)\big)\, dz
-L_{zz}(\omega,0)z_0 +O(z_0^2).
\end{equation}

In conclusion, the leading orders in the expansion of $\Psi_{2,0}^{\text{ext}}$ are:
 \begin{multline}
   (r^2-t^2)\Psi_{2,0}^{\text{ext}}[L]=\int_{z_0}^1\frac{L(\omega,z)z\ud z}{\sqrt{z^2-z_0^2}}\sim \int_0^1 L(\omega,z)\ud z+\frac{z_0^2}{4}\ln(z_0^{-2})L_z(\omega,0)\\
   +\frac{z_0^2}{2} \biggl[-L(\omega,1)+\int_0^1\frac{L_z(\omega,z)-L_z(\omega,0)}{z}\ud z +(\ln 2+\frac{1}{2})L_z(\omega,0)\biggr]+\mathcal{O}(z_0^3),
 \end{multline}
 or
  \begin{equation}
   \Psi_{2,0}^{\text{ext}}[L]\sim \frac{1}{2r}\frac{1}{r-t}\int_0^1 L(\omega,z)\ud z+\frac{1}{4 r^2}\ln\frac{2r}{r-t}L_z(\omega,0)
   +\frac{1}{2r^2} \biggl[-L(\omega)+\widetilde{L_z}(\omega) +\frac{1}{2}L_z(\omega,0)\biggr]+\mathcal{O}\Bigl(\tfrac{(r-t)}{r^3}\ln\tfrac{2r}{r-t}\Bigr)\,.
 \end{equation}

\subsection{Inversion of the formula for the homogeneous solution in the exterior}
\label{sec:abel}
We can view $r\Psi_{1,0}^{\text{ext}}[N]$ as the result of the following transformation of the function $N(\omega)$:
\begin{equation}
  \widehat{N}(\omega,z_0)=\int_{z_0}^1 N(\omega,z)(z^2-z_0^2)^{-1/2}\ud z\,.
\end{equation}
Multiplying both sides by $z_0(z_0^2-w^2)^{-1/2}$ and integrating in $z_0$ on $[w,1]$, we obtain
\begin{equation}
  \label{eq:49}
  \int_w^1 \widehat{N}(\omega,z_0)z_0(z_0^2-w^2)^{-1/2}\ud z_0=\int_w^1N(\omega,z)\Bigl[\int_w^z \frac{z_0\ud z_0}{(z^2-z_0^2)^{1/2}(z_0^2-w^2)^{1/2}}\Bigr]\ud z\, ,
\end{equation}
where we interchanged the order of integration.
 If we make a change of variables $z_0^2 =z^2s$ and set $\alpha=(w/z)^2$ the inner integral becomes
\begin{equation}
\int_w^z \frac{z_0\ud z_0}{(z^2-z_0^2)^{1/2}(z_0^2-w^2)^{1/2}}=\frac{1}{2}\int_{\alpha}^1 \frac{\, \ud s}{(1-s)^{1/2}(s-\alpha)^{1/2}}=\frac{1}{2}\int_{\alpha}^1 \frac{\, \ud s}{\sqrt{(\frac{1-\alpha}{2})^2 -( s-\frac{1+\alpha}{2})^2}}\, .
\end{equation}
If we make a further change of variables
$ t=\frac{2}{1-\alpha}\big( s-\frac{1+\alpha}{2}\big)$,
we get
\begin{equation}
\int_w^z \frac{z_0\ud z_0}{(z^2-z_0^2)^{1/2}(z_0^2-w^2)^{1/2}}
=\frac{1}{2}\int_{-1}^1 \frac{\, \ud t}{\sqrt{1-t^2}}=\frac{\pi}{2}.
\end{equation}

 Therefore
\begin{gather}
  N(\omega,z)=-\frac{2}{\pi}\frac{\ud}{\ud w}\Bigl[  \int_w^1 \widehat{N}(\omega,z_0)z_0(z_0^2-w^2)^{-1/2}\ud z_0 \Bigr]\Big|_{w=z},\\
  N(\omega)=N(\omega,1)=-\frac{2}{\pi}\frac{\ud}{\ud w}\Bigl[  \int_w^1 \widehat{N}(\omega,z_0)z_0(z_0^2-w^2)^{-1/2}\ud z_0 \Bigr]\Big|_{w=1}\,.
\end{gather}
This shows that the initial data $N$ can be reconstructed from the homogeneous solution $\Psi_{1,0}^{\text{ext}}[N]$. Indeed, in view of \eqref{eq:phi:Psi10:ext} we have $r\Psi_{1,0}^{\text{ext}}[N](t,r\omega)\!=\!\widehat{N}(\omega,z_0(t,r))$.
This type of argument was used in \cite{H99}.

 \subsection{The Funk transform as a transformation on homogeneous functions \cite{BEGM03}.}
 \label{sec:transform:sphere}

 In our discussion of the asymptotics of homogeneous solutions in the exterior of the light cone we have encountered several transformations of smooth functions on the sphere.
 We will review their basic properties here as well as the relevant statements about their invertibility that are proven in \cite{BEGM03}.

Most importantly  we have the Funk transform which maps a function $M(\omega)$ on the sphere to
\begin{equation}
    \mathcal{F}[M](\omega)= \frac{1}{2\pi}\int_{C(\omega)} M(\sigma)\ud s(\sigma) ,
  \end{equation}
  where $C(\omega)=\{\sigma\in\mathbb{S}^2:\sigma\cdot\omega=0\}$ is the great circle obtained by intersecting the plane orthogonal to $\omega$ with $\mathbb{S}^2$. See \cite{G76} for a modern discussion of the geometric setting where it first appeared \cite{F13}.
  The Funk transform vanishes on odd functions, and $M\mapsto \mathcal{F}[M]$ maps even to even functions on the sphere.

  A related transformation is
\begin{equation}
  \mathcal{G}[N](\omega)= \frac{1}{2\pi}\int_{C(\omega)} \frac{\partial N}{\partial n}(\sigma)\ud s (\sigma)=\frac{1}{2\pi}\int_{C(\omega)} \nang N(\sigma)\cdot \omega\ \ud s(\sigma),
\end{equation}
where  $\partial N/\partial n$ denotes the normal derivative to the great circle $C(\omega)$; see Figure~\ref{fig:funk:normal}.
\begin{figure}
  \centering
  \includegraphics[scale=0.3]{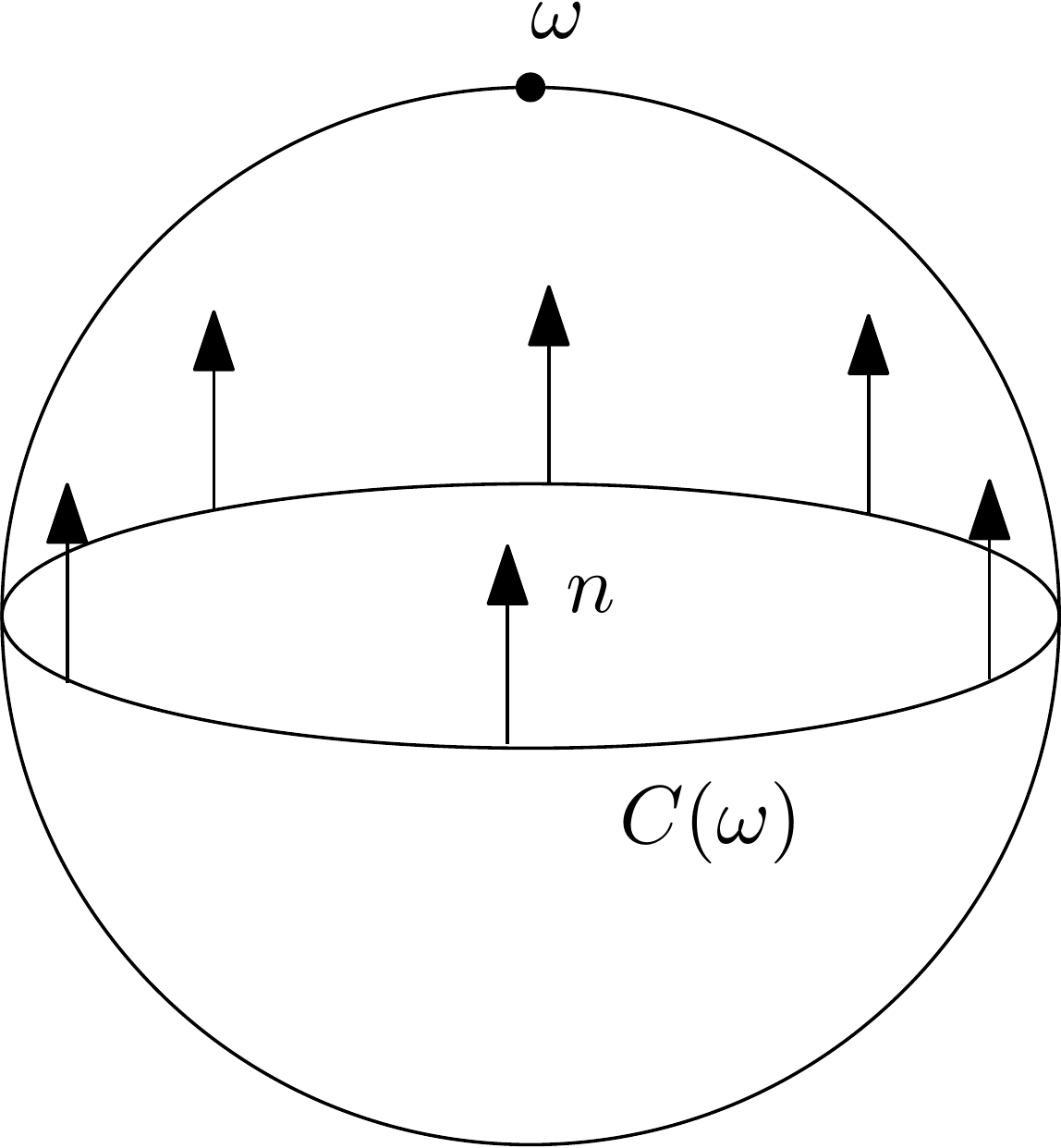}
  \caption{The Funk transform of the normal derivative.}
  \label{fig:funk:normal}
\end{figure}
As opposed to the Funk transform, $\mathcal{G}[M]$ vanishes if $M$ is even. In fact,  if we  choose coordinates so that $\omega_1=1$ and $\omega_2=\omega_3=0$, then
\begin{multline}
 \mathcal{F}[\nang_1 M]( 1,0,0 )
 =\frac{1}{2\pi} \int_{0}^{2\pi} \pa_1 M(0,\cos\theta,\sin\theta) \,d\theta\\
 =\lim_{\varepsilon\to 0} \frac{1}{4\pi\varepsilon}\int_{0}^{2\pi}\Big( M(\varepsilon,\sqrt{1-\varepsilon^2}\cos\theta,\sqrt{1-\varepsilon^2}\sin\theta)
 -M(-\varepsilon,\sqrt{1-\varepsilon^2}\cos\theta,\sqrt{1-\varepsilon^2}\sin\theta)\Big)\,d\theta=0,
\end{multline}
because when $M$ is even,
\begin{multline}
M(-\varepsilon,\sqrt{1-\varepsilon^2}\cos\theta,\sqrt{1-\varepsilon^2}\sin\theta)
=M(\varepsilon,-\sqrt{1-\varepsilon^2}\cos\theta,-\sqrt{1-\varepsilon^2}\sin\theta)\\
=M\big(\varepsilon,\sqrt{1-\varepsilon^2}\cos(\theta+\pi),\sqrt{1-\varepsilon^2}\sin(\theta+\pi)\big).
\end{multline}
Moreover the transformation $M\mapsto\mathcal{G}[M]$ maps odd functions to odd functions on the sphere.
Geometrically it can also be viewed as a Funk transform of the 1-form ${}^\ast\ud M$, see e.g.~\cite{G76,M78}.

In \cite{BEGM03} the Funk transform is viewed as a map that sends even functions on $\mathbb{R}^3\setminus \{0\}$, which are homogeneous of degree $-2$, to even functions which are homogeneous of degree $-1$.
Similarly $\mathcal{G}$ is viewed as a transformation which maps odd functions  on $\mathbb{R}^3\setminus \{0\}$, which are homogeneous of degree $-1$, to odd functions which are homogeneous of degree $-2$.
In general, it is of course always possible to associate to a function on $\mathbb{S}^2$ its  extension to $\mathbb{R}^3\setminus \{0\}$ as a homogeneous function of a given degree.
However, the specific degrees of homogeneity of the transformations studied in \cite{BEGM03} have geometric significance, and are all examples of transformations which map functions on $\mathbb{R}^3\setminus\{0\}$ which are homogeneous of degree $k$ to homogeneous functions of degree $-k-3$.

In \cite{BEGM03} it is proven that $\mathcal{F}$ restricted to even functions, and $\mathcal{G}$ restricted to odd functions, are invertible transformations onto their image.
In fact, the inverse transformations, $\mathcal{T}$ and $\mathcal{S}$,
\begin{equation}
  \mathcal{T}\circ\mathcal{F}=\mathrm{id}\,,\qquad   \mathcal{S}\circ\mathcal{G}=\mathrm{id}\,,
\end{equation}
are given in explicit form.
$\mathcal{T}$ maps even functions on $\mathbb{R}^3\setminus\{0\}$, which are homogeneous of degree $-1$, to even functions which are homogeneous of degree $-2$;
in rectangular coordinates so that $\omega=(1,0,0)$ it is given by
\begin{equation}
  \mathcal{T}[\phi](1,0,0)=\frac{1}{4\pi}\int_{\mathbb{S}^2}\log\Bigl\lvert \frac{x}{y} \Bigr\rvert \frac{\partial^2 \phi}{\partial x^2}\,,\qquad \phi \text { homogeneous degree }-1\,,
\end{equation}
Similarly, $\mathcal{S}$ maps odd to odd functionn on $\mathbb{R}^3\setminus\{0\}$, and is given by
\begin{equation}
      \mathcal{S}[\psi](1,0,0))=-\frac{1}{4\pi}\int_{\mathbb{S}^2}\log\Bigl\lvert \frac{x}{y} \Bigr\rvert \frac{\partial \psi}{\partial x}\,,\qquad \psi \text { homogeneous degree }-2\,.
\end{equation}

For our purposes we need  to write the transformation $\mathcal{S}$ in a form intrinisic to the sphere $\mathbb{S}^2$.
It is shown on page 590 of \cite{BEGM03} that for functions $\psi$  homogeneous of degree $-2$,
 \begin{equation}
     -\frac{1}{4\pi}\int_{\mathbb{S}^2}\log\Bigl\lvert \frac{x}{y} \Bigr\rvert \frac{\partial \psi}{\partial x} = -\frac{1}{8\pi}  \lim_{\epsilon\to 0} \int_{\mathbb{S}^2}\log\Bigl\lvert \frac{x^2+\epsilon^2 y^2}{y^2} \Bigr\rvert \frac{\partial \psi}{\partial x}
     =\frac{1}{8\pi}  \lim_{\epsilon\to 0} \int_{\mathbb{S}^2}  \frac{2x \psi}{x^2+\epsilon^2 y^2}\,.
 \end{equation}
 This formula is related to that in 1 dimension $\frac{\ud}{\ud  x}\log |x|\!=\!\mathrm{Pv}(\frac{1}{x})$, in the sense of distributions, i.e.
 \begin{equation}
     -\int_{\mathbb{R}}\log|x|\frac{\partial \psi}{\partial x} \ud x=\frac{\ud}{\ud x}\log|x|\ [\psi]=\text{pv}(1/x)[\psi]=\lim_{\epsilon\to 0}\int\frac{x \psi(x)}{x^2+\epsilon^2}
     =\lim_{\epsilon\to 0}\int_{|x|\geq \epsilon}\frac{\psi(x)}{x}\ud x=\int\frac{\psi(x)-\psi(0)}{x}\ud x\,.
 \end{equation}
Here we obtain similarly that
 \begin{equation}
   \frac{1}{4\pi}  \lim_{\epsilon\to 0} \int_{\mathbb{S}^2}  \frac{x \psi}{x^2+\epsilon^2 y^2}=\lim_{\epsilon\to 0}\frac{1}{4\pi} \int_{\substack{\mathbb{S}^2,\,\,\,|x|\geq \epsilon}}  \frac{ \psi}{x}\,.
 \end{equation}
If $\psi(x)$ denotes the \emph{average} over the circle obtained by intersecting the plane of constant $x$ with $\mathbb{S}^2$, we get
 \begin{equation}
   \lim_{\epsilon\to 0} \frac{1}{4\pi} \int_{\substack{\mathbb{S}^2,\,\,|x|\geq \epsilon}}  \frac{ \psi}{x} = \lim_{\epsilon\to 0} \frac{1}{2} \int_{|x|\geq \epsilon}\frac{\psi(x)}{x} \ud x=\frac{1}{2}\int_{-1}^1 \frac{\psi(x)-\psi(0)}{x} \ud x\,.
 \end{equation}

This shows that  the quantity $\widetilde{N}$ in the Section~\ref{sec:hom:ext} is in fact $\mathcal{S}[N]$, when $N$ is odd.
Therefore 
\begin{equation}
  \mathcal{S}:  M(\omega)\mapsto \mathcal{S}[M](\omega)=\frac{1}{2}\int_{-1}^1 \frac{M(\omega,z)-M(\omega,0)}{z}\ud z
\end{equation}
is an inverse of the transformation $\mathcal{G}$, so $\mathcal{G}\circ\mathcal{S}=\mathcal{S}\circ\mathcal{G}=\mathrm{id}$.

\subsection{Geometric interpretation and the inverse of the Funk transform \cite{H99}.}
\label{sec:funk:geometric}
Given a function $N(\sigma)$ on the sphere $\mathbb{S}^2$  the Funk transform $\mathcal{F}[N]$ maps any great circle $C$ of $\mathbb{S}^2$ to the integral
\begin{equation}
  \mathcal{F}[N](C)=\int_C N\, ,
\end{equation}
 of $N(\sigma)$ over $C$.
Here the integral is taken with respect to the induced line element on the sphere. This is the analogue of the Radon transform $\mathcal{R}$, which  takes functions $f$ on $\mathbb{R}^3$ to functions $\mathcal{R}[f]$
on the space of planes $P$ in $\mathbb{R}^ 3$, given by  $\mathcal{R}[f](P)=\int_P f$.
The Funk transform vanishes on odd functions, but is invertible on even functions, to which we  henceforth restrict ourselves in this section.

In this section we review the the reconstruction formula for even functions proven in \cite{H99}, which is given in terms of a dual transformation, and its generalisations that we will discuss next.

The dual transformation $\mathcal{F}^\ast$ takes functions $\nu$ on the space of  great circles to functions on the sphere, by taking the average over all great circles passing through a given point $\omega\in\mathbb{S}^2$:
\begin{equation}
  \mathcal{F}^\ast[\nu](\omega)=\int_{\omega\in C}\nu(C) \ud\mu (C)\,.
\end{equation}
A generalisation of the Funk transform of a function $N$
appears in \cite{H99}. For $p>0$,  consider the transformation
\begin{equation}
  \mathcal{F}_p[N](C)=\int_{d(\sigma,C)=p}N(\sigma) \ud s(\sigma) ,
\end{equation}
which is an integral that extends over two circles that are at equal distance $p$ from the great circle $C$; the set $\{\sigma:d(\sigma,C)=p\}$ consists of two disjoint circles of radius $\cos(p)<1$.
For even functions $N$, these are related to  the averages over circles  introduced in \eqref{eq:avg} as follows:
\begin{equation}
  \mathcal{F}_p[N](C_\omega)=2\int_{\langle \omega,\sigma\rangle=\sin(p)}\!\!\!\!\!\!\!\!\!\!\!N(\sigma) \ud s(\sigma)=4\pi \cos q \:N(\omega,z)\,,\quad z=\sin(q)\,,\quad \ C_\omega=\{\sigma\in\mathbb{S}^2:\langle \sigma,\omega\rangle=0\}\,.
\end{equation}
The associated dual transformation is
\begin{equation}
  \mathcal{F}^\ast_p[\nu](\omega)=\int_{d(\omega,C)=p}\nu(C)\ud\mu(C)\,,
\end{equation}
which is the  average taken over all great circles $C$ with distance $p$ from $\omega\in\mathbb{S}^2$.

\begin{prop}[\cite{H99}]
  Let $N(\sigma)$ be an even function on $\mathbb{S}^2$, and let $\mathcal{F}[N]$ be the Funk transform of $N$. Then
  \begin{equation}
    N(\sigma)=\frac{1}{\pi}\frac{\ud}{\ud t}\Bigl[\int_0^t \mathcal{F}^\ast_{\arccos(v)}[\mathcal{F}[N]](\sigma)\frac{v \ud v}{\sqrt{t^2-v^2}}\Bigr]\rvert_{t=1} \,.
  \end{equation}
\end{prop}

For the proof,   consider $\mathcal{F}_p^\ast[\mathcal{F}[N]]$: This is the average of $N$
over $\mathbb{S}^2\setminus K$, where $K=K_+\cup K_-$ is the union of two caps,
where $K_\pm=\{\sigma:d(\sigma,\pm\omega)<p\}$.
This average can either be expressed as an integral over great circles at distance $p$ from $\omega$, or alternatively, as an integral over circles a fixed distance from $\omega$, cf.~Fig.~\ref{fig:funk}.
\begin{figure}
  \centering
  \includegraphics[scale=1.2]{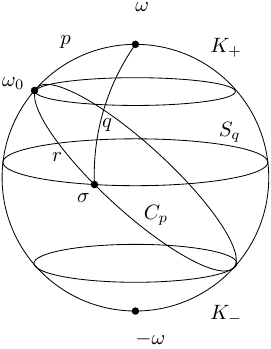}
  \caption{The generalised Funk transform.}
  \label{fig:funk}
\end{figure}
For this purpose, let  $(M^q N)(\omega)$ denote the average of $N$ over the circle $S_q=\{\sigma:d(\sigma,\omega)=q\}$ of distance $q$ from $\omega$.
Now fix any great circle $C_p$ at distance $p$ from $\omega$, then we have
\begin{equation}
  \mathcal{F}^\ast_p[\mathcal{F}[N]](\omega)=\int_{C_p} (M^{d(\omega,\sigma)} N)(\omega)\ud s(\sigma)\,.
\end{equation}
In other words, instead of averaging over all great circles of distance $p$ from $\omega$, we average over the circles $S_q$, parametrized by the points on one fixed great circle $C_p$.
Let $\omega_0$ be the point on $C_p$ closest to $\omega$, so  $d(\omega,\omega_0)=p$, and let $r=d(\omega_0,\sigma)$, and $q=d(\omega,\sigma)$ for $\sigma\in C_p$.
Note that by the spherical pythagoras theorem, 
$\cos(q)\!=\!\cos(p)\cos(r)$, because $C_p$ makes a right angle with the great circle through $\omega$ at $\omega_0$,
and for fixed $p$, we can view $q$ as a function of $r$.
Then for even functions $N$, $(M^qN)(\omega)$ is symmetric around $r\!=\!\pi/2$, and we have
\begin{equation}
  \label{eq:funk:p}
  \mathcal{F}^\ast_p[\mathcal{F}[N]](\omega)=2\int_0^{\pi/2} (M^q N)(\omega)\ud r\,.
\end{equation}
For fixed $\omega$, let us now make the following change of variables: set $v=\cos(p)>0$ and $u=v\cos(r)=\cos(q)>0$, then substituting for $r$ in \eqref{eq:funk:p}, $\ud u=-(v^2-u^2)^{1/2}\ud r$ and denoting by $L(u)=(M^qN)(\omega)$, and $\widehat{L}(v)=\mathcal{F}^\ast_p[\mathcal{F}[N]](\omega)$,
the formula \eqref{eq:funk:p} reads
\begin{equation}
  \widehat{L}(v)=2\int_0^vL(u)(v^2-u^2)^{-1/2}\ud u\,.
\end{equation}
Like the  transformation in the Appendix~\ref{sec:abel}, this transformation is invertible. Since $L(1)\!=\!(M^0N)(\omega)\!=\!N(\omega)$, this gives the reconstruction formula for $N$ in terms of $\widehat{L}$, and hence in terms of the Funk transfrom $\mathcal{F}[N]$.

\end{document}